\documentclass[11pt,twoside]{amsart}
\usepackage{latexsym,amssymb,amsmath}
\usepackage[all]{xy}
\usepackage{tikz,pgfplots}
\usepackage{extpfeil}
\usepackage{afterpage}
\usepackage{float}
\usepackage{hyperref}
\numberwithin{equation}{section}
\newtheorem{theorem}{Theorem}[section]
\newtheorem{lemma}[theorem]{Lemma}
\newtheorem{proposition}[theorem]{Proposition}
\newtheorem{corollary}[theorem]{Corollary}

\theoremstyle{definition}
\newtheorem{definition}[theorem]{Definition} 
\newtheorem{procedure}[theorem]{Procedure} 
\newtheorem{remark}[theorem]{Remark}
\newtheorem{example}[theorem]{Example}

\begin{document}

\title[projective
Reed--Muller-type codes]
{Indicator functions, v-numbers and Gorenstein rings in the theory of projective
Reed--Muller-type codes}  

\author[Gonz\'alez-Sarabia et al.]{Manuel Gonz\'alez-Sarabia} 

\address{Instituto Polit\'ecnico Nacional, 
UPIITA, Av. IPN No. 2580,
Col. La Laguna Ticom\'an,
Gustavo A. Madero C.P. 07340,
Ciudad de M\'exico, M\'exico, Departamento de Ciencias B\'asicas.}
\email{mgonzalezsa@ipn.mx}

\thanks{
The first and fifth authors were partially supported by SNII, M\'exico. 
The second and third authors were supported by a scholarship from
CONAHCYT, M\'exico. The fourth author is supported by
grant PID2020- 116641GB-I00 funded by MCIN/AEI/10.13039/501100011033}

%\author[Mu\~noz-George]{Humberto Mu\~noz-George}
\author[]{Humberto Mu\~noz-George}
\address{
Departamento de
Matem\'aticas\\
Centro de Investigaci\'on y de Estudios
Avanzados del
IPN\\
Apartado Postal
14--740 \\
07000 Ciudad de M\'exico, M\'exico.
}
\email{munozgeorge426@gmail.com}

\author[]{Jorge A. Ordaz}
%\author[J. Ordaz]{Jorge A. Ordaz}
\address{
Departamento de
Matem\'aticas\\
Centro de Investigaci\'on y de Estudios
Avanzados del
IPN\\
Apartado Postal
14--740 \\
07000 Ciudad de M\'exico, M\'exico.
}
\email{JorgeOrdaz00@hotmail.com}

%\author[E. S\'aenz-de-Cabez\'on]{Eduardo S\'aenz-de-Cabez\'on}
\author[]{Eduardo S\'aenz-de-Cabez\'on}
\address{Departamento de Matem\'aticas y Computaci\'on, Universidad
de La Rioja, Logro\~no, Spain.}
\email{eduardo.saenz-de-cabezon@unirioja.es}

%\author[R. H. Villarreal]{Rafael H. Villarreal}
\author[]{Rafael H. Villarreal}
\address{
Departamento de
Matem\'aticas\\
Centro de Investigaci\'on y de Estudios
Avanzados del
IPN\\
Apartado Postal
14--740 \\
07000 Ciudad de M\'exico, M\'exico.
}
\email{vila@math.cinvestav.mx}
%\thanks{*Corresponding author}

\begin{abstract} 
For projective Reed--Muller-type 
codes we give a global duality criterion in terms of the v-number and the Hilbert function of a
vanishing ideal. As an application, we provide a global duality
theorem for projective Reed--Muller-type 
codes over Gorenstein vanishing ideals, generalizing the
known case where the vanishing ideal is a complete intersection. 
We classify self dual Reed--Muller-type codes
over Gorenstein ideals using the regularity and a parity check
matrix. 
For projective evaluation codes, we give a 
duality theorem inspired by that of affine evaluation codes. 
We show how to compute the regularity index of the $r$-th generalized Hamming weight 
function in terms of the standard indicator functions of the set of
evaluation points. 
\end{abstract}

\maketitle 

\section{Introduction}\label{section-introduction}
Let $S=K[t_1,\ldots,t_s]=\bigoplus_{d=0}^\infty S_d$ be a polynomial ring over a finite
field $K=\mathbb{F}_q$ with the standard grading, let 
$\prec$ be a graded monomial order on $S$,  and let
$\mathbb{X}=\{[P_1],\ldots,[P_m]\}$, $|\mathbb{X}|=m\geq 2$, be a set of
distinct points in the projective space $\mathbb{P}^{s-1}$ over the
field $K$.
Let $I=I(\mathbb{X})$ be the graded vanishing ideal of
$\mathbb{X}$ and let $\mathfrak{p}_i$ be the vanishing 
ideal $I([P_i])$ of $[P_i]$.  

The v-number of $I$, denoted ${\rm
v}(I)$, is an algebraic invariant of
$I$ that was introduced in \cite{min-dis-generalized} to study the asymptotic behavior 
of the minimum distance function of $I$ and the minimum distance of 
projective Reed--Muller-type codes (Eq.~\eqref{jun17-23}). 
One can define the v-number ${\rm v}_{\mathfrak{p}_i}(I)$ of $I$
locally at each $\mathfrak{p}_i$ (Eq.~\eqref{jun17-23-1}). 
The notion of v-number is related to indicator
functions (Definitions~\ref{indicator-f} and \ref{sif}). 
These functions are used in coding
theory \cite{min-dis-generalized,dual,sorensen}, Cayley--Bacharach 
schemes \cite{geramita-cayley-bacharach,Guardo-Marino-Van-Tuyl,tohaneanu-vantuyl}, 
and interpolation problems \cite{cocoa-book}. An indicator function of
$[P_i]$ can be computed using \cite[Corollary~6.3.11]{cocoa-book}, see also
\cite{Ceria-etal} and references therein.     

The contents of this work are as follows. 
In Section~\ref{section-prelim}, we introduce definitions and well
known results from commutative algebra and coding theory. We refer 
to this section for all unexplained
terminology and additional information.

In Lemma~\ref{if} and Proposition~\ref{indicator-function-prop}, we
show the following properties of indicator functions of
projective points. For a  thorough study of indicator functions of
affine points see \cite{dual}. 

{\rm(i)} The least degree of an indicator function of $[P_i]$ is equal 
to ${\rm v}_{\mathfrak{p}_i}(I)$.

{\rm(ii)} ${\rm v}_{\mathfrak{p}_i}(I)\leq r_0$ for all $i$ and ${\rm
v}_{\mathfrak{p}_i}(I)=r_0$ for some $i$, where $r_0$ is the
regularity of $S/I$.

{\rm(iii)} For each $[P_i]$ there exists a unique, up to multiplication by a scalar from
$K^*$, standard indicator function $f_i$ of $[P_i]$ of degree ${\rm
v}_{\mathfrak{p}_i}(I)$, where $K^*:=K\setminus\{0\}$.

{\rm (iv)} $(I\colon\mathfrak{p}_i)/I$ is a principal ideal of $S/I$
generated by $\overline{f_i}=f_i+I$ for $i=1,\ldots,m$.

Note that unlike \cite{dual}, we do not
require $f_i(P_i)=1$. 
Based on (iv), 
we obtain an algebraic method
to compute the set $F:=\{f_1,\ldots,f_m\}$ of standard indicator
functions of $\mathbb{X}$, and give an
implementation in \textit{Macaulay}$2$ \cite{mac2}
(Example~\ref{gorenstein-2essential-example},
Procedure~\ref{sep12-18}).

Fix a degree $d\geq 1$. Let $C_\mathbb{X}(d)$ be the 
projective Reed--Muller-type code of
degree $d$ on the set $\mathbb{X}$, let
$\mathcal{L}$ be a linear subspace of $S_d$, and let 
$\mathcal{L}_\mathbb{X}$ be the projective evaluation code of degree
$d$ on the set $\mathbb{X}$ (Subsection~\ref{RM}).  

In this work we emphasize the use of indicator functions, 
{\rm v}-numbers and Gorenstein rings in the theory of projective
Reed--Muller-type codes 
and projective evaluation codes. The affine case is treated in detail
in the article of L\'opez, Soprunov and Villarreal \cite{dual}. 

Given projective evaluation codes $C_1,C_2$ of degrees $d,k$ on
$\mathbb{X}$, we are interested in determining \textit{duality criteria}, that 
is, finding necessary
and sufficient conditions for the monomial equivalence 
(Definition~\ref{me}) of the dual
$C_1^\perp$ of $C_1$ and
$C_2$, using the algebraic invariants of $I$ and the standard indicator functions of
$\mathbb{X}$.
For Reed--Muller-type codes, a duality
criterion is \textit{global} if it classifies the monomial equivalence of
$C_\mathbb{X}(d)^\perp$ and $C_\mathbb{X}(k)$ within a certain
range for $d$ and $k$. 

There is a global duality criterion for affine Reed--Muller-type
codes \cite[Theorem~6.5]{dual} and a global duality theorem for
projective Reed--Muller type codes over complete intersections
\cite[Theorem~2]{sarabia7}. It is well known that an affine
Reed--Muller-type code is a projective Reed--Muller-type 
code \cite{cartesian-codes,affine-codes}. We formulate and prove an extension of
these duality results for arbitrary sets of
projective points, using the Hilbert function $H_I$ of $I$ and the 
algebraic invariants of $I$.

We give a projective global duality criterion that classifies
when $C_\mathbb{X}(r_0-d-1)$ is monomially equivalent to
$C_\mathbb{X}(d)^\perp$ for all $0\leq d\leq r_0$, where 
$r_0$ is ${\rm reg}(S/I)$, the
regularity of $S/I$. Since 
$\dim_K(C_\mathbb{X}(d)^\perp)$ is equal to 
$|\mathbb{X}|-H_I(d)$, a necessary condition for this equivalence
is the equality 
$$H_I(d)+H_I(r_0-d-1)=|\mathbb{X}|\ \mbox{ for all }\ 0\leq d\leq r_0.$$
\quad This equality is equivalent to the symmetry of the $\mathrm{h}$-vector
of $S/I$ (Proposition~\ref{jul5-23}). Another necessary condition comes from the fact 
that the monomial equivalence of $C_\mathbb{X}(0)^\perp$ and
$C_\mathbb{X}(r_0-1)$ implies that ${\rm v}_{\mathfrak{p}_i}(I)=r_0$
for $i=1,\ldots,m$ (Lemma~\ref{sarabia-vila-1}). 
These two 
conditions give us an effective projective global duality 
criterion, that we implement in \textit{Macaulay}$2$ \cite{mac2}
(Procedure~\ref{sep12-18}), to determine when  
$C_\mathbb{X}(r_0-d-1)$ is monomially equivalent to
$C_\mathbb{X}(d)^\perp$ for all $0\leq d\leq r_0$.

We come to one of our main results. 

\noindent {\bf Theorem~\ref{duality-criterion}.}{\rm\ (Projective duality
criterion)}\textit{\ The
following conditions are equivalent. 
\begin{enumerate} 
\item[(a)] $C_\mathbb{X}(r_0-d-1)$ is monomially equivalent to
$C_\mathbb{X}(d)^\perp$  for all $0\leq d\leq r_0$.  
\item[(b)] $H_I(d)+H_I(r_0-d-1)=|\mathbb{X}|$ for all $0\leq d\leq r_0$ 
and $r_0={\rm v}_{\mathfrak{p}_i}(I)$ for all $\mathfrak{p}_i$.
\item[(c)] There is a vector $\beta=(\beta_1,\ldots,\beta_m)\in K^m$ such
that $\beta_i\neq 0$ for all $i$ and
\begin{equation*}
C_\mathbb{X}(d)^\perp=(\beta_1,\ldots,\beta_m)\cdot
C_\mathbb{X}(r_0-d-1)\ \text{ for all \ $0\leq d\leq r_0$}.
\end{equation*}
Moreover, $\beta$ is any vector that defines a parity  
check matrix of $C_\mathbb{X}(r_0-1)$. 
\end{enumerate}
}

The vector $\beta$ of Theorem~\ref{duality-criterion}(c) is unique up to 
multiplication by a scalar from $K^*$ (Lemma~\ref{sep5-23}).  
This criterion will be used to show duality criteria for some
interesting families and recover some known results. 
The condition ``$r_0={\rm v}_{\mathfrak{p}_i}(I)$ for
all $\mathfrak{p}_i$'' that appears in the projective duality criterion
defines a Cayley--Bacharach scheme (CB-scheme) when $K$ is an infinite field 
\cite[Definition~2.7]{geramita-cayley-bacharach}, and is related to
Hilbert functions \cite{geramita-cayley-bacharach}.

If $I$ is a Gorenstein ideal (Subsection~\ref{CA}), then the two
algebraic conditions of Theorem~\ref{duality-criterion}(b) are satisfied
(Proposition~\ref{jul6-23}), and projective global duality
holds in this case (Corollary~\ref{duality-criterion-g}). 
Since complete intersection ideals are Gorenstein
(Proposition~\ref{ci-gor}), 
we recover the projective global duality theorem for complete intersections 
that was shown in 2004 by Gonz\'alez-Sarabia and Renter\'\i a
\cite[Theorem~2]{sarabia7}
(Corollary~\ref{duality-criterion-original}).   

Let $\Delta_\prec(I)$ be the set of standard monomials of $S/I$ with
respect to $\prec$. The notion of essential monomial was introduced 
in \cite{dual} (Definition~\ref{essential-m}). 
The standard indicator functions of $\mathbb{X}$ and the essential monomials 
depend on the monomial order $\prec$ we choose, and essential monomials 
do not always exist (Example~\ref{gorenstein-essential-example}).

We give the following version for projective evaluation codes of the local duality 
theorem of L\'opez, Soprunov and Villarreal \cite[Theorem~5.4]{dual} for affine 
evaluation codes. 

\noindent {\bf Theorem~\ref{combinatorial-condition}.}
\textit{If there exists $t^e$ essential monomial,
and if $\beta_i$ is the coefficient of $t^e$ in the $i$-th standard indicator
function~$f_i$ of $\mathbb{X}$, then for any
$\Gamma_1\subset\Delta_\prec(I)_d$, 
$\Gamma_2\subset \Delta_\prec(I)_k$
satisfying  
\begin{enumerate}
\item $d+k=r_0$,
\item $|\Gamma_1|+|\Gamma_2|=|\mathbb{X}|$, and 
\item $t^e$ does not appear in the remainder on division of $u_1u_2$ by $I$
for every $u_1\in\Gamma_1$, $u_2\in\Gamma_2$,
\end{enumerate}
we have 
$\gamma\cdot{\rm ev}_d(K\Gamma_1)={\rm
ev}_k(K\Gamma_2)^\perp$,
where $\gamma=(\beta_1,\dots,
\beta_m)\cdot(f_1(P_1)^{-1},\ldots,f_m(P_m)^{-1})$, $K\Gamma_i$ is the
linear space generated by $\Gamma_i$, and ${\rm ev}_d$
is the evaluation map of degree $d$ in Eq.~\eqref{evaluation-map}.
}

We classify self dual projective Reed--Muller-type codes
over Gorenstein ideals, up to
monomial equivalence, in terms of the regularity 
(Proposition~\ref{self-dual}).

The following theorem classifies self dual projective Reed--Muller-type codes
over Gorenstein ideals in terms of the regularity and a parity check
matrix.

\noindent {\bf Theorem~\ref{self-dual-general}.}\textit{
Let $\mathbb{X}$ be a subset of
$\mathbb{P}^{s-1}$, let 
$I$ be its vanishing ideal, and let $r_0$ be the regularity of $S/I$.
If $I$ is Gorenstein, then $C_\mathbb{X}(d)^\perp=C_\mathbb{X}(d)$ 
for some $1\leq d\leq r_0$ if and only if
$r_0=2d+1$ and the vector $(1,\ldots,1)$ defines a parity check matrix of
$C_\mathbb{X}(2d)$. 
}

We give a simple algorithm implemented in \textit{Macaulay}$2$
\cite{mac2}  
to determine whether or not a given $C_\mathbb{X}(d)$ is 
self orthogonal or self dual (Example~\ref{algorithm-example}, 
Procedure~\ref{sep29-23}). The self dual codes of the form
$C_\mathbb{X}(d)$, $\mathbb{X}=\mathbb{P}^{s-1}$, were classified 
by S{\o}rensen \cite[Theorem~2]{sorensen}.   

The property of
being a Gorenstein ideal is purely algebraic, and because of this there are
few tools that make this property more intuitive, see
\cite{Davis-etal,Eisen,geramita-cayley-bacharach,Kreuzer,Mats,Sta1} and the
introduction of the recent article \cite{Elias-Rossi}.

The content of Section~\ref{section-gorenstein} is as follows. 
The theory of Gorenstein vanishing ideals is examined using 
indicator functions, and we use this
theory to show existence of essential monomials and to give a duality
theorem for projective evaluation codes over Gorenstein 
ideals (Corollary~\ref{combinatorial-condition-projective}). Using a
result of Kreuzer \cite[Corollary~2.5(b)]{Kreuzer}, we prove
that $I(\mathbb{X})$ is Gorenstein if and only if the projective
duality criterion of Theorem~\ref{duality-criterion} holds
(Theorem~\ref{conjecture-gor}). The results in 
this section can be used to study the family of parameterized projective codes over
algebraic toric sets introduced and studied by Renter\'\i a, Simis and
Villarreal \cite{algcodes}
(Example~\ref{example-parameterized}). For this family Maria Vaz
Pinto (personal communication) has conjectured that $I(\mathbb{X})$ is Gorenstein if and only
if $I(\mathbb{X})$ is a complete intersection. 
     
We set $J=(I,h)$, where $h$ is a form of degree $1$ regular on $S/I$.
The ring $S/J$ is Artinian and its socle is denoted by 
${\rm Soc}(S/J)$ (Eq.~\eqref{socle-def}). 
The notion of \textit{level} algebras---algebras whose socle is in one
degree---was introduced by Stanley \cite{Stanley-level}, see 
\cite{boij-level} for the related notion of level module. Gorenstein
algebras are level \cite{Fro3}, \cite[Lemma~5.3.3]{monalg-rev}, and many of the graded 
algebras that we find in algebraic geometry and in combinatorics are level
algebras.  

If $I$ is Gorenstein, we describe the socle of $S/J$ using
$\Delta_\prec(J)$ (Theorem~\ref{socle-gorenstein}(a)), and if $S/I$ is a level ring and has
symmetric h-vector, we show that $S/I$ is Gorenstein 
(Theorem~\ref{socle-gorenstein}(b)). As a consequence, 
if $f_1,\ldots,f_m$ are the standard indicator functions
of $\mathbb{X}$ and $I$ is Gorenstein, then
$\Delta_\prec(J)_{r_0}=\{t^a\}$ and for each $i$, $\deg(f_i)=r_0$, 
${\rm Soc}(S/J)=K(f_i+J)=K(t^a+J)$, 
and the remainder on division of $f_i$ by $J$ has the form
$r_{f_i}=\lambda_i t^a$ for some $\lambda_i\in K^*$ (Corollary~\ref{socle-gorenstein-coro}).

Let $\prec$ be the graded reverse lexicographic order on $S$. This is
the default order in
\textit{Macaulay}$2$ \cite{mac2}, 
in large part because it is often the most efficient order for use
with Gr\"obner bases.  
If $t_s$ is regular on $S/I$ and $I$ is Gorenstein, there exists 
an essential monomial $t^a\in\Delta_\prec(I)_{r_0}$ such
that $t_s\notin{\rm supp}(t^a)$, and every $f_i$ has the form
$f_i=\lambda_it^a+t_sG_i$, for some $\lambda_i\in K^*$, $G_i\in S_{r_0-1}$ 
(Proposition~\ref{gorenstein-essential}). In particular, essential
monomials exist if $\mathbb{X}$ is a projective torus because in this
case $I(\mathbb{X})$ is a complete intersection \cite{GRH}
(Example~\ref{example-parameterized}). We apply the results of
Section~\ref{section-gorenstein} to show a  
projective local duality theorem that complements
Theorem~\ref{combinatorial-condition}
(Corollary~\ref{combinatorial-condition-projective}). 

In Section~\ref{section-affine}, using the projective closure of an
affine set, we recover the global duality
criterion of L\'opez, Soprunov and Villarreal \cite[Theorem~6.5]{dual} for affine
Reed--Muller-type codes using our projective global duality criterion
(Corollary~\ref{duality-criterion-affine}). 

Section~\ref{section-v-number}  shows an effective method to compute 
the regularity index of the $r$-th generalized Hamming weight 
function in terms of the standard indicator functions of the set of
evaluation points and the notion of $r$-th v-number
(Theorem~\ref{sarabia-vila-2}). 

Let $1\leq r\leq\dim_K(C_\mathbb{X}(d))$ be an integer. 
The $r$-th {\it generalized Hamming weight\/} of $C_\mathbb{X}(d)$ is denoted
$\delta_r(C_\mathbb{X}(d))$ or $\delta_\mathbb{X}(d,r)$
(Eq~\eqref{ghw}). If
$r=1$, $\delta_\mathbb{X}(d,1)$ is the minimum
distance of $C_\mathbb{X}(d)$.
The {\it footprint matrix\/} $({\rm fp}_{I}(d,r))$ 
and the {\it weight matrix} $(\delta_{\mathbb{X}}(d,r))$ of $I$ are the 
matrices of size $r_0\times |\mathbb{X}|$ 
whose $(d,r)$-entries are ${\rm fp}_{I}(d,r)$ and 
$\delta_{\mathbb{X}}(d,r)$, respectively. By convention, in 
the weight and footprint matrices, we add
$\infty$ in the positions where $r>\dim_K(C_\mathbb{X}(d))$. We refer to
\cite{rth-footprint,hilbert-min-dis} for the 
definition of ${\rm fp}_{I}(d,r)$ and the theory of footprint
functions of vanishing ideals. A main result is that 
${\rm
fp}_{I}(d,r)\leq\delta_{\mathbb{X}}(d,r)$
\cite[Theorem~4.9]{rth-footprint} and in certain
cases they are equal (Example~\ref{Hiram-first-counterxample}). The footprint ${\rm
fp}_{I}(d,r)$ is much easier to compute than 
$\delta_{\mathbb{X}}(d,r)$ \cite[Procedure~7.1]{rth-footprint}
(Procedure~\ref{sep12-18}).  
The entries of each row of the weight matrix
$(\delta_{\mathbb{X}}(d,r))$ form an increasing sequence until they
stabilize \cite{wei} and the entries of each column of the weight matrix
$(\delta_{\mathbb{X}}(d,r))$ form a decreasing sequence until they
stabilize \cite[Theorem~5.3]{min-dis-generalized}. 

The $r$-th
v-number of $I$ is denoted by ${\rm{v}}_r(I)$
(Eq.~\eqref{mar8-24-1}) and the regularity index of the $r$-th generalized Hamming weight 
function $\delta_{\mathbb{X}}(\,\cdot\,,r)$ is denoted by $R_r$
(Eq.~\eqref{mar8-24-2}). 
Part of the original interest in the v-number is due to the
equality  ${\rm v}(I)=R_1$
\cite[Proposition~4.6]{min-dis-generalized}, which gives us a formula to
compute $R_1$  
(Example~\ref{Hiram-first-counterxample}, see a \textit{Macaulay}$2$
\cite{mac2} implementation in Appendix~\ref{Appendix}). We generalize this 
equality and show that the regularity index $R_r$ of 
the $r$-th generalized Hamming weight function 
$\delta_{\mathbb{X}}(\,\cdot\,,r)$ is the $r$-th v-number of $I$.

We come to the main result of Section~\ref{section-v-number}. 

\noindent {\bf Theorem~\ref{sarabia-vila-2}.}\textit{
Let $f_1,\ldots, f_m$ be standard indicator
functions of $\mathbb{X}$. The following hold. 
\begin{enumerate}
\item[(a)] $R_r \leq {\rm{reg}}(S/I)$. 
\item[(b)] If $\deg(f_1)\leq\cdots\leq\deg(f_m)$, then
${\rm{v}}_r(I)=\deg(f_r)=R_r$ for $r=1,\ldots,m$.
\item[(c)] Let $\pi$ be a permutation of $\{1,\ldots,m\}$ such that
$\deg(f_{\pi(i)})\leq\deg(f_{\pi(j)})$ for $i<j$. Then, 
${\rm{v}}_r(I):=\deg(f_{\pi(r)})=R_r$ for $r=1,\ldots,m$.
\end{enumerate}
}

As a consequence, our last result shows that if
$\delta_\mathbb{X}(\,\cdot\,,1)$ stabilizes at 
${\rm reg}(S/I)$, then so does $\delta_\mathbb{X}(\,\cdot\,,r)$ for
all $1\leq r\leq m$ (Corollary~\ref{then-so-does}).

We include one section with examples illustrating some of our main 
results (Section~\ref{examples-section}), 
and include an appendix with the \textit{Macaulay}$2$ \cite{mac2} 
implementations of the algorithms used in some of the examples 
to compute the basic parameters and generalized Hamming
weights of projective Reed--Muller-type codes, footprints, h-vectors, v-numbers, 
and the standard indicator 
functions over a finite field (Appendix~\ref{Appendix}). The Gorenstein property is determined
using either the minimal graded free resolution of a vanishing ideal or 
an Artinian reduction. There is a coding theory 
package for \textit{Macaulay}$2$ \cite{mac2} that can be used to study 
evaluation codes \cite{package}.

\section{Preliminaries: Commutative algebra and coding theory}\label{section-prelim} 

In this section of preliminaries, we introduce definitions and results 
from commutative algebra and coding theory. 

\subsection{Commutative algebra}\label{CA} 

Let $S=K[t_1,\ldots,t_s]=\bigoplus_{d=0}^\infty S_d$ be a polynomial ring over a finite
field $K=\mathbb{F}_q$ with the standard grading, let
$\mathbb{X}=\{[P_1],\ldots,[P_m]\}$, $|\mathbb{X}|=m\geq 2$, be a set of
distinct points in the projective space $\mathbb{P}^{s-1}$ over the
field $K$, let $I=I(\mathbb{X})$ be the graded \textit{vanishing ideal} of
$\mathbb{X}$ generated by the homogeneous polynomials in $S$ that
vanish at all points of $\mathbb{X}$, 
and let 
${\mathbf F}$ be the minimal graded free resolution of $S/I$ as an
$S$-module \cite[p.~344]{Vas1}:
\[
{\mathbf F}:\ \ \ 0\rightarrow
\bigoplus_{j}S(-j)^{b_{g,j}}
\stackrel{}{\rightarrow} \cdots
\rightarrow\bigoplus_{j}
S(-j)^{b_{1,j}}\stackrel{}{\rightarrow} S
\rightarrow S/I \rightarrow 0.
\]
\quad The \textit{projective dimension} of $S/I$, denoted ${\rm
pd}_S(S/I)$, is equal to $g$. The {\it Castelnuovo--Mumford regularity\/} of $S/I$ ({\it
regularity} of $S/I$ for short) and the \textit{minimum socle degree}
($s$-\textit{number} for short) of $I$ are defined as \cite[p.~346]{Vas1}:
$${\rm reg}(S/I):=\max\{j-i \mid b_{i,j}\neq 0\}\ \mbox{ and }\ 
s(I):=\min\{j-g \mid b_{g,j}\neq 0\}.$$ 
\quad If 
there is a unique $j$ such that $b_{g,j}\neq 0$, then the ring $S/I$
is called {\em level}. 
In particular, a level ring for 
which the unique $j$ such that $b_{g,j}\neq 0$ satisfies $b_{g,j}=1$
is called {\em Gorenstein}. We say the ideal $I$ is
\textit{Gorenstein} if the quotient ring $S/I$ is
Gorenstein. 

The monomials of $S$ are denoted
$t^c:=t_1^{c_1}\cdots t_s^{c_s}$, $c=(c_1,\dots,c_s)$ in $\mathbb{N}^s$, where
$\mathbb{N}=\{0,1,\ldots\}$. The \textit{support} of $t^c$ is 
${\rm supp}(t^c):=\{t_i\mid c_i\neq 0\}$.
Let $\prec$ be a graded monomial order on $S$, that is, monomials are first compared
by their total degrees \cite[p.~54]{CLO}. The \textit{graded reverse
lexicographic}  order  (GRevLex order) on the 
monomials of $S$ is 
defined by: $t^a\succ t^b$ if either $\deg(t^a)>\deg(t^b)$ or
$\deg(t^a)=\deg(t^b)$ and the last non-zero entry of the vector of
integers $a-b$ is negative. 
  
We denote the initial monomial of a
non-zero polynomial $f\in S$ by ${\rm in}_\prec(f)$ and the initial
ideal of $I$ by ${\rm in}_\prec(I)$ \cite[p.~1]{Stur1}. Let 
$\mathcal{G}=\{g_1,\ldots, g_n\}$ be a {\it Gr\"obner basis\/} of $I$,
that is, ${\rm
in}_\prec(I)$ is equal to $({\rm in}_\prec(g_1),\ldots,{\rm
in}_\prec(g_n))$. 
A monomial $t^a$ is called a 
\textit{standard monomial} of the quotient ring $S/I$, with respect 
to $\prec$, if $t^a\notin{\rm in}_\prec(I)$. The \textit{footprint} of $S/I$, denoted
$\Delta_\prec(I)$, is the set of all standard monomials of $S/I$. We
denote the set of standard monomials of degree $d$ by
$\Delta_\prec(I)_d$. 
The $K$-linear subspace of $S$ spanned by $\Delta_\prec(I)$ is denoted by $K\Delta_\prec(I)$.
A polynomial $f$ is called
a \textit{standard polynomial} of $S/I$ if $f\neq 0$ and $f$ is in
$K\Delta_\prec(I)$. The
footprint is used to study the basic parameters of many kinds of
codes 
\cite{carvalho,geil-2008,geil-hoholdt,geil-pellikaan,rth-footprint,Pellikaan,toric-codes,stefan}.
 
Standard polynomials occur in Gr\"obner basis theory. Given a
polynomial $f\in S$, according to \cite[Proposition~1, p.~81]{CLO},  
there is a unique $r\in S$ with the following
properties:
\begin{enumerate}
\item[(i)] No term of $r$ is divisible by one of ${\rm
in}_\prec(g_1),\ldots,{\rm in}_\prec(g_n)$, that is,
$r\in K\Delta_\prec(I)$. 
\item[(ii)] There is $g\in I$ such that $f=g+r$.
\end{enumerate}
\quad The polynomial $r$, denoted $r_f$, is called the \textit{remainder on division} of $f$
by $\mathcal{G}$ or permitting an abuse of terminology $r$ is also
called the \textit{remainder on division} of $f$ by $I$.

Below, we introduce the notions of regularity index and h-vector. Given an integer 
$d\geq 0$, we let $I_d:=I\cap S_d$ be the $d$-th component
of the graded ideal $I$. The function
\begin{equation}\label{hf-def}
H_I(d):=\dim_K(S_{d}/I_{d}),\ \ \ d=0,1,2,\ldots
\end{equation}
is  called the \textit{Hilbert function} of $S/I$. 
We also denote $H_I$
by $H_\mathbb{X}$. 
By convention
$H_I(d)=0$ for $d=-1$. The Hilbert function is defined for any graded
ideal.

The image of $\Delta_\prec(I)$, under the canonical 
map $S\mapsto S/I$, $x\mapsto \overline{x}$, is a basis of $S/I$ as a
$K$-vector space \cite[Proposition~6.52]{Becker-Weispfenning}. This
is a classical 
result of Macaulay (for a modern approach
see \cite[Chapter~5, Section~3, Proposition 4]{CLO}). In
particular, $H_I(d)$ is the number of standard
monomials of degree $d$ for all $d\geq 0$, that is, $H_I(d)=|\Delta_\prec(I)_d|$ for 
all $d\geq 0$.

Let $H_I$ be the Hilbert function of $S/I$ in Eq.~\eqref{hf-def}. 
By a theorem of Hilbert \cite[p.~58]{Sta1},
there is a unique polynomial 
$h_I(x)\in\mathbb{Q}[x]$ of 
degree $k-1$ such that $H_I(d)=h_I(d)$ for  $d\gg 0$. By convention, the
degree of the zero polynomial is $-1$. The integer $k$ is the
\textit{Krull dimension} of $S/I$ and is denoted by $\dim(S/I)$, and  the {\it degree\/} or 
{\it multiplicity\/} of $S/I$ is the 
positive integer 
$$
\deg(S/I):=\left\{\begin{array}{ll}(k-1)!\, \lim_{d\rightarrow\infty}{H_I(d)}/{d^{k-1}}
&\mbox{if }k\geq 1,\\
\dim_K(S/I) &\mbox{if\ }k=0.
\end{array}\right.
$$ 
\quad Note that the notions of regularity, Hilbert function, Krull
dimension, and
degree can be defined as before for any graded ideal. These invariants
come from algebraic geometry \cite{Eisen,eisenbud-syzygies} and can be
computed using Gr\"obner bases \cite{CLO}.

The \textit{height} of $I$, denoted 
${\rm ht}(I)$,  is equal to $\dim(S)-\dim(S/I)$. 
An element $f\in S$ is called a {\it zero-divisor\/} of $S/I$---as an
$S$-module---if there is
$\overline{0}\neq \overline{a}\in S/I$ such that
$f\overline{a}=\overline{0}$, and $f$ is called {\it regular\/} on
$S/I$ otherwise. The \textit{depth}
of $S/I$, denoted ${\rm depth}(S/I)$, is the largest integer $\ell$
so that there is a homogeneous 
sequence $g_1,\ldots,g_\ell$ in $\mathfrak{m}=(t_1,\ldots,t_s)$ with $g_i$ not a
zero-divisor of $S/(I,g_1,\ldots,g_{i-1})$ for all $1\leq i\leq\ell$.  
The Auslander--Buchsbaum formula \cite[Theorem 3.1]{EvGr}:
\begin{equation}\label{AB}
{\rm pd}_S(S/I) + {\rm depth}(S/I) = \dim(S)=s
\end{equation}
relates the depth and the projective dimension. 

Let 
$A=\bigoplus_{d\geq 0} A_d$ be an Artinian standard graded 
$K$-algebra \cite{Sta1}, that is, $A=S/J$ for some graded ideal $J$ of
$S$ and $\dim(A)=0$. The ring $A$ has finitely many nonzero graded
components because $A$ is Artinian if and only if
$\dim_K(A)<\infty$ \cite[Lemma~2.1.38]{monalg-rev}. 
Letting $\mathfrak{m}=(t_1,\ldots,t_s)$, recall that the {\it socle} of $A$ is
the ideal of $A$ given by 
\begin{equation}\label{socle-def1}
\mathrm{Soc}(A):=(0: A_{+})=(J\colon\mathfrak{m})/J, 
\end{equation}
where $A_{+}:=\bigoplus_{d> 0} A_d=\mathfrak{m}/J$ and
$(J\colon\mathfrak{m}):=\{g\in S\mid g\,\mathfrak{m}\subset J\}$. 
Therefore, as $A$ is also a $K$-algebra,
the socle of $A$ has the structure of a finitely dimensional $K$-vector space. 

The regularity is related to the socle of  Artinian algebras 
\cite{Fro3}, \cite[Lemma~5.3.3]{monalg-rev}. 
We set $J=(I,h)$, where $h$ is a form of degree $1$ regular on $S/I$.
Then, $S/J$ is an Artinian standard graded $K$-algebra and is 
called an \textit{Artinian
reduction} of
$S/I$ (Example~\ref{gorenstein-essential-example}), the socle of
$S/J$ is a $K$-vector space of finite dimension given by
\begin{equation}\label{socle-def}
{\rm Soc}(S/J):=(J\colon\mathfrak{m})/J,
\end{equation}
and the \textit{type} of $S/I$ is equal to 
$\dim_K({\rm Soc}(S/J))$. The ideal $I$ is \textit{Gorenstein} if and
only if ${\rm type}(S/I)=1$ \cite[Corollary~5.3.5]{monalg-rev}. 

\begin{lemma}\label{primdec-ix-a}\cite{hilbert-min-dis} Let
$\mathbb{X}$ be a finite 
subset of $\mathbb{P}^{s-1}$, let $[\alpha]$ be a point in
$\mathbb{X}$  
with $\alpha=(\alpha_1,\ldots,\alpha_s)$
and $\alpha_j\neq 0$ for some $j$, and let
$I([\alpha])$ be the vanishing ideal of $[\alpha]$. Then $I([\alpha])$ is a prime ideal, 
\begin{equation*}
I([\alpha])=(\{\alpha_jt_i-\alpha_it_j\mid j\neq i\in\{1,\ldots,s\}),\
\deg(S/I([\alpha]))=1,\,  
\end{equation*}
${\rm ht}(I([\alpha]))=s-1$, 
and $I(\mathbb{X})=\bigcap_{[\beta]\in{\mathbb{X}}}I([\beta])$ is the primary
decomposition of $I(\mathbb{X})$. 
\end{lemma}  

An \textit{associated prime} of the ideal $I$ is a prime
ideal $\mathfrak{p}$ of $S$ of the form $\mathfrak{p}=(I\colon f)$
for some $f\in S$, where $(I\colon f):=\{g\in S\, \vert\, gf\subset I \}$ is a
\textit{colon ideal}. Note that an element $f\in S$ is a zero-divisor of $S/I$ if
and only if $(I\colon f)\neq I$. The radical of $I$ is denoted by ${\rm rad}(I)$.

\begin{theorem}{\cite[Lemma~2.1.19,
Corollary~2.1.30]{monalg-rev}}\label{zero-divisors} If $L$ is an
ideal of $S$ and
$L=\mathfrak{q}_1\bigcap\cdots\bigcap\mathfrak{q}_m$ is
an irredundant primary decomposition, then the set of zero-divisors
$\mathcal{Z}_S(S/L)$  of $S/L$ is equal to
$\bigcup_{i=1}^m {\rm rad}(\mathfrak{q}_i)$,
and ${\rm rad}(\mathfrak{q}_1),\ldots,{\rm rad}(\mathfrak{q}_m)$ 
are the associated primes of
$L$.
\end{theorem}

\begin{remark}\label{mar17-24} 
The set of all zero-divisors of $S/I$ is equal to 
$\bigcup_{[\beta]\in{\mathbb{X}}}I([\beta])$, where $I([\beta])$ is
the vanishing ideal of $[\beta]$. In particular, a homogeneous
polynomial $h\in S$
is regular on $S/I$ if and only if $h(\beta)\neq 0$ for all
$[\beta]\in\mathbb{X}$. This follows from Lemma~\ref{primdec-ix-a} and
Theorem~\ref{zero-divisors}. 
\end{remark}

The depth of $S/I$ is at most $\dim(S/I)$ 
\cite{monalg-rev}. The ring $S/I$ is
called \textit{Cohen--Macaulay} 
if ${\rm depth}(S/I)=\dim(S/I)$. We say $I$ is Cohen--Macaulay if
$S/I$ is Cohen--Macaulay.  

\begin{remark}\cite[Lemma~1.1]{Sta1}\label{gor-iff-ext} 
Let $\overline{K}$ be the algebraic closure of $K$ \cite[p.~11]{AM}, let 
$\overline{S}=\overline{K}[t_1,\ldots,t_s]$ be the extension of $S$ to
$\overline{S}$ and let
$\overline{I}=I\overline{S}$ be the extension of $I$ to
$\overline{S}$. 
The ideal 
$I$ is Gorenstein (resp. Cohen--Macaulay) 
if and only if $\overline{I}$ is Gorenstein (resp. Cohen--Macaulay). 
\end{remark}

\begin{lemma}\label{aug1-23} {\rm(a)} ${\rm depth}(S/I)=\dim(S/I)=1$, that is, $S/I$ is
Cohen--Macaulay. 

{\rm (b)} ${\rm pd}_S(S/I)={\rm ht}(I)=s-1$.
\end{lemma}

\begin{proof} (a) In general, by \cite[Lemma~2.3.6]{monalg-rev}, one
has ${\rm depth}(S/I)\leq\dim(S/I)$. We may assume that $K$ is an infinite
field (Remark~\ref{gor-iff-ext}). Recall that $\mathfrak{p}_i$ is the
vanishing ideal of $[P_i]$. By Lemma~\ref{primdec-ix-a}, 
the associated primes of $I$ are
$\mathfrak{p}_1,\ldots,\mathfrak{p}_m$,  
and ${\rm ht}(\mathfrak{p}_i)$ is equal to $s-1$ for $i=1,\ldots,m$. Hence, 
$\mathfrak{m}=(t_1,\ldots,t_s)$ is not an associated prime of $I$.
Then, by Theorem~\ref{zero-divisors},
$\mathfrak{m}\not\subset\mathcal{Z}_S(S/I)=
\bigcup_{i=1}^m\mathfrak{p}_i$. Hence, as $K$ is infinite, by
\cite[Proposition~5.6.1]{monalg-rev} we 
can pick a linear form $h$ in $\mathfrak{m}$ which is regular on
$S/I$. Therefore, ${\rm depth}(S/I)\geq 1$.  Thus, $1\leq{\rm depth}(S/I)\leq\dim(S/I)=1$.

(b) By the formula of Eq.~\eqref{AB}, we have ${\rm pd}_S(S/I) + {\rm depth}(S/I) =
\dim(S)=s$. Thus, the
equality follows from part (a). 
\end{proof}

The Hilbert function $H_I$ is strictly increasing until it 
stabilizes. By \cite[Theorem~2.9]{min-dis-generalized},
\cite{geramita-cayley-bacharach}, there is an integer 
$r_0\geq 0$ such that 
\begin{equation}\label{sep18-23}
1=H_I(0)<H_I(1)<\cdots<H_I(r_0-1)<H_I(d)=|\mathbb{X}|
\end{equation}
for all $d\geq r_0$. The \textit{regularity index} of
$H_I$, denoted ${\rm reg}(H_I)$, is equal to $r_0$. 
As $S/I$ is Cohen--Macaulay and has Krull dimension $1$ (Lemma~\ref{aug1-23}), the
regularity ${\rm reg}(S/I)$ of $S/I$ is also equal to $r_0$
\cite[p.~256]{hilbert-min-dis}. 
Note that $r_0\geq 1$ because $|\mathbb{X}|\geq 2$. 

The \textit{Hilbert series} of $S/I$, denoted $F_I(x)$, is the
generating function of $\{H_I(d)\}_{d=0}^\infty$ and 
is given by 
$F_I(x):=\sum_{d=0}^\infty
H_I(d)x^d$. 
This series is a rational function of the form
\begin{equation}\label{hs}
F_I(x)=\frac{h(x)}{1-x},
\end{equation}
where $h(x)$ is a polynomial with non-negative integer coefficients 
\cite{Sta1}, \cite[Theorem~5.2.6]{monalg-rev}, of degree $r_0$
\cite[Corollary B.4.1]{Vas1}, which is called the h-\textit{polynomial} of
$S/I$. 

\begin{definition}\label{h-vector-def} The vector
formed with the coefficients of $h(x)$ is called the
$\mathrm{h}$-\textit{vector} of 
$S/I$ and is denoted by
$\mathrm{h}(S/I)=(h_0,\ldots,h_{r_0})$. The h-vector is
\textit{symmetric} if $h_d=h_{r_0-d}$ for all $0\leq d\leq r_0$.
\end{definition}

A result of Stanley shows that Gorenstein rings have symmetric
h-vectors \cite[Theorems~4.1 and 4.2]{Sta1}. The
degree of the ring $S/I$, which is being denoted by $\deg(S/I)$, is equal to $|\mathbb{X}|$ and
$h(1)=h_0+\cdots+h_{r_0}=|\mathbb{X}|$ (\cite[p.~251]{hilbert-min-dis}, \cite[Remark~5.1.7]{monalg-rev}).

The ideal $I$ is called a \textit{complete intersection} if $I$ can be generated by
${\rm ht}(I)=s-1$ homogeneous polynomials.

\begin{proposition}\cite[Corollary~21.19]{Eisen}\label{ci-gor} If $I$ is a complete
intersection, then $I$ is Gorenstein.
\end{proposition}

\begin{lemma}\label{h-vector-H} If $\mathrm{h}(S/I)=(h_0,\ldots,h_{r_0})$ is
the $\mathrm{h}$-vector of
$S/I$, then 
$$
h_d=H_{I}(d)-H_{I}(d-1)\ \mbox{ for all }\ 0\leq d\leq r_0.
$$
\end{lemma}

\begin{proof} Let $\overline{K}$ be the algebraic closure of $K$ \cite[p.~11]{AM}. We set
$\overline{S}:=S\otimes_K\overline{K}=\overline{K}[t_1,\ldots,t_s]$
and $\overline{I}:=I\overline{S}$. From \cite[Lemma~1.1]{Sta1},
$H_I(d)=H_{\overline{I}}(d)$ for all $d\geq 0$. Since the field $\overline{K}$
is infinite and $\overline{S}/\overline{I}$ is Cohen--Macaulay
(Remark~\ref{gor-iff-ext}, 
Lemma~\ref{aug1-23}), there is a linear form $h\in\overline{S}$ that is regular
on $\overline{S}/\overline{I}$ \cite[Lemma 3.1.28]{monalg-rev}. 
Taking Hilbert series in the exact
sequence 
\begin{equation}\label{jun7-23}
0\longrightarrow
(\overline{S}/\overline{I})(-1)\stackrel{h}{\longrightarrow}
\overline{S}/\overline{I}\longrightarrow\overline{S}/(\overline{I},h) 
\longrightarrow 0,
\end{equation}
we get $F_{\overline{I}}(x)=xF_{\overline{I}}(x)+F_A(x)$, where $F_A(x)$
 and $F_{\overline{I}}(x)$ are the Hilbert series of
 $A=\overline{S}/(\overline{I},h)$ and $\overline{S}/\overline{I}$,
 respectively. Then, using Eq.~\eqref{hs}, we obtain
$$
F_A(x)=(1-x)F_{\overline{I}}(x)=(1-x)F_I(x)=h_0+h_1x+\cdots+
h_{r_0}x^{r_0}.
$$ 
\quad Therefore, by Eq.~\eqref{jun7-23}, we get
$$
h_d=\dim_{\overline{K}}(A_d)=H_{\overline{I}}(d)-H_{\overline{I}}(d-1)=H_{I}(d)-H_{I}(d-1)
$$
for all $0\leq d\leq r_0$, and the proof is complete.
\end{proof}

\begin{proposition}\label{jul5-23} Let $\mathrm{h}(S/I)=(h_0,\ldots,h_{r_0})$ be the
$\mathrm{h}$-vector of
$S/I$. The following conditions are equivalent.
\begin{enumerate}
\item[(a)] The $\mathrm{h}$-vector of $S/I$ is symmetric, that is,
$h_d=h_{r_0-d}$ for all $0\leq d\leq r_0$.
\item[(b)] $H_I(d)+H_I(r_0-d-1)=|\mathbb{X}|$ for all $0\leq d\leq r_0$.
\end{enumerate}
\end{proposition}

\begin{proof} (a)$\Rightarrow$(b) Recall, $r_0={\rm reg}(S/I)$.
Assume that $\mathrm{h}(S/I)$ is symmetric.  
Then, by Lemma~\ref{h-vector-H}, one has
\begin{equation}\label{jun6-23}
h_d=H_I(d)-H_I(d-1)=H_I(r_0-d)-H_I(r_0-d-1)=h_{r_0-d}\ \mbox{ for all }\
0\leq d\leq r_0.
\end{equation}
\quad To show the equality in (b) we use induction on
$d$. If
$d=0$, by Eq.~\eqref{jun6-23}, one has
$$
H_I(0)=H_I(0)-H_I(-1)=H_I(r_0)-H_I(r_0-1)=|\mathbb{X}|-H_I(r_0-1)
$$
and equality holds in (b) for $d=0$. Assume $r_0>d-1\geq 0$ and that equality
holds in (b) for $d-1$, that is, we have the equality 
\begin{equation}\label{jun6-23-1}
H_I(d-1)+H_I(r_0-d)=|\mathbb{X}|.
\end{equation}
\quad Then, by Eqs.~\eqref{jun6-23}-\eqref{jun6-23-1}, one has
$$
H_I(d)+H_I(r_0-d-1)=H_I(d-1)+H_I(r_0-d)=|\mathbb{X}|,
$$
and equality holds in (b) for $d$.

(b)$\Rightarrow$(a) Assume that $H_I(d)+H_I(r_0-d-1)=|\mathbb{X}|$
for all $0\leq d\leq r_0$. Then
$$
H_I(d-1)+H_I(r_0-d)=H_I(d)+H_I(r_0-d-1)\mbox{ for all }0\leq d\leq r_0
$$
and, by Lemma~\ref{h-vector-H},
$h_d=H_I(d)-H_I(d-1)=H_I(r_0-d)-H_I(r_0-d-1)=h_{r_0-d}$. Thus, the
$\mathrm{h}$-vector of $S/I$ is symmetric and the proof is complete.
\end{proof}

Let ${\rm Ass}(I)=\{\mathfrak{p}_1,\ldots,\mathfrak{p}_m\}$ be the 
set of \textit{associated primes} of 
$I$, that is, $\mathfrak{p}_i$ is the vanishing 
ideal $I([P_i])$ of $[P_i]$ and $I=\bigcap_{j=1}^m\mathfrak{p}_j$ is the primary
decomposition of $I$ (Lemma~\ref{primdec-ix-a}). 
The v-{\em number} of $I$, denoted ${\rm
v}(I)$, is the following algebraic invariant of
$I$ that was introduced in \cite{min-dis-generalized}:
\begin{equation}\label{jun17-23}
{\rm v}(I):=\min\{d\geq 0 \mid\exists\, f 
\in S_d \mbox{ and }\mathfrak{p} \in {\rm Ass}(I) \mbox{ with } (I\colon f)
=\mathfrak{p}\}.
\end{equation}
\quad One can define the v-number of $I$ locally at each associated 
prime $\mathfrak{p}_i$ of $I$\/:
\begin{equation}\label{jun17-23-1}
{\rm v}_{\mathfrak{p}_i}(I):=\mbox{min}\{d\geq 0\mid \exists\, f\in S_d
\mbox{ with }(I\colon f)=\mathfrak{p}_i\}. 
\end{equation}
\quad Thus, one has the equality ${\rm v}(I)={\rm min}\{{\rm
v}_{\mathfrak{p}_i}(I)\}_{i=1}^m$.  
The v-number of $I$ can be computed using 
the following description for the v-number of $I$ locally 
at $\mathfrak{p}_i$
\cite[Proposition~4.2]{min-dis-generalized}:
\begin{equation*}
{\rm
v}_{\mathfrak{p}_i}(I)=\alpha\!\left((I\colon\mathfrak{p}_i)/{I}\right),\
i=1,\ldots,m,
\end{equation*}
where $\alpha\!\left((I\colon\mathfrak{p}_i)/{I}\right)$ is the 
minimum degree of the non-zero elements of
$(I\colon\mathfrak{p}_i)/{I}$ (Example~\ref{local-duality-example}). 
The degree $\deg_\mathbb{X}([P_i])$ of
$[P_i]$, in the
sense of \cite[Definition~2.1]{geramita-cayley-bacharach}, is equal to
${\rm v}_{\mathfrak{p}_i}(I)$. 

\subsection{Coding theory}\label{CT} In this subsection we introduce
some generalities on linear codes and Reed--Muller-type codes. 
\subsubsection{Linear codes} Let $K=\mathbb{F}_q$ be a finite
field and let $C$ be an  
$[m,k]$-\textit{linear
code} of {\it length} $m$ and {\it dimension} $k$, 
that is, $C$ is a linear subspace of $K^m$ with $k=\dim_K(C)$. Let
$1\leq r\leq k$ be an integer.   
Given a linear subspace $D$ of $C$, the \textit{support} of
$D$, denoted $\chi(D)$, is the set of nonzero positions of $D$, that is,
$$
\chi(D):=\{i\,\colon\, \exists\, (a_1,\ldots,a_m)\in D,\, a_i\neq 0\}.
$$
\quad The $r$-th \textit{generalized Hamming weight} of $C$, denoted
$\delta_r(C)$, is given by \cite[p.~283]{Huffman-Pless}:
\begin{equation}\label{ghw}
\delta_r(C):=\min \{|\chi(D)|\colon D \mbox{ is a subspace of }C,\, \dim_K (D)=r\}.
\end{equation}
\quad 
As usual we call the set $\{\delta_1(C), \ldots, \delta_k(C)\}$ the 
\textit{weight hierarchy} of the linear code $C$. 
The $1$st Hamming weight of $C$ is the \textit{minimum
distance} $\delta(C)$ of $C$, that is, one has
$$
\delta_1(C)=\delta(C)=\min\{\|\alpha\|\colon \alpha \in C \setminus
\{0\}\},
$$ 
where $\|\alpha\|$ is the (Hamming) weight of $\alpha$, that is,
$\|\alpha\|$ is the number of non-zero
entries of $\alpha$. 
To determine the
minimum distance
is essential to find good error-correcting codes
\cite{MacWilliams-Sloane}. 

Generalized Hamming weights of linear codes are parameters
of interest in many applications 
\cite{rth-footprint,Pellikaan,Johnsen,olaya,schaathun-willems,tsfasman,wei,Yang}
and they have been nicely related to  
the minimal graded free resolution of the ideal of cocircuits of the matroid
of a linear code \cite{JohVer,Johnsen}, to the nullity
function of the dual matroid of a linear code \cite{wei}, and to the
enumerative combinatorics of linear codes 
\cite{Britz,Klove-1992,MacWilliams-Sloane}.
Because of this, 
their study has attracted considerable attention, 
but determining them is in general a
difficult problem. Generalized Hamming weights were 
introduced by Helleseth, Kl{\o}ve and Mykkeltveit 
\cite{helleseth} and were first used systematically by Wei 
\cite{wei}.

The \textit{dual} of $C$, denoted $C^\perp$, is the
set of all $\alpha\in K^m$ such that $\langle
\alpha,\beta\rangle=0$ for all $\beta\in C$, where
$\langle\ ,\, \rangle$ is the ordinary inner product in $K^m$. 
The dual $C^\perp$ of $C$ is a linear code
of length $m$ and dimension
$m-k$ \cite[Theorem~1.2.1]{Huffman-Pless}. 
A linear code $C\subset K^m$ is called \textit{self dual} if
$C=C^\perp$ and \textit{self orthogonal} if $C\subset C^\perp$. 

\begin{definition}\label{me}
We say that two linear codes $C_1,C_2$ in $K^m$
are \textit{monomially equivalent} if there is a vector $\beta=(\beta_1,\ldots,\beta_m)$
in $K^m$ such that $\beta_i\neq 0$ for all $i$ and 
$$C_2 = \beta\cdot C_1=\{\beta\cdot c \mid c\in C_1\},$$
where $\beta\cdot c$ is the vector given by
$(\beta_1c_1,\ldots,\beta_mc_m)$ for 
$c=(c_1,\ldots,c_m)\in C_1$. 
\end{definition}
Monomially equivalent codes have the same basic parameters.

\begin{remark}\label{dual-equation}
Monomial equivalence of linear codes is an equivalence relation. Let
$C_1,C_2$ be two linear codes in $K^m$ and let
$\beta=(\beta_1,\ldots,\beta_m)\in(K^*)^m$.  
If $C_2=\beta\cdot C_1$, then
$C_1=\beta^{-1}\cdot C_2$ and $(\beta\cdot
C_2)^\perp=\beta^{-1}\cdot C_2^\perp$,  where
$\beta^{-1}=(\beta_1^{-1},\ldots,\beta_m^{-1})$. If 
$C_1^\perp=\beta\cdot C_2$, then
$C_2^\perp=\beta\cdot C_1$.

If we multiply a generator matrix of $C_1$ to
the right by the $m\times m$ diagonal matrix with entries on the main
diagonal 
$\beta_1,\ldots,\beta_m$, we
obtain a generator matrix for $C_2$. This is part of what monomial 
equivalent codes means: multiply by a monomial matrix, which is a
diagonal 
matrix times a permutation matrix.
\end{remark}

\subsubsection{Projective Reed--Muller-type codes}\label{RM} 
Fix a degree $d\geq
1$ and choose a set of representatives
$\mathcal{P}=\{P_1,\ldots,P_m\}$ for the points of $\mathbb{X}$. 
There is a $K$-linear map given by  
\begin{equation}\label{evaluation-map}
{\rm ev}_d\colon S_d\rightarrow K^{m},\ \ \ \ \ 
f\mapsto
\left(f(P_1),\ldots,f(P_m)\right).
\end{equation}
\quad This map is called the \textit{evaluation map} of
degree $d$. The image of $S_d$ under ${\rm ev}_d$, denoted
$C_\mathbb{X}(d)$, is called a {\it projective Reed--Muller-type code\/} of
degree $d$ on $\mathbb{X}$ \cite{duursma-renteria-tapia,GRT}. If
$\mathcal{L}$ is a $K$-linear subspace of $S_d$, the image of
$\mathcal{L}$ under ${\rm ev}_d$, denoted $\mathcal{L}_\mathbb{X}$, is
called a \textit{projective evaluation code} of degree $d$ on $\mathbb{X}$. The
points in $\mathbb{X}$ are called evaluation points. The {\it basic parameters} of the linear
code $C_\mathbb{X}(d)$ are its {\it length\/} $|\mathbb{X}|$, {\it
dimension\/} 
$\dim_K(C_\mathbb{X}(d))$, and {\it minimum distance\/} 
$\delta_\mathbb{X}(d):=\delta(C_\mathbb{X}(d))$. 

The basic parameters of $C_\mathbb{X}(d)$ can be
expressed in terms of the
algebraic invariants of $S/I$:  
\begin{itemize}
\item[\rm(i)] $|\mathbb{X}|=\deg(S/I)$ \cite[p.~83]{algcodes},   
\item[\rm(ii)] $\dim_K(C_\mathbb{X}(d))=H_I(d)$ \cite[p.~83]{algcodes},
\item[\rm(iii)] $\delta_\mathbb{X}(d)=\min\{\deg(S/(I\colon f))\mid
f\in S_d\setminus I\}$ \cite[Theorem~4.4]{hilbert-min-dis}. 
\end{itemize}

The v-number of the vanishing ideal $I$ of $\mathbb{X}$ is related to
the asymptotic behavior of
$\delta_\mathbb{X}(d)$ for $d\gg 0$. By
\cite[Theorem~5.3]{min-dis-generalized} and
\cite[Theorem~4.5]{rth-footprint}, there is $R_1\in \mathbb{N}$ such that  
$$
|\mathbb{X}|=\delta_\mathbb{X}(0)>\delta_\mathbb{X}(1)>\cdots>
\delta_\mathbb{X}(R_1-1)>\delta_\mathbb{X}(R_1)=\delta_\mathbb{X}(d)=1\
\mbox{ for all }\ d\geq R_1.
$$
\quad The number $R_1$, denoted ${\rm reg}(\delta_\mathbb{X})$, is called the 
\textit{regularity index} of $\delta_\mathbb{X}$. By the 
Singleton bound \cite[p.~71]{Huffman-Pless}, the Hilbert function of
$S/I$ at $d$ and  
the minimum distance of $C_\mathbb{X}(d)$ are related by 
$$
1\leq\delta_\mathbb{X}(d)\leq|\mathbb{X}|-H_I(d)+1\ \mbox{ for all }\
d\geq 0
$$
and, as a consequence, one has ${\rm reg}(\delta_\mathbb{X})\leq{\rm
reg}(H_I)$. 

The vanishing ideal
of $\mathbb{X}$ and the parameters of $C_\mathbb{X}(d)$ are independent of the set of representatives
$\mathcal{P}=\{P_1,\ldots,P_m\}$ that we
choose for the points of $\mathbb{X}$:

\begin{remark}\label{changing-rep} 
Let $\mathcal{P}=\{P_1,\ldots,P_m\}$ and
$\mathcal{Q}=\{Q_1,\ldots,Q_m\}$ be two sets of representatives for
the points of $\mathbb{X}$ and let ${\rm ev}_d$ and ${\rm ev}_d'$ be
the evaluation maps of degree $d$ of Eq.~\eqref{evaluation-map},  
defined by $\mathcal{P}$ and
$\mathcal{Q}$, respectively. If $P_i=\lambda_iQ_i$, $\lambda_i\in
K^*$, for $i=1,\ldots,m$, then
$$
{\rm ev}_d(S_d)=(\lambda_1^d,\ldots,\lambda_m^d)\cdot{\rm ev}_d'(S_d),
$$
that is, ${\rm ev}_d(S_d)$ is monomially equivalent to ${\rm
ev}_d'(S_d)$, and the basic parameters and generalized Hamming weights
of ${\rm ev}_d(S_d)$ and ${\rm ev}_d'(S_d)$ are the same. Thus, 
the parameters of $C_\mathbb{X}(d)$ are independent of the set of representatives 
that we choose for the points of $\mathbb{X}$. This can also be
seen using the fact that the parameters of $C_\mathbb{X}(d)$ depend only on $I(\mathbb{X})$
\cite[Theorem~4.5]{rth-footprint} and observing that 
$f(P_i)=\lambda_i^ef(Q_i)$ for any homogeneous polynomial $f\in S$ of
degree $e$. This proves that $I(\mathbb{X})$ is independent
of the set of representatives we
choose for the points of $\mathbb{X}$.
\end{remark}

\section{Duality criteria for projective Reed--Muller-type codes}
\label{duality-section}

To avoid repetitions, we continue to employ 
the notations and definitions used in
Section~\ref{section-prelim}. 
\begin{definition}\cite{sorensen}\label{indicator-f} 
A homogeneous polynomial $f\in S$ is called an
\textit{indicator function} of $[P_i]$ if $f(P_i)\neq 0$ and $f(P_j)=0$ if
$j\neq i$.  
\end{definition}

An indicator
function is also called a separator
\cite[Definition~2.1]{geramita-cayley-bacharach}.
For other types of indicator functions,
see \cite{Eduardo-Saenz-indicator}. 
An indicator function $f$ of $[P_i]$ can be normalized to have value 
$1$ at $[P_i]$ by considering $f/f(P_i)$. 
The following lemma lists basic properties of indicator functions of
projective points. 

\begin{lemma}\label{if} {\rm (a)} The set of indicator functions of $[P_i]$ is the set of
homogeneous polynomials in $(I\colon\mathfrak{p}_i)\setminus I$,
where 
$\mathbb{X}=\{[P_1],\ldots,[P_m]\}$, $I=I(\mathbb{X})$, and $\mathfrak{p}_i$ is the
vanishing ideal of $[P_i]$.

{\rm(b)} The least degree of an indicator function of $[P_i]$ is equal 
to ${\rm v}_{\mathfrak{p}_i}(I)$.

{\rm(c)} If for each $i$, $h_i$ is an indicator function of $[P_i]$, 
then $\{h_i\}_{i=1}^m$ is $K$-linearly independent.

{\rm(d)} If $r_0={\rm reg}(H_I)$, then ${\rm
v}_{\mathfrak{p}_i}(I)\leq r_0$ for all $i$ and ${\rm
v}_{\mathfrak{p}_i}(I)=r_0$ for some $i$.

{\rm(e)} If $f,g$ are indicator functions of 
$[P_i]$ in $K\Delta_\prec(I)_d$, then $g(P_i)f=f(P_i)g$.
\end{lemma}

\begin{proof} (a) Recall the equality
$I=\bigcap_{j=1}^m\mathfrak{p}_j$ (Lemma~\ref{primdec-ix-a}). Then, 
$(I\colon\mathfrak{p}_i)=\bigcap_{j\neq i}\mathfrak{p}_j$. Let $f$ be
an indicator function of $[P_i]$, then $f\not\in\mathfrak{p}_i$ and $f\in\mathfrak{p}_j$ 
for $j\neq i$. Thus,
$f\in(I\colon\mathfrak{p}_i)$ and $f\notin I$. Conversely, take a
homogeneous polynomial $f$ in 
 $(I\colon\mathfrak{p}_i)\setminus I$. Then, $f\in \bigcap_{j\neq
i}\mathfrak{p}_j$ and $f\notin\mathfrak{p}_i$. Thus, $f$ is an
indicator function of $[P_i]$.

(b) A homogeneous polynomial $f\in S$ is an indicator function 
of $[P_i]$ if and only if $(I\colon f)=\mathfrak{p}_i$. 
This follows using that the primary decomposition of $I$ is given by 
$I=\bigcap_{j=1}^m\mathfrak{p}_j$ and 
noticing the equality $(I\colon f)=\bigcap_{f\notin
\mathfrak{p}_j}\mathfrak{p}_j$. 
Hence, by the definition of ${\rm v}_{\mathfrak{p}_i}(I)$ given in
Eq.~\eqref{jun17-23-1}, ${\rm v}_{\mathfrak{p}_i}(I)$ is the minimum degree of an
indicator function of $[P_i]$. 
 
(c) The linear independence over $K$ follows using that $h_i(P_j)=0$ if
$i\neq j$ and $h_i(P_i)\neq 0$. 

(d) Since $H_I(r_0)=\dim_K(C_\mathbb{X}(r_0))=m$, maximum possible 
(Eq.~\eqref{sep18-23}), one
has $C_\mathbb{X}(r_0)=K^m$. Hence, for each $1\leq i\leq m$ there is
$w_i\in S_{r_0}$ such that 
${\rm ev}_{r_0}(w_i)=(w_i(P_1),\ldots,w_i(P_m))=e_i$, where $e_i$ is
the $i$-th unit vector in $K^m$. Thus, $w_i$ is
an indicator function of 
$[P_i]$ and, by part (b), we obtain that ${\rm
v}_{\mathfrak{p}_i}(I)\leq r_0$. To show that ${\rm
v}_{\mathfrak{p}_i}(I)=r_0$ for some $i$, we argue by contradiction
assuming that ${\rm v}_{\mathfrak{p}_i}(I)<r_0$ for $i=1,\ldots,m$.
For each $1\leq i\leq m$ there is an indicator function $g_i$ of $[P_i]$ such
that $\deg(g_i)={\rm v}_{\mathfrak{p}_i}(I)$, and there is $1\leq
j_i\leq s$ such that the $j_i$-th entry of $P_i$ is non-zero, that is,
$(P_i)_{j_i}\neq 0$. Consider the homogeneous polynomials
$$
F_i:=t_{j_i}^{r_0-1-\deg(g_i)}g_i,\quad i=1,\ldots,m.
$$
\quad Note that $F_i$ is an 
indicator function of $[P_i]$ of degree $r_0-1$ for $i=1,\ldots,m$.
By the division algorithm \cite[Theorem~3, 
p.~63]{CLO}, we can write $F_i=G_i+h_i$ for some polynomial $G_i\in I$
homogeneous of degree $r_0-1$ and some $h_i\in
K\Delta_\prec(I)_{r_0-1}$, that is, $h_i$ is a standard indicator
function of $[P_i]$ of degree $r_0-1$ for $i=1,\ldots,m$. By part (c),
the set
$\{h_i\}_{i=1}^m$ is linearly independent over $K$. It follows 
readily that the set of cosets $\{h_i+I\}_{i=1}^m$ is a $K$-linearly independent
subset of the $(r_0-1)$-th homogeneous component
$(S/I)_{r_0-1}=S_{r_0-1}/I_{r_0-1}$ of $S/I$. Therefore, one has 
$m=H_I(r_0)>H_I(r_0-1)\geq m$, a contradiction (see Eq.~\eqref{sep18-23}).

(e) The polynomial $h=g(P_i)f-f(P_i)g$ is homogeneous and 
 vanishes at all points of $\mathbb{X}$, that is, $h\in I$. If $h\neq 0$, then
the initial monomial of $h$ is in the initial ideal of $I$, a contradiction
since all monomials that appear in $h$ are standard. Thus, $h=0$ and
$g(P_i)f=f(P_i)g$. 
\end{proof}

\begin{definition}\label{sif} We say $f\in S$ is a 
\textit{standard indicator function} of $[P_i]$ if $f$ is
an indicator function of $[P_i]$ and $f\in
K\Delta_\prec(I)$. 
For each $[P_i]\in\mathbb{X}$, 
the unique (up to multiplication by a scalar from
$K^*$) standard indicator
function $f_i$ of $[P_i]$ of degree ${\rm v}_{\mathfrak{p}_i}(I)$ is
called the $i$-th \textit{standard indicator function} of $[P_i]$ and
$F:=\{f_1,\ldots,f_m\}$ is called \textit{the set of standard indicator
functions} of $\mathbb{X}$. The existence and 
uniqueness of $f_i$ is proved in the following
result.
\end{definition}  

\begin{proposition}\label{indicator-function-prop} Let
$\mathbb{X}=\{[P_1],\ldots,[P_m]\}$ be a subset of
$\mathbb{P}^{s-1}$, let $I=I(\mathbb{X})$ be its vanishing ideal, and let $\prec$ be a graded
monomial order on $S$. The following hold.
\begin{enumerate} 
\item[(a)] For each $[P_i]$ there exists a unique, up to multiplication by a scalar from
$K^*$, standard indicator function $f_i$ of $[P_i]$ of degree ${\rm
v}_{\mathfrak{p}_i}(I)$, where $\mathfrak{p}_i$ is the vanishing ideal
of $[P_i]$.
\item[(b)] $(I\colon\mathfrak{p}_i)/I$ is a principal ideal of $S/I$
generated by $\overline{f_i}=f_i+I$ for $i=1,\ldots,m$.
\end{enumerate}
\end{proposition}

\begin{proof} (a) By Lemma~\ref{if}(b), there is $g\in S_{d}$, $d={\rm
v}_{\mathfrak{p}_i}(I)$, such that
$g$ is an indicator function of $[P_i]$. As $I$ is
a graded ideal and $g$ is homogeneous of degree $d$, by
the division algorithm \cite[Theorem~3, 
p.~63]{CLO}, we can write $g=h+r_g$ for some $h\in I$ homogeneous of
degree $d$ and some $r_g\in
K\Delta_\prec(I)$ with $\deg(r_g)=d$. Note that
$r_g(P_i)\neq 0$ and $r_g(P_j)=0$ for $j\neq i$. 
Thus, $f_i:=r_g$ is a standard indicator function of $[P_i]$ of degree
${\rm v}_{\mathfrak{p}_i}(I)$. This proves the existence of $f_i$.  The
uniqueness of $f_i$ follows at once from Lemma~\ref{if}(e).

(b) By Lemma~\ref{if}(a), $f_i\in(I\colon\mathfrak{p}_i)\setminus I$.
Thus, $\overline{f_i}\in (I\colon\mathfrak{p}_i)/I$ and
$(S/I)\overline{f_i}\subset (I\colon\mathfrak{p}_i)/I$. To show the
other inclusion take $\overline{0}\neq
\overline{g}\in(I\colon\mathfrak{p}_i)/I$. As
$(I\colon\mathfrak{p}_i)$ is graded, we may assume that $g$ is
homogeneous of degree $d_1$ and $g\in(I\colon\mathfrak{p}_i)\setminus
I$. Then, by Lemma~\ref{if}(a), $g$ is an indicator function of 
$[P_i]$ and, by part (a) and Lemma~\ref{if}(b), 
$d:={\rm v}_{\mathfrak{p}_i}(I)=\deg(f_i)\leq d_1$. As $P_i\neq 0$, we
can pick $t_j$ such that $t_j(P_i)\neq 0$. Then, the homogeneous polynomial
$$ 
h=\frac{t_j^{d_1-d}f_i}{(t_j^{d_1-d}f_i)(P_i)}-\frac{g}{g(P_i)}
$$
vanishes at all points of $\mathbb{X}$ and is either $0$ or has 
degree $d_1$. Thus, $h\in I$ and 
$\overline{g}\in(S/I)\overline{f_i}$.
\end{proof}

\begin{proposition}\cite[Proposition~4.6]{min-dis-generalized}\label{reg-min-dis} 
 ${\rm v}(I)={\rm reg}(\delta_\mathbb{X})$.
\end{proposition}

\begin{proposition}\label{sarabia-vila} 
If ${\rm v}_{\mathfrak{p}_i}(I)=r_0$ for $i=1,\ldots,m$
and $|\mathbb{X}|=1+H_I(r_0-1)$, then $\delta_\mathbb{X}(r_0-1)=2$.
\end{proposition}

\begin{proof} By Proposition~\ref{reg-min-dis}, one has ${\rm
v}(I)={\rm reg}(\delta_\mathbb{X})=r_0$. Thus, $\delta_\mathbb{X}(r_0)=1$
and $\delta_\mathbb{X}(r_0-1)\geq 2$. Using the Singleton bound for
the minimum distance \cite[p.~71]{Huffman-Pless}, we get
\begin{align*}
\delta_\mathbb{X}(r_0-1)&\leq|\mathbb{X}|-\dim(C_\mathbb{X}(r_0-1))+1
=|\mathbb{X}|-H_I(r_0-1)+1\\
&=|\mathbb{X}|-(|\mathbb{X}|-1)+1=2,
\end{align*}
and consequently the equality $\delta_\mathbb{X}(r_0-1)=2$ holds.
\end{proof}

\begin{lemma}\label{sarabia-vila-1} 
If $C_\mathbb{X}(0)^\perp=(b_1,\ldots,b_m)\cdot C_\mathbb{X}(r_0-1)$
and $0\neq b_i\in K$ for all $i$, then 
${\rm v}_{\mathfrak{p}_i}(I)=r_0$ for $i=1,\ldots,m$.
\end{lemma}

\begin{proof} 
By Lemma~\ref{if}(d), one has
${\rm v}_{\mathfrak{p}_i}(I)\leq r_0$ for $i=1,\ldots,m$. Recall that  
${\rm v}_{\mathfrak{p}_i}(I)$ is the least degree of an indicator
function of $[P_i]$ (Lemma~\ref{if}(b)). Let $f$ be any 
indicator function of $[P_i]$. It suffices to show that $\deg(f)\geq
r_0$. We proceed by contradiction assuming that $\deg(f)<r_0$. Since
$P_i\neq 0$, there is $t_k$ such that $t_k(P_i)\neq 0$. Then, setting
$g=t_k^\ell f$, $\ell=r_0-1-\deg(f)$, one has that $g$ is an indicator
function of $[P_i]$ of degree $r_0-1$. For simplicity assume $i=1$.
Note that $C_\mathbb{X}(0)=K(1,\ldots,1)$.  
The point $Q=(g(P_1),0\ldots,0)$ is in $C_\mathbb{X}(r_0-1)$
because ${\rm ev}_{r_0-1}(g)=Q$. Hence, letting $b=(b_1,\ldots,b_m)$,
$b_i\neq 0$, 
we get $b\cdot Q\in C_\mathbb{X}(0)^\perp$. Thus
$$
0=\langle b\cdot Q,(1,\ldots,1)\rangle=\langle
(b_1g(P_1),0,\ldots,0),(1,\ldots,1)\rangle=b_1g(P_1),
$$ 
and consequently $b_1=0$, a contradiction. 
\end{proof}

We come to one of our main results. 
\begin{theorem}{\rm(Projective duality
criterion)}\label{duality-criterion} Let
$\mathbb{X}=\{[P_1],\ldots,[P_m]\}$ be a subset of
$\mathbb{P}^{s-1}$, let $I$ be its vanishing ideal, and
let $r_0$ be the regularity of $S/I$. The following are equivalent. 
\begin{enumerate} 
\item[(a)] $C_\mathbb{X}(r_0-d-1)$ is monomially equivalent to $C_\mathbb{X}(d)^\perp$  for 
all $0\leq d\leq r_0$.  
\item[(b)] $H_I(d)+H_I(r_0-d-1)=|\mathbb{X}|$ for all $0\leq d\leq r_0$ 
and $r_0={\rm v}_{\mathfrak{p}}(I)$ for all $\mathfrak{p}\in{\rm Ass}(I)$.
\item[(c)] There is a vector $\beta=(\beta_1,\ldots,\beta_m)\in K^m$ such
that $\beta_i\neq 0$ for all $i$ and
\begin{equation*}
C_\mathbb{X}(d)^\perp=(\beta_1,\ldots,\beta_m)\cdot
C_\mathbb{X}(r_0-d-1)\ \text{ for all \ $0\leq d\leq r_0$}.
\end{equation*}
Moreover, $\beta$ is any vector that defines a parity  
check matrix of $C_\mathbb{X}(r_0-1)$. 
\end{enumerate}
\end{theorem}

\begin{proof} (a)$\Rightarrow$(b) Since
$\dim_K(C_\mathbb{X}(r_0-d-1))=H_I(r_0-d-1)$ and
$\dim_K(C_\mathbb{X}(d)^\perp)=|\mathbb{X}|-H_I(d)$, we obtain that 
$H_I(r_0-d-1)+H_I(d)=|\mathbb{X}|$ for all $0\leq d\leq r_0$ because
equivalent codes have the same dimension. As 
$C_\mathbb{X}(r_0-1)$ is equivalent to $C_\mathbb{X}(0)^\perp$, there is
$b\in(K^*)^m$ such that $C_\mathbb{X}(0)^\perp=b\cdot C_\mathbb{X}(r_0-1)$. Then, by
Lemma~\ref{sarabia-vila-1}, ${\rm v}_{\mathfrak{p}}(I)=r_0$ for
all  $\mathfrak{p}\in{\rm Ass}(I)$.

(b)$\Rightarrow$(c) By Proposition~\ref{sarabia-vila}, one has
$\delta_\mathbb{X}(r_0-1)=2$. If $d=0$, by assumption, we get
$H_I(r_0-1)=|\mathbb{X}|-1=m-1$. Pick any parity check matrix $H$ of 
$C_\mathbb{X}(r_0-1)$ \cite[p.~4]{Huffman-Pless}. Since
$\dim_K(C_\mathbb{X}(r_0-1))=m-1$, $H=(\beta_1,\ldots, \beta_m)$ is a
$1\times m$ matrix 
with $\beta_i\in K$ for all $i$ and 
\begin{equation}\label{jun8-23}
C_\mathbb{X}(r_0-1)=\{x\in K^m\mid Hx=0\}.
\end{equation}
\quad We set $\beta=(\beta_1,\ldots,\beta_m)$. If $\beta_i=0$ for some
$1\leq i\leq m$, then the $i$-th unit vector $e_i$ satisfies $He_i=0$
and, by Eq.~\eqref{jun8-23}, we obtain $e_i\in C_\mathbb{X}(r_0-1)$. 
Hence, $\delta_\mathbb{X}(r_0-1)=1$, a
contradiction. Thus, $\beta_i\neq 0$ for $i=1,\ldots,m$. From 
Eq.~\eqref{jun8-23}, we get 
$\beta\in C_\mathbb{X}(r_0-1)^\perp$ and, since
$\dim_K(C_\mathbb{X}(r_0-1)^\perp)=1$, we obtain  
\begin{equation}\label{jun8-23-1}
C_\mathbb{X}(r_0-1)^\perp=K(\beta_1,\ldots,\beta_m)=K\beta.
\end{equation}
\quad From the equalities
\begin{align*}
\dim_K(C_\mathbb{X}(d)^\perp)&=|\mathbb{X}|-H_I(d)=H_I(r_0-d-1)\\
&=\dim_K(C_\mathbb{X}(r_0-d-1))
=\dim_K(\beta\cdot C_\mathbb{X}(r_0-d-1)),
\end{align*}
we get $\dim_K(C_\mathbb{X}(d)^\perp)=\dim_K(\beta\cdot
C_\mathbb{X}(r_0-d-1))$. Thus, we need only show 
the inclusion 
$$
C_\mathbb{X}(d)^\perp\supset\beta\cdot C_\mathbb{X}(r_0-d-1).
$$
\quad Take $\gamma_1\in\beta\cdot C_\mathbb{X}(r_0-d-1)$. Then,
$\gamma_1=(\beta_1g(P_1),\ldots,\beta_mg(P_m))$, where $g$ 
is a homogeneous polynomial of degree $r_0-d-1$. Take any
$\gamma_2\in C_\mathbb{X}(d)$. Then, $\gamma_2=(f(P_1),\ldots,f(P_m))$,
where $f$ is a homogeneous polynomial of degree $d$. Then
$$
\langle\gamma_1,\gamma_2\rangle=\sum_{i=1}^m\beta_i(gf)(P_i).
$$
\quad By Eq.~\eqref{jun8-23-1}, $C_\mathbb{X}(r_0-1)^\perp=K\beta$, 
if $c:=((gf)(P_1),\ldots,(gf)(P_m))\in C_\mathbb{X}(r_0-1)$, then 
$\langle\beta,c\rangle=0$, and so, from above,
$\langle\gamma_1,\gamma_2\rangle=0$. Hence, $\gamma_1\in
C_\mathbb{X}(d)^\perp$, and the proof is complete. 

(c)$\Rightarrow$(a) This follows at once from the definition of
monomially equivalent
codes.
\end{proof}

\begin{lemma}\label{sep5-23}
If $C_\mathbb{X}(r_0-1)$ is monomially equivalent to
$C_\mathbb{X}(0)^\perp$, then there is a unique, up to 
multiplication by a scalar from $K^*$, vector $\beta$ in $(K^*)^m$
such that 
$$
C_\mathbb{X}(0)^\perp=\beta\cdot C_\mathbb{X}(r_0-1).
$$
\end{lemma}

\begin{proof} Let $\beta$, $\beta'$ be two vectors in $(K^*)^m$ such
that 
\begin{equation}\label{sep5-23-1}
K(1,\ldots,1)^\perp=C_\mathbb{X}(0)^\perp=\beta\cdot
C_\mathbb{X}(r_0-1)=\beta'\cdot C_\mathbb{X}(r_0-1).
\end{equation}
\quad Note that $\dim_K(C_\mathbb{X}(r_0-1)^\perp)=1$ because
$\dim_K(C_\mathbb{X}(0)^\perp)=m-1$. Thus, it suffices to show that
$\beta$ and $\beta'$ are in $C_\mathbb{X}(r_0-1)^\perp$. From
Eq.~\eqref{sep5-23-1}, one has
$$
\langle \beta\cdot\gamma,(1,\ldots,1)\rangle=\langle
\beta'\cdot\gamma,(1,\ldots,1)\rangle=0\ \ \forall \gamma\in
C_\mathbb{X}(r_0-1),
$$
that is, $\langle\beta,\gamma\rangle=\langle\beta',\gamma\rangle=0$ for 
all $\gamma\in C_\mathbb{X}(r_0-1)$. Then, $\beta$ and $\beta'$ are in
$C_\mathbb{X}(r_0-1)^\perp$.
\end{proof}

\begin{proposition}\label{jul6-23} Let $I=I(\mathbb{X})\subset S$ be the vanishing
ideal of $\mathbb{X}=\{[P_1],\ldots,[P_m]\}$ and let $\mathfrak{p}_i$
be the vanishing ideal of $[P_i]$. The following hold.
\begin{enumerate}
\item[(a)] If $S/I$ is a level ring and $f_i$ is the $i$-th standard
indicator function of $\mathbb{X}$, then 
$$\deg(f_i)={\rm v}_{\mathfrak{p}_i}(I)={\rm reg}(S/I).$$
\item[(b)] If $I$ is Gorenstein, then
$H_I(d)+H_I(r_0-d-1)=|\mathbb{X}|$ for all $0\leq d\leq r_0$.  
\end{enumerate}
\end{proposition}

\begin{proof} (a) The first equality follows from
Proposition~\ref{indicator-function-prop}(a). The second equality was
shown in a more general setting for Geramita level rings in \cite[Theorem~4.10]{min-dis-generalized}.

(b) As $S/I$ is a graded Gorenstein algebra, 
its $\mathrm{h}$-vector is symmetric \cite[Theorems~4.1 and 4.2]{Sta1}.
Then, by Proposition~\ref{jul5-23}, the required equality
follows.
\end{proof}

\begin{corollary}\label{duality-criterion-g} 
If $S/I(\mathbb{X})$ is a Gorenstein ring with regularity $r_0$, then 
there is a vector $\beta=(\beta_1,\ldots,\beta_m)\in K^m$ such
that $\beta_i\neq 0$ for all $i$ and
\begin{equation*}
C_\mathbb{X}(d)^\perp=(\beta_1,\ldots,\beta_m)\cdot
C_\mathbb{X}(r_0-d-1)\ \text{ for all \ $0\leq d\leq r_0$}.
\end{equation*}
\quad Moreover, $\beta$ is any vector that defines a parity  
check matrix of $C_\mathbb{X}(r_0-1)$. 
\end{corollary}

\begin{proof} Assume that $S/I(\mathbb{X})$ is a Gorenstein ring. Then, by
Proposition~\ref{jul6-23}, the two conditions 
in Theorem~\ref{duality-criterion}(b) are satisfied and the duality
follows from Theorem~\ref{duality-criterion}(c).  
\end{proof}

\begin{corollary}\cite[Theorem~2]{sarabia7}\label{duality-criterion-original} 
If $I(\mathbb{X})$ is a complete intersection and
$a_\mathbb{X}=r_0-1$ is the $a$-invariant of $S/I(\mathbb{X})$, then
for any integer $d$ such that $0\leq d\leq a_\mathbb{X}$ the linear codes  
$C_\mathbb{X}(d)^\perp$ and 
$C_\mathbb{X}(a_\mathbb{X}-d)$ are monomially equivalent. 
\end{corollary}

\begin{proof} Since complete intersection graded ideals are 
Gorenstein (Proposition~\ref{ci-gor}), the result follows at 
once from Corollary~\ref{duality-criterion-g}.
\end{proof}

\begin{definition}\label{essential-m}\cite{dual} We say a standard monomial
$t^e\in\Delta_\prec(I)$ is {\it essential} if it appears in each
standard indicator function $f_i$ of $\mathbb{X}$. 
\end{definition}

Recall that given a finite set $\Gamma$ of standard
monomials of degree $d$, 
${\rm ev}_d(K\Gamma)$ is the image of $K\Gamma$, under the evaluation map 
${\rm ev}_d$ of Eq.~\eqref{evaluation-map}. 

\begin{theorem}\label{combinatorial-condition} 
Let $\mathbb{X}$ be a subset of $\mathbb{P}^{s-1}$, $m=|\mathbb{X}|\geq 2$, and
let $I=I(\mathbb{X})$ be the vanishing ideal of $\mathbb{X}$. 
If there exists $t^e$ essential monomial,
and if $\beta_i$ is the coefficient of $t^e$ in the $i$-th standard indicator
function~$f_i$ of $\mathbb{X}$, for $i=1,\dots,m$, then for any $\Gamma_1\subset\Delta_\prec(I)_d$,
$\Gamma_2\subset \Delta_\prec(I)_k$
satisfying  
\begin{enumerate}
\item $d+k={\rm reg}(S/I)$,
\item $|\Gamma_1|+|\Gamma_2|=|\mathbb{X}|$, and 
\item $t^e$ does not appear in the remainder on division of $u_1u_2$ by $I$
for every $u_1\in\Gamma_1$, $u_2\in\Gamma_2$,
\end{enumerate}
we have 
$\gamma\cdot{\rm ev}_d(K\Gamma_1)={\rm
ev}_k(K\Gamma_2)^\perp$,
where $\gamma=(\beta_1,\dots,
\beta_m)\cdot(f_1(P_1)^{-1},\ldots,f_m(P_m)^{-1})$.
\end{theorem}

\begin{proof} Let $r_0$ be the regularity of $S/I$ and let 
$\Delta_\prec(I)_{r_0}=\{t^{a_1},\ldots,t^{a_m}\}$ be the set of
standard monomials of $S/I$ of degree $r_0$. Then, $t^e=t^{a_\ell}$ for
some $1\leq\ell\leq m$. Since $\{t^{a_i}\}_{i=1}^m$ and
$\{f_i\}_{i=1}^m$ are $K$-bases of $K\Delta_\prec(I)_{r_0}$, there is
an invertible matrix $\Lambda=(\lambda_{i,k})$ of order $m$ with
entries in the field $K$ such that 
\begin{equation}\label{jul28-23-1}
t^{a_i}=\sum_{k=1}^m\lambda_{i,k}f_k,\quad i=1,\ldots,m\ \Rightarrow\
t^{a_i}(P_j)=\lambda_{i,j}f_j(P_j)\quad \forall\ i,j.
\end{equation}
\quad Then,
$\Lambda(f_1,\ldots,f_m)^\top=(t^{a_1},\ldots,t^{a_m})^\top$. Let $\Lambda_i$ be the $i$-th row of $\Lambda$ and let
$\Lambda^{-1}=(\mu_{k,j})$ be the inverse of $\Lambda$. From the
equalities
\begin{equation}\label{jul28-23-2}
f_k=\sum_{j=1}^m\mu_{k,j}t^{a_j}=
\mu_{k,1}t^{a_1}+\cdots+\mu_{k,\ell}t^{a_\ell}+\cdots+\mu_{k,m}t^{a_m},
\quad k=1,\ldots,m,
\end{equation}
we get that $\beta:=(\beta_1,\ldots,\beta_m)=(\mu_{1,\ell},\ldots,\mu_{m,\ell})$, that is,
$\beta$ is the transpose of the $\ell$-th column of $\Lambda^{-1}$.
Then, from the equality $\Lambda\Lambda^{-1}=I_m$, we obtain
\begin{equation}\label{jul28-23-3}
\langle\Lambda_i,\beta\rangle=0\ \mbox{ for }\ i\neq\ell\ \mbox{ and
}\ \langle\Lambda_\ell,\beta\rangle=1. 
\end{equation}
\quad Take $f\in K\Delta_\prec(I)_{r_0}$. We claim that
$\langle{\rm ev}_{r_0}(f),\gamma\rangle=0$ if and only if
the essential monomial $t^{a_\ell}=t^e$ does not appear in $f$.
Writing $f=\sum_{i=1}^m c_it^{a_i}$, $c_i\in K$, from 
Eqs.~\eqref{jul28-23-1} and \eqref{jul28-23-3}, we get 
\begin{align*}
\langle{\rm
ev}_{r_0}(f),\gamma\rangle&=\langle(f(P_1),\ldots,f(P_m)),\gamma\rangle\\
&=\langle((c_1\lambda_{1,1}+\cdots+c_m\lambda_{m,1})f_1(P_1),\ldots,(c_1\lambda_{1,m}+
\cdots+c_m\lambda_{m,m})f_m(P_m)),\gamma\rangle\\
&=\langle c_1\Lambda_1+\cdots+c_m\Lambda_m,\beta\rangle=
c_\ell\langle\Lambda_\ell,\beta\rangle=c_\ell,
\end{align*}
and the claim is proved. Since $\Gamma_1\subset\Delta_\prec(I)_d$ and 
$\Gamma_2\subset \Delta_\prec(I)_k$, it is seen that 
${\rm ev}_d(K\Gamma_1)$ and ${\rm ev}_k(K\Gamma_2)$ have dimension
$|\Gamma_1|$ and $|\Gamma_2|$, respectively. Then, by 
condition (2), it suffices to show the
inclusion $\gamma\cdot{\rm ev}_d(K\Gamma_1)\subset{\rm
ev}_k(K\Gamma_2)^\perp$ because these two vector spaces have 
dimension $|\Gamma_1|$. Take $g\in K\Gamma_1$, $h\in K\Gamma_2$. Then, by
condition (1),  
$\deg(gh)=d+k=r_0$. By the division algorithm \cite[Theorem~3, 
p.~63]{CLO} and condition (3), it
follows that the remainder on division of $gh$ by $I$ is a polynomial $f$ in
$\Delta_\prec(I)_{r_0}$ such that $t^e$ does not appear in $f$. Then,
by the previous claim, one has
$$
0=\langle{\rm ev}_{r_0}(f),\gamma\rangle=\langle{\rm
ev}_{r_0}(gh),\gamma\rangle=\langle\gamma\cdot{\rm ev}_{d}(g),{\rm
ev}_{k}(h)\rangle.
$$
\quad Thus, $\gamma\cdot{\rm ev}_{d}(g)\in{\rm
ev}_k(K\Gamma_2)^\perp$ and the proof is complete.
\end{proof}

\begin{proposition}\label{self-dual}
Let $\mathbb{X}$ be a subset of
$\mathbb{P}^{s-1}$, let 
$I$ be its vanishing ideal, and let $r_0$ be the regularity of $S/I$.
If $I$ is Gorenstein, 
then $C_\mathbb{X}(d)$ is monomially equivalent to
$C_\mathbb{X}(d)^\perp$ for some $1\leq d\leq r_0$ if and only if
$r_0=2d+1$. 
\end{proposition}

\begin{proof} $\Rightarrow$) By Corollary~\ref{duality-criterion-g}, 
$\beta\cdot
C_\mathbb{X}(r_0-d-1)=C_\mathbb{X}(d)^\perp=\beta'\cdot C_\mathbb{X}(d)$, where
$\beta,\beta'\in(K^*)^m$. Then, since monomially equivalent codes have
the same dimension, we get
\begin{equation}\label{self-dual-eq}
H_I(r_0-d-1)=m-H_I(d)=H_I(d).
\end{equation}
\quad Hence, recalling that $H_I(e)<H_I(e+1)$ for $-1\leq e\leq
r_0-1$ 
(Eq.~\eqref{sep18-23}), we get the equality 
$r_0-d-1=d$. Thus, $r_0=2d+1$.

$\Leftarrow$) From the equality $d=(r_0-1)/2$ and
Corollary~\ref{duality-criterion-g}, we get
$$
C_\mathbb{X}(d)^\perp=\beta\cdot C_\mathbb{X}(r_0-d-1)=\beta\cdot
C_\mathbb{X}(d),
$$
for some $\beta\in(K^*)^m$, and the proof is complete. 
\end{proof}

\begin{remark} If $C_\mathbb{X}(d)$ is monomially equivalent to
$C_\mathbb{X}(d)^\perp$ for some $d$, then 
$|\mathbb{X}|=2H_I(d)$. This follows from the second equality 
of Eq.~\eqref{self-dual-eq}.  
\end{remark}

\begin{theorem}\label{self-dual-general}
Let $\mathbb{X}$ be a subset of
$\mathbb{P}^{s-1}$, let 
$I$ be its vanishing ideal, and let $r_0$ be the regularity of $S/I$.
If $I$ is Gorenstein, then $C_\mathbb{X}(d)^\perp=C_\mathbb{X}(d)$ 
for some $1\leq d\leq r_0$ if and only if
$r_0=2d+1$ and the vector $(1,\ldots,1)$ defines a parity check matrix of
$C_\mathbb{X}(2d)$. 
\end{theorem}

\begin{proof} $\Rightarrow$) By Proposition~\ref{self-dual}, $r_0=2d+1$.
We set $\beta=(\beta_1,\ldots,\beta_m)=(1,\ldots,1)$. 
By assumption, $C_\mathbb{X}(d)^\perp=\beta\cdot C_\mathbb{X}(d)$. 
To show the equality $C_\mathbb{X}(2d)=\{x\mid \langle
x,\beta\rangle=0\}$ it suffices to show the inclusion ``$\subset$''
because these vector spaces have dimension $m-1$
(Proposition~\ref{jul6-23}). Take $x\in C_\mathbb{X}(2d)$. Then,
$x=(g(P_1),\ldots,g(P_m))$ for some $g\in S_{2d}$. We can write 
$g=\sum_{i=1}^k\lambda_it^{a_i}t^{c_i}$, with $\lambda_i\in K$ and 
$t^{a_i}$, $t^{c_i}\in S_d$ for all $i$. Then
\begin{align*}
\langle x,\beta\rangle&=\langle(g(P_1),\ldots,g(P_m)),\beta\rangle
=\left\langle\left(\left(\sum_{i=1}^k\lambda_it^{a_i}t^{c_i}\right)(P_1),\ldots,
\left(\sum_{i=1}^k\lambda_it^{a_i}t^{c_i}\right)(P_m)\right),\beta
\right\rangle\\
&=\beta_1\left(\sum_{i=1}^k\lambda_it^{a_i}t^{c_i}\right)(P_1)+\cdots+
\beta_m\left(\sum_{i=1}^k\lambda_it^{a_i}t^{c_i}\right)(P_m)\\
&=\lambda_1\left(\sum_{i=1}^m\beta_it^{a_1}(P_i)t^{c_1}(P_i)\right)+\cdots+
\lambda_k\left(\sum_{i=1}^m\beta_it^{a_k}(P_i)t^{c_k}(P_i)\right)\\
&=\sum_{i=1}^k\lambda_i\left\langle\beta\cdot(t^{a_i}(P_1),\ldots,t^{a_i}(P_m)),
(t^{c_i}(P_1),\ldots,t^{c_i}(P_m))\right\rangle=0,
\end{align*}
where the last equality follows recalling that 
$(t^{a_i}(P_1),\ldots,t^{a_i}(P_m))$,
$(t^{c_i}(P_1),\ldots,t^{c_i}(P_m))$ are in 
$C_\mathbb{X}(d)$ and the equality $C_\mathbb{X}(d)^\perp=\beta\cdot
C_\mathbb{X}(d)$. Thus, $\langle x,\beta\rangle=0$ and the proof is
complete. 

$\Leftarrow$) Since $r_0 = 2d + 1$, by Corollary~\ref{duality-criterion-g},
$C_\mathbb{X}(d)^\perp=\beta\cdot
C_\mathbb{X}(d)$, and $\beta$ is
a parity check matrix of $C_\mathbb{X}(2d)$. By assumption, we then 
have $\beta=c(1,\ldots,1)$, for some constant $c\neq 0$. 
But this implies $C_\mathbb{X}(d)^\perp=
C_\mathbb{X}(d)$.
\end{proof}

\begin{remark} Affine Cartesian codes \cite{cartesian-codes} are 
projective Reed--Muller type codes and their vanishing ideals are
complete intersections. The self dual Cartesian codes can be 
studied using Theorem~\ref{self-dual-general} and 
Proposition~\ref{gorenstein-essential}.
\end{remark}

\begin{lemma}\label{self-dual-lemma}
Let $P$ be the matrix with rows $P_1,\ldots,P_m$ and let 
$C_1,\ldots,C_s$ be the column vectors of $P$. If $t_s(P_i)=1$ for all $i$ and 
$C_\mathbb{X}(d)^\perp=C_\mathbb{X}(d)$ for some 
$0\leq d\leq r_0$, then 
$\langle C_i,C_j\rangle=0$ for all $1\leq i\leq j\leq s$ and $m\equiv
0\ {\rm mod} (p)$, $p={\rm char}(K)$.  
\end{lemma}

\begin{proof} The set of monomials $\{t_it_s^{d-1}\}_{i=1}^s$ is
contained in $S_d$. Then
$$
C_i=((t_it_s^{d-1})(P_1),\ldots,(t_it_s^{d-1})(P_m))\in
C_\mathbb{X}(d)\mbox{ for }i=1,\ldots,s.
$$
\quad Thus, $C_i\in C_\mathbb{X}(d)^\perp$ for $i=1,\ldots,s$, and 
$\langle C_i,C_j\rangle=0$ for all $1\leq i\leq j\leq s$. Since 
$C_s$ is equal to $(1,\ldots,1)$, we get $m\cdot 1=0$ and $m\equiv
0\ {\rm mod} (p)$.
\end{proof}

\begin{proposition}\label{self-dual-1}
Let $\mathbb{X}$ be a subset of
$\mathbb{P}^{s-1}$, let 
$I$ be its vanishing ideal, and let $C_1,\ldots,C_s$ be the columns of
the matrix $P$ with rows $P_1,\ldots,P_m$. Suppose $t_s(P_i)=1$ for all
$i$ and $I$ contains no linear form. The following hold.
\begin{enumerate} 
\item $\{C_i\}_{i=1}^s$ is linearly independent over $K$ 
and $C_\mathbb{X}(1)=K\{C_i\}_{i=1}^s$.
\item $C_\mathbb{X}(1)^\perp=C_\mathbb{X}(1)$ if and only if $m=2s$ and 
$\langle C_i,C_j\rangle=0$ for all $1\leq i\leq j\leq s$.
\end{enumerate}
\end{proposition}

\begin{proof} (1) Let $\prec$ be the GRevLex order. As $t_i\notin
{\rm in}_\prec(I)$ for all $i$, one has
$\Delta_\prec(I)_{1}=\{t_i\}_{i=1}^s$. Hence, 
$\{{\rm ev}_1(t_i)\}_{i=1}^s$ is linearly independent over
$K$, and  ${\rm ev}_1(t_i)=C_i$ is the $i$-th column of $P$ for
all $i$. Thus, since $C_\mathbb{X}(1)={\rm ev}_1(S_1)$,
we get $C_\mathbb{X}(1)=K\{C_i\}_{i=1}^s$.

(2) $\Rightarrow$) Note that $H_I(1)$ is equal to $s$ because $I$ contains no linear
form. Hence, by the self duality equality
$C_\mathbb{X}(1)^\perp=C_\mathbb{X}(1)$ and Lemma~\ref{self-dual-lemma},
we have $m=2H_I(1)=2s$ and $\langle C_i,C_j\rangle=0$ for all $1\leq i\leq j\leq s$.

$\Leftarrow$) By part (1), $C_\mathbb{X}(1)=K\{C_i\}_{i=1}^s$ and
$s=\dim_K(C_\mathbb{X}(1))=H_I(1)$. Then, by hypothesis, $C_\mathbb{X}(1)\subset
C_\mathbb{X}(1)^\perp$. To show equality notice that $C_\mathbb{X}(1)$
and $C_\mathbb{X}(1)^\perp$ have dimension $s$ because by hypothesis 
$s=m-s$.
\end{proof}

\begin{proposition}\label{selfo-char} If $0\leq d\leq r_0$ and 
 $\beta=(1,\ldots,1)$, then $C_\mathbb{X}(d)\subset
 C_\mathbb{X}(d)^\perp$ if and only if 
\begin{equation}\label{mar6-24}
C_\mathbb{X}(2d)\subset\{x\mid \langle x,\beta\rangle=0\}.
\end{equation}
\quad Furthermore, if $C_\mathbb{X}(d)$ is self orthogonal, then
$2d<r_0$.
\end{proposition}

\begin{proof} $\Rightarrow$) This implication follows from the proof of
Theorem~\ref{self-dual-general}.

$\Leftarrow$) Take $x\in C_\mathbb{X}(d)$, $y\in C_\mathbb{X}(d)$.
Then, $x=(g(P_1),\ldots,g(P_m))$, $y=(h(P_1),\ldots,h(P_m))$ for some
$g,h\in S_d$. Thus, 
$$
x\cdot y=(g(P_1),\ldots,g(P_m))\cdot(h(P_1),\ldots,h(P_m))
=((gh)(P_1),\ldots,(gh)(P_m))\in C_\mathbb{X}(2d).
$$
\quad Hence, $\langle x\cdot y,\beta\rangle=\langle x,
y\rangle=0$, and $x\in C_\mathbb{X}(d)^\perp$.

Now, assume that $C_\mathbb{X}(d)$ is self orthogonal. 
Then, by Eq.~\eqref{mar6-24}, $H_I(2d)\leq m-1$ because 
$\dim_K(C_\mathbb{X}(2d))$ is equal to $H_{I}(2d)$ and the dimension
of $\{x\mid \langle 
x,\beta\rangle=0\}$ is $m-1$. To show that $2d<r_0$ we argue by
contradiction assuming that  
$2d\geq r_0$. Then, by Eq.~\eqref{sep18-23}, one has $H_I(2d)=H_I(r_0)=m$, a 
contradiction.
\end{proof}

\begin{corollary}\label{selfd-char} If $0\leq d\leq r_0$ and 
 $\beta=(1,\ldots,1)$, then $C_\mathbb{X}(d)=C_\mathbb{X}(d)^\perp$ if and only if 
$$C_\mathbb{X}(2d)\subset\{x\mid \langle x,\beta\rangle=0\}\ \mbox{ and
}\ |\mathbb{X}|=2H_I(d). $$
\end{corollary}

\begin{proof} This follows readily from
Proposition~\ref{selfo-char}. 
\end{proof}

\section{Duality of projective evaluation codes and 
Gorenstein ideals}\label{section-gorenstein}

To avoid repetitions, we continue to employ 
the notations and definitions used in
Section~\ref{section-prelim}. The
following lemma relates the regularity and 
the socle of Artinian rings (Eq.~\eqref{socle-def}). 

\begin{lemma}{\rm(\cite[Proposition~4.14]{eisenbud-syzygies}, \cite{Fro3}, \cite[Lemma
5.3.3]{monalg-rev})}\label{reg-socle} Let  $J$ be a homogeneous ideal
in a polynomial ring $S$. Suppose $S/J$ is Artinian. Then,

{\rm(a)} The regularity ${\rm reg}(S/J)$ of $S/J$ is the largest $d$ such that
$(S/J)_d\neq(0)$.

{\rm(b}) The ring $S/J$ is level if and only if there is
$d\geq 1$ such that ${\rm Soc}(S/J)$ is generated by homogeneous
elements of degree $d$.
\end{lemma}

\begin{theorem}\label{socle-gorenstein} 
Let $S/I$ be a graded ring with regularity $r_0$ and let $S/J$
be an Artinian reduction of $S/I$. The following hold.
\begin{enumerate} 
\item[(a)] If $S/I$ is a Gorenstein ring, then
$\Delta_\prec(J)_{r_0}=\{t^a\}$ and 
${\rm Soc}(S/J)=K(t^a+J)$.
\item[(b)] If $S/I$ is a level ring with
symmetric $\mathrm{h}$-vector, then $S/I$ is Gorenstein.
\end{enumerate}
\end{theorem}

\begin{proof} (a) As $S/I$ is a Cohen--Macaulay standard graded
$K$-algebra (Lemma~\ref{aug1-23}) 
its $\mathrm{h}$-vector has
the form $\mathrm{h}(S/I)=(h_0,\ldots,h_{r_0})$ 
\cite[Theorem~6.4.1]{monalg-rev}. A result of Stanley shows that the 
$\mathrm{h}$-vector of $S/I$ is symmetric \cite[Theorems~4.1 and 4.2]{Sta1} 
and in particular we have $h_{r_0}=h_0=1$. By \cite{Sta1},
\cite[Theorem~5.2.5]{monalg-rev}, one has
$\mathrm{h}(S/I)=\mathrm{h}(S/J)$. Then, $h_d=\dim_K(S/J)_d$ for all 
$0\leq d\leq r_0$ and $\dim_K(S/J)_{r_0+1}=0$. Therefore, 
$\Delta_\prec(J)_{r_0}=\{t^a\}$ and $t^a+J$ is a non-zero element of
the socle of $S/J$. Then, as $S/I$ is Gorenstein, the socle of $S/J$
has dimension $1$ as a $K$-vector space
\cite[Corollary~5.3.5]{monalg-rev}, and consequently ${\rm
Soc}(S/J)=K(t^a+J)$.

(b) By the proof of part (a), one has that
$\Delta_\prec(J)_{r_0}=\{t^a\}$ and 
$t^a+J$ is a non-zero element of the socle of $S/J$. Thus, as $S/I$ is
level, the socle of $S/J$ is generated by homogeneous elements of
degree $r_0$ (cf. Lemma~\ref{reg-socle}(b)). It suffices to show that ${\rm
Soc}(S/J)\subset K(t^a+J)$. Take any non-zero
homogeneous element $f+J$ in the socle of $S/J$ of degree $r_0$. The
remainder $r_f$ on division of $f$ by $J$ is a standard polynomial of
$S/J$ of degree $r_0$. Hence, $r_f=\lambda t^a$, $\lambda\in K^*$, and
$f+J=r_f+J=\lambda(t^a+J)$. Thus, $f+J\in K(t^a+J)$ and the proof is
complete.
\end{proof}

\begin{corollary}\label{socle-gorenstein-coro} 
Let $F=\{f_1,\ldots,f_m\}$ be the set of standard indicator functions
of $\mathbb{X}$ and let $r_0$ be the regularity of $S/I$. If $I$ is
Gorenstein, then $\Delta_\prec(J)_{r_0}=\{t^a\}$ and for each $i$, $\deg(f_i)=r_0$,   
$$
{\rm Soc}(S/J)=K(f_i+J)=K(t^a+J),
$$
and the remainder on division of $f_i$ by $J$ has the form
$r_{f_i}=\lambda_i t^a$ for some $\lambda_i\in K^*$.
\end{corollary}

\begin{proof} By Theorem~\ref{socle-gorenstein}, $\Delta_\prec(J)_{r_0}=\{t^a\}$ and 
${\rm Soc}(S/J)=K(t^a+J)$. Any Gorenstein ring is a level ring and, by
Proposition~\ref{jul6-23}, $\deg(f_i)={\rm v}_{\mathfrak{p}_i}(I)=r_0$
for all $i$. We claim that $f_i\notin J$ for all
$i$. We argue by contradiction assuming that $f_i\in J$ for some $i$.
Writing $f_i=G_i+hH_i$, where $G_i\in I_{r_0}$, $H_i\in S_{r_0-1}$,
and noticing that $h(P_j)\neq 0$ for all $j$ because $h$ is 
regular on $S/I$ (see Remark~\ref{mar17-24}), we obtain that $H_i$ is
an indicator function for 
$[P_i]$ of degree $r_0-1$, a contradiction to Lemma~\ref{if}(b). This
proves the claim. Then, dividing $f_i$ by $J$, gives a non-zero 
remainder $r_{f_i}$ which is a standard polynomial of $S/J$ of degree
$r_0$. Thus, $r_{f_i}=\lambda_i t^a$, $\lambda_i\in K^*$, and
$f_i+J=\lambda_i(t^a+J)$.
\end{proof}

\begin{proposition}\label{aug6-23} Let $\prec$ be the GRevLex order on
$S$. Suppose $t_s$ is regular on $S/I$. The following hold. {\rm (a)}
If $t^a\in\Delta_\prec(I)$, then
$t_s^\ell t^a\in\Delta_\prec(I)$ for all $\ell\geq 0$.

{\rm (b)} If $\mathcal{G}=\{g_1,\ldots,g_n\}$ is the reduced Gr\"obner
basis of $I$, then $t_s$ does not divide ${\rm
in}_\prec(g_i)$ for all $i$, and $\mathcal{G}\cup\{t_s\}$ is a Gr\"obner basis of
the ideal $J=(I,t_s)$.

{\rm (c)} $\Delta_\prec(I)\setminus\{t^c\mid t_s\in{\rm
supp}(t^c)\}=\Delta_\prec(J)$.
\end{proposition}

\begin{proof} (a) Let $\mathcal{G}=\{g_1,\ldots,g_n\}$ be a Gr\"obner
basis of $I$. We proceed  by induction on $\ell\geq 0$. If
$\ell=0$, the assertion is clear. Assume $\ell\geq 1$ and 
$t_s^k t^a\in\Delta_\prec(I)$ for all $0\leq k<\ell$. We argue by
contradiction assuming that $t_s^\ell t^a\notin\Delta_\prec(I)$. 
Then, $t_s^\ell t^a\in{\rm in}_\prec(I)$ and $t_s^\ell t^a=t^c{\rm
in}_\prec(g_i)$ for some $i$ and some $t^c$. Note that $t_s$ divides 
${\rm in}_\prec(g_i)$, otherwise $t_s$ divides $t^c$. If the latter, then
$$
t_s^{\ell-1}t^a=(t^c/t_s){\rm in}_\prec(g_i)\in{\rm
in}_\prec(I),
$$
a contradiction. Thus, since $\prec$ is the GRevLex order, we get
that $t_s$ divides each term of $g_i$, that is, $t_s$ divides $g_i$.
Hence, we can write $g_i=t_sG_i$ for some $G_i\in S$. By the 
regularity of $t_s$, one has $G_i\in I$. Then
$$ 
t_s^\ell t^a=t^c{\rm in}_\prec(g_i)=t^ct_s{\rm
in}_\prec(G_i)\Rightarrow 
t_s^{\ell-1} t^a=t^c{\rm
in}_\prec(G_i)\in{\rm in}_\prec(I),
$$
a contradiction. Thus, $t_s^\ell t^a\in\Delta_\prec(I)$ and the
induction process is complete.

(b) To show that $t_s$ does not divides ${\rm in}_\prec(g_i)$ for all
$i$, we argue by contradiction assuming that $t_s$ divides ${\rm
in}_\prec(g_i)$ for some $i$. Then, by the proof of part (a),
we obtain that $g_i=t_sG_i$ for some $G_i\in I$. Thus, looking at
leading terms, 
one has
$$ 
{\rm in}_\prec(g_i)=t_s{\rm
in}_\prec(G_i)=t_s t^b{\rm in}_\prec(g_j),
$$
for some $j\neq i$ and some $t^b$, a contradiction because by the
reducedness of $\mathcal{G}$ no term of $g_i$ is divisible by ${\rm
in}_\prec(g_j)$. This proves the first part. From this, ${\rm
in}_\prec(g_i)$ is relatively prime to $t_s$ for all $i$, and by
\cite[Lemma~3.3.16]{monalg-rev} and Buchberger's criterion \cite{Buch},
it follows readily that $\mathcal{G}\cup\{t_s\}$ is a Gr\"obner basis
of the ideal $J=(I,t_s)$.

(c) By part (b), one has ${\rm in}_\prec(I)=(\{{\rm
in}_\prec(g_i)\}_{i=1}^n)$ and ${\rm in}_\prec(J)=(\{{\rm
in}_\prec(g_i)\}_{i=1}^n,t_s)$. Then, 
$t^a\in\Delta_\prec(I)\setminus\{t^c\mid t_s\in{\rm
supp}(t^c)\}$ if and only if $t^a\notin{\rm in}_\prec(I)$ and
$t_s\notin{\rm supp}(t^a)$. The latter conditions hold if and only if
$t^a\notin{\rm in}_\prec(J)$, that is, if and only if
$t^a\in\Delta_\prec(J)$.
\end{proof}

\begin{proposition}\label{gorenstein-essential} Let $\prec$ be the GRevLex order on $S$, let
$F=\{f_1,\ldots,f_m\}$ be the set of standard indicator functions 
of $\mathbb{X}$, let ${\rm lc}(f_i)$ be the leading coefficient of
$f_i$, and let $r_0$ be the regularity of $I$. Suppose $t_s$
is regular on $S/I$ and $I$ is Gorenstein. The following hold.
\begin{enumerate}
\item[(a)] There is
$t^a\in\Delta_\prec(I)_{r_0}$ such that $t_s\notin{\rm supp}(t^a)$ and
every $f_i$ has the form $f_i=\lambda_it^a+t_sG_i$, 
for some $\lambda_i\in K^*$, $G_i\in S_{r_0-1}$. In
particular, $t^a$ is essential and $\lambda_i={\rm lc}(f_i)$.
\item[(b)] If $t_s(P_i)=1$ for all $i$, then a parity check 
matrix $H$ of $C_\mathbb{X}(r_0-1)$ is the $1\times m$ matrix 
$H=({\rm lc}(f_1)f_1(P_1)^{-1},\ldots,{\rm lc}(f_m)f_m(P_m)^{-1})$. 
\end{enumerate}
\end{proposition}

\begin{proof} (a) We set $J=(I,t_s)$. 
By Corollary~\ref{socle-gorenstein-coro}, one has
$\Delta_\prec(J)_{r_0}=\{t^a\}$ and for each $i$, $\deg(f_i)$ is equal
to $r_0$ and
$f_i\notin J$. Note that $t_s\notin{\rm supp}(t^a)$ because $t^a$ is a
standard monomial of $S/J$. By Proposition~\ref{aug6-23}(c), any
monomial $t^c$ that appears in $f_i$ and does not contains $t_s$ in
its support is equal to $t^a$. Then, as $f_i\notin J$, we can write
$f_i=\lambda_it^a+t_sG_i$, where $G_i\in S_{r_0-1}$ and 
$\lambda_i\in K^*$.

(b)  To prove that $H$ is a parity check matrix of
$C_\mathbb{X}(r_0-1)$ we show the equality
\begin{equation}\label{sep19-23}
C_\mathbb{X}(r_0-1)=\{x\in K^m\mid Hx=0\}.
\end{equation}
\quad By Proposition~\ref{jul6-23}(b), $\dim_K(C_\mathbb{X}(r_0-1))=H_I(r_0-1)=m-1$. Hence, to show the
equality of Eq.~\eqref{sep19-23}, it suffices to show the inclusion ``$\subset$''. Take $x\in
C_\mathbb{X}(r_0-1)$, that is, $x={\rm ev}_{r_0-1}(f)$ for some $f\in
S_{r_0-1}$. By considering the remainder on division of $f$ by $I$,
we may assume that 
$f\in K\Delta_\prec(I)_{r_0-1}$. Then, by Proposition~\ref{aug6-23}(a)
and recalling that $\{f_i\}_{i=1}^m$ is a $K$-basis of $K\Delta_\prec(I)_{r_0}$,
one has $t_sf\in K\Delta_\prec(I)_{r_0}$ and there are $\mu_1,\ldots,\mu_m$ in
$K$ such
that
\begin{align}
t_sf&=\mu_1f_1+\cdots+\mu_mf_m\label{mar7-24}\\
&=\mu_1(\lambda_1t^a+t_sG_1)+\cdots+\mu_m(\lambda_mt^a+t_sG_m)\nonumber\\
&=(\mu_1\lambda_1+\cdots+\mu_m\lambda_m)t^a+t_s(\mu_1G_1+\cdots+\mu_mG_m).
\nonumber
\end{align}  
\quad Thus, making $t_s=0$ and recalling $t_s\notin{\rm supp}(t^a)$, we get
$\mu_1\lambda_1+\cdots+\mu_m\lambda_m=0$. Then
\begin{align*}
\langle x,H\rangle&\ =\ \langle{\rm ev}_{r_0-1}(f),H\rangle=\langle{\rm
ev}_{r_0}(t_sf),H\rangle=\langle ((t_sf)(P_1),\ldots,(t_sf)(P_m)),H\rangle\\
&\stackrel{\eqref{mar7-24}}{=}\langle(\mu_1f_1(P_1),\ldots,\mu_mf_m(P_m)),H\rangle=
\mu_1\lambda_1+\cdots+\mu_m\lambda_m=0.
\end{align*}  
\quad Hence, $\langle x,H\rangle=0$, that is, $Hx=0$ and the proof of the
inclusion ``$\subset$'' is
complete.
\end{proof}

\begin{corollary}\label{combinatorial-condition-projective} 
Let $\prec$ be the GRevLex order. Let $\mathbb{X}$ be a subset of
$\mathbb{P}^{s-1}$ and let 
$I$ be its vanishing ideal. Suppose
$t_s(P_i)=1$ for all $i$ and $I$ is Gorenstein. Then, there is a
unique essential monomial $t^a$ such that $t_s\notin{\rm supp}(t^a)$
and for any $\Gamma_1\subset\Delta_\prec(I)_d$,
$\Gamma_2\subset \Delta_\prec(I)_k$
satisfying  
\begin{enumerate}
\item $d+k\leq {\rm reg}(S/I)$,
\item $|\Gamma_1|+|\Gamma_2|=|\mathbb{X}|$,
\item $t^a$ does not appear in the remainder on division of $u_1u_2$
by $I$
for every $u_1\in\Gamma_1$, $u_2\in\Gamma_2$,
\end{enumerate}
we have 
$\gamma\cdot{\rm ev}_d(K\Gamma_1)={\rm
ev}_k(K\Gamma_2)^\perp$,
where $\gamma=(\beta_1,\dots,
\beta_m)\cdot(f_1(P_1)^{-1},\ldots,f_m(P_m)^{-1})$ and $\beta_i$ is the
coefficient of $t^a$ in the $i$-th standard indicator
function~$f_i$ of $\mathbb{X}$. 
\end{corollary}

\begin{proof}  Let $r_0$ be the regularity of $S/I$. 
As $I$ is Gorenstein, by
Corollary~\ref{socle-gorenstein-coro} 
and Proposition~\ref{gorenstein-essential}, there exists a unique
$t^a\in\Delta_\prec(I)_{r_0}$ such that $t^a$ is essential and
$t_s\notin{\rm supp}(t^a)$. If $d+k=r_0$, the result follows 
from Theorem~\ref{combinatorial-condition}. Assume that $d+k<r_0$. We 
set $\epsilon=r_0-d-k$ and $t_s^\epsilon\Gamma_1=\{t_s^\epsilon u\mid
u\in\Gamma_1\}$. By Proposition~\ref{aug6-23},
$t_s^\epsilon\Gamma_1\subset\Delta_\prec(I)_{r_0-k}$. Take any $t_s^\epsilon u_1\in t_s^\epsilon\Gamma_1$,
$u_2\in\Gamma_2$. Then, 
$$
u_1u_2=h+r_{u_1u_2},
$$
where $h\in I_{d+k}$ and $r_{u_1u_2}\in K\Delta_\prec(I)_{d+k}$ is the remainder on division of
$u_1u_2$ by $I$. Then,
$$
t_s^\epsilon u_1u_2=t_s^\epsilon h+t_s^\epsilon r_{u_1u_2}
$$ 
and, since $t_s\notin I$, $t_s$ is regular
on $S/I$, by Proposition~\ref{aug6-23}(a),
$t_s^{\epsilon}r_{u_1u_2}\in K\Delta_\prec(I)_{r_0}$. Note that 
$\deg(t^a) = r_0 > d + k =
\deg(r_{u_1u_2})$, so $t^a$ cannot appear in $r_{u_1u_2}$. Thus, by the
uniqueness of the remainder \cite[p.~81]{CLO}, one has 
$r_{t_s^\epsilon u_1u_2}=t_s^{\epsilon}r_{u_1u_2}$, and $t^a$ does not
appear in $r_{t_s^\epsilon u_1u_2}$ because $\epsilon\geq 1$. Hence, applying
Theorem~\ref{combinatorial-condition} to $t_s^\epsilon\Gamma_1$ and
$\Gamma_2$, we obtain
\begin{equation}\label{sep18-23-1}
\gamma\cdot{\rm ev}_{r_0-k}(K(t_s^\epsilon\Gamma_1))={\rm
ev}_k(K\Gamma_2)^\perp,
\end{equation}
where $\beta_i$ is the
coefficient of $t^a$ in the $i$-th standard indicator
function~$f_i$ of $\mathbb{X}$. Since we are assuming that $t_s(P_i)=1$ for $i=1,\ldots,m$,
one has
$$
{\rm ev}_{r_0-k}(K(t_s^\epsilon\Gamma_1))={\rm
ev}_d(K\Gamma_1),
$$
and by Eq.~\eqref{sep18-23-1} the proof is complete.
\end{proof}

The following theorem was shown by Kreuzer \cite[Corollary~2.5(b)]{Kreuzer} 
when $K$ is algebraically closed. We show how to 
use Kreuzer's result to prove the theorem for finite fields.  

\begin{theorem}\label{conjecture-gor} The vanishing ideal $I$ of
$\mathbb{X}$ is
Gorenstein if and only if
\begin{equation}\label{apr1-24}
H_I(d)+H_I(r_0-d-1)= |\mathbb{X}|\mbox{ for all }0\leq d\leq
r_0\mbox{ and }r_0={\rm v}_{\mathfrak{p}}(I)\mbox{ for all } \mathfrak{p}\in{\rm Ass}(I).
\end{equation}
\end{theorem}

\begin{proof} 
$\Rightarrow$) Assume that $I$ is Gorenstein. Then, $I$ is level and,
by Proposition~\ref{jul6-23}, the two conditions of
Eq.~\eqref{apr1-24} hold. 

$\Leftarrow$) Assume that the two conditions of
Eq.~\eqref{apr1-24} hold. Let $\overline{K}$ be the algebraic closure
of $K$ \cite[p.~11]{AM}.
We set
$\overline{S}:=S\otimes_K\overline{K}=\overline{K}[t_1,\ldots,t_s]$
and $\overline{I}:=I\overline{S}$. Note that $K\hookrightarrow
\overline{K}$ is a faithfully flat extension.  
Apply the functor $S\otimes_{K}(-)$. By base change, it follows that
$S\hookrightarrow \overline{S}$ is a faithfully flat extension. 
The ideal theory of $S$ is related to that of $\overline{S}$, 
see \cite[Sections~3.H, 5.D and 9.C]{Mat}. 
From \cite[Lemma~1.1]{Sta1}, 
$H_I(d)=H_{\overline{I}}(d)$ for all $d\geq 0$ and $I$ is Gorenstein
if and only if $\overline{I}$ is Gorenstein. In particular, $r_0={\rm
reg}(S/I)={\rm reg}(\overline{S}/\overline{I})$ because ${\rm
reg}(S/I)$ is the regularity index of $H_I$
\cite[p.~256]{hilbert-min-dis}, and one has  
\begin{equation}\label{feb29-24}
H_{\overline{I}}(d)+H_{\overline{I}}(r_0-d-1)= |\mathbb{X}|\mbox{ for all }0\leq d\leq
r_0={\rm reg}(\overline{S}/\overline{I}).
\end{equation}
\quad Let ${\mathfrak{p}}_i$ be the vanishing ideal of
$[P_i]$ for $i=1,\ldots,m$. The prime ideal
$\overline{\mathfrak{p}}_i:={\mathfrak{p}}_i\overline{S}$ is the
vanishing ideal of $[P_i]$ in $\overline{S}$ (Lemma~\ref{primdec-ix-a}).
The
minimal primary decomposition of $\overline{I}$ is 
$$
\overline{I}=I\overline{S}=\Bigg(\bigcap_{i=1}^m\mathfrak{p}_i\Bigg)\overline{S}=
\bigcap_{i=1}^m\big(\mathfrak{p}_i\overline{S}\big)=
\bigcap_{i=1}^m\overline{\mathfrak{p}}_i,
$$
and $\overline{I}\cap S=I$. Thus, $\overline{I}$ is the 
vanishing ideal of $\mathbb{X}$ in $\overline{S}$. By 
Lemma~\ref{if}, ${\rm
v}_{\overline{\mathfrak{p}}_i}(\overline{I})\leq{\rm
reg}(\overline{S}/\overline{I})=r_0$ for all $i$. Let
$g_i\in\overline{S}$ be the
$i$-th standard indicator function for $[P_i]$. Then, by Lemma~\ref{if}, 
$g_i\in(\overline{I}\colon\overline{\mathfrak{p}}_i)\setminus\overline{I}$.
Let $\{h_1,\ldots,h_k\}$ be a generating set of 
$(I\colon\mathfrak{p}_i)\subset S$ consisting of homogeneous
polynomials of $S$. Since
$(I\colon\mathfrak{p}_i)\overline{S}=(I\overline{S}
\colon\mathfrak{p}_i\overline{S})$, we can write 
$g_i=\sum_{i=1}^kh_ia_i$, where $a_j$ is a homogeneous polynomial in 
$\overline{S}$ for all $j$ and $\deg(g_i)=\deg(h_ja_j)$ if $h_ja_j\neq
0$. 
Note that $h_j\notin I$ for some $1\leq
j\leq k$ such that $h_ja_j\neq 0$, otherwise $g_i\in
I\overline{S}=\overline{I}$,
a contradiction because $g_i(P_i)\neq 0$. Thus $h_j\in (I\colon\mathfrak{p}_i)\setminus I$
and, by 
Lemma~\ref{if}(a)-(b), $h_j$ is an indicator function for $[P_i]$ and ${\rm
v}_{\mathfrak{p}_i}(I)\leq\deg(h_j)$. Therefore, one has 
\begin{equation}\label{feb29-24-1}
r_0={\rm v}_{\mathfrak{p}_i}(I)\leq\deg(h_j)\leq\deg(g_i)={\rm
v}_{\overline{\mathfrak{p}}_i}(\overline{I})\leq r_0={\rm
reg}(\overline{S}/\overline{I}),
\end{equation}
and we have equality everywhere. As $\overline{K}$ is algebraically 
closed, using Eqs.~\eqref{feb29-24}--\eqref{feb29-24-1}, we obtain that 
the triplet ($\mathbb{X}$, $\overline{S}$, $\overline{I})$ satisfies 
the two conditions of 
\cite[Corollary~2.5(b)]{Kreuzer} (cf.~\cite[Theorem~5]{Davis-etal}), 
and consequently $\overline{I}$ is
Gorenstein, and so is $I$. 
\end{proof}

\section{Affine duality criterion}\label{section-affine}
In this section we show how to recover the global duality
criterion of L\'opez, Soprunov and Villarreal \cite{dual} for affine
Reed--Muller-type codes using our projective global duality criterion. 

Let $X=\{Q_1,\ldots,Q_m\}$ be a subset of the affine space
$\mathbb{A}^s=K^s$, $m=|X|\geq 2$, and let $I(X)$ be its
vanishing ideal. Given an integer
$d\geq 0$, we let $S_{\leq 
d}=\bigoplus_{i=0}^dS_i$ be the $K$-linear subspace of $S$ of all 
polynomials of degree at most $d$ and let $I(X)_{\leq d}=I(X)\bigcap S_{\leq d}$.  
The function
$$
H_X^a(d):=\dim_K(S_{\leq d}/I(X)_{\leq d}),\ \ \ d=0,1,2,\ldots
$$
is  called the \textit{affine Hilbert function} of $S/I(X)$. 
Let $r_0={\rm reg}(H_X^a)$ be the regularity index of $H_X(d)$, that
is, $r_0$ is the 
least integer $\ell\geq 0$ such that $H_X^a(d)=|X|$ for $d\geq\ell$. 

To make the connection between affine and projective Reed--Muller-type
codes \cite{cartesian-codes,affine-codes}, we consider the GRevLex 
order $\prec$ on $S[u]$, where $u$
is a new variable such that $t_1\succ\cdots\succ t_s\succ u$. 
\quad Let $\mathcal{G}=\{g_1,\ldots,g_p\}$ be a Gr\"obner basis of
$I(X)$, let $g_i^h$ be the homogenization of $g_i$ with respect to
$u$, and let $I(X)^h$ be the homogenization of $I(X)$. 
As $K$ is a finite field, the projective closure 
$Y$ of $X$ is given by \cite[p.~133]{monalg-rev}:
$$  
Y=[X,1]=\{[(Q_1,1)],\ldots,[(Q_m,1)]\}\subset\mathbb{P}^s.
$$ 

\begin{proposition}\label{affine-projective} 
{\rm(i)} \cite[Lemma~3.7]{affine-codes}  
$I(Y)=I(X)^h=(g_1^h,\ldots,g_p^h)$ and $\{g_1^h,\ldots,g_p^h\}$ is a
Gr\"obner basis of $I(Y)$. 

{\rm(ii)} \cite[Lemma~2.8, Proposition~2.9]{cartesian-codes} 
$H_X^a(d)=H_Y(d)\mbox{ and }C_X(d)=C_Y(d)$ for $d\geq 0$.
\end{proposition}

\begin{lemma}\label{jul14-23} {\rm(i)} If $H_i$ $($resp. $G_i$$)$ 
is an indicator function of $Q_i$ $($resp. $[(Q_i,1)]$$)$,  
then $H_i^h$ $($resp. $L_i:=G_i(t_1,\ldots,t_s,1)$$)$ is an indicator
function of 
$[(Q_i,1)]$ $($resp. $Q_i$$)$. 

{\rm(ii)} Let $\mathfrak{q}_i$ be the defining
ideal of $Q_i$, and let $\mathfrak{p}_i$ be the defining ideal of
$[(Q_i,1)]$. Then, 
${\rm v}_{\mathfrak{q}_i}(I(X))={\rm
v}_{\mathfrak{p}_i}(I(Y))$. 

{\rm(iii)} $r_0={\rm reg}(H_X^a)={\rm reg}(H_Y)={\rm reg}(S[u]/I(Y))$.

{\rm(iv)} The following conditions are equivalent.

\begin{enumerate} 
\item $H_X^a(d)+H_X^a(r_0-d-1)=|X|$ for $0\leq d\leq r_0$ 
and $r_0={\rm v}_{\mathfrak{q}_i}(I(X))$ for $i=1,\ldots,m$. 
\item $H_Y(d)+H_Y(r_0-d-1)=|Y|$ for $0\leq d\leq r_0$ 
and $r_0={\rm v}_{\mathfrak{p}_i}(I(Y))$ for $i=1,\ldots,m$.
\end{enumerate} 
\end{lemma}

\begin{proof} (i) $H_i^h$ is an indicator function of $[(Q_i,1)]$ because
$H_i^h(Q_j,1)=H_i(Q_j)$ for all $j$, and $L_i$ is an indicator function of $Q_i$
because $L_i(Q_j)=G_i(Q_j,1)$ for all $j$.  

(ii) By \cite[Lemma~4.4]{dual}, ${\rm
v}_{\mathfrak{q}_i}(I(X))$ is the least degree of an indicator
function of the affine point $Q_i$ and, by Lemma~\ref{if}(b), ${\rm
v}_{\mathfrak{p}_i}(I(Y))$ is the least degree of an indicator
function of the projective point $[(Q_i,1)]$. Then, by part (i), 
the equality ${\rm v}_{\mathfrak{q}_i}(I(X))={\rm
v}_{\mathfrak{p}_i}(I(Y))$ follows.   

(iii) It follows from Proposition~\ref{affine-projective}(ii). 

(iv) It follows from the fact that $H_X^a(d)=H_Y(d)$ for $d\geq 0$
(Proposition~\ref{affine-projective}), and using parts (ii) and (iii).
\end{proof}

\begin{corollary}\cite[Theorem~6.5]{dual}{\rm\ (Affine duality
criterion)}\label{duality-criterion-affine} 
Let $X$ be a subset of the affine space $\mathbb{A}^s=K^s$, $m=|X|\geq 2$, let $I(X)$ be its
vanishing ideal, let $r_0$ be the regularity index of $H_X^a$. 
The following conditions are equivalent.
\begin{enumerate} 
\item[(b')] $H_X^a(d)+H_X^a(r_0-d-1)=|X|$ for $0\leq d\leq r_0$ 
and $r_0={\rm v}_{\mathfrak{p}}(I(X))$ for $\mathfrak{p}\in{\rm Ass}(I(X))$.
\item[(c')] There is 
$\beta=(\beta_1,\ldots,\beta_m)\in K^m$ such that $\beta_i\neq 0$ for all $i$ and
$$
C_X(d)^\perp=(\beta_1,\ldots,\beta_m)\cdot C_X(r_0-d-1)\ \text{ for
all
\ $0\leq d\leq r_0$}.
$$
Moreover, $\beta$ is a vector that defines a parity  
check matrix of $C_X(r_0-1)$. 
\end{enumerate} 
\end{corollary}

\begin{proof}  It follows readily using
Proposition~\ref{affine-projective}, Lemma~\ref{jul14-23}, recalling
that $m=|X|=|Y|$, and
applying to $Y$ the 
projective duality criterion of Theorem~\ref{duality-criterion}. 
\end{proof}

\section{The $r$-th v-number}\label{section-v-number}
To avoid repetitions, we continue to employ 
the notations and definitions used in
Sections~\ref{section-introduction} and \ref{section-prelim}. 
In this section  we relate the regularity index of the $r$-th
generalized Hamming weight function 
$\delta_{\mathbb{X}}(\,\cdot\,,r)$ of $I$ and the 
$r$-th v-number of $I$.  

Let $F=\{f_1,\ldots, f_m\}$ be the set of standard indicator
functions of $\mathbb{X}=\{[P_1], \ldots, [P_m]\}$. There is a
permutation $\pi$ of $\{1,\ldots,m\}$ such that
$\deg(f_{\pi(1)}) \leq\cdots\leq\deg(f_{\pi(m)})$. 
The $r$-th
v-\textit{number} of $I=I(\mathbb{X})$, 
denoted by ${\rm{v}}_r(I)$, is given by:
\begin{equation}\label{mar8-24-1}
{\rm v}_r(I):=\deg(f_{\pi(r)})\ \mbox{ for }\ r=1,\ldots,m.
\end{equation}
\quad The $r$-th v-number is well defined, that is, ${\rm{v}}_r(I)$ is independent of $\pi$. 
Recall that $\dim_K(C_\mathbb{X}(d))$ is equal to $H_I(d)$. The \textit{initial degree} of
$\delta_\mathbb{X}(\,\cdot\,, r)$, denoted $\rho_r$, is given by
$$
\rho_r:=\max\{e\mid r>H_I(e)\}.
$$
\quad Note that $\rho_r<d$ if and only if $r\leq H_I(d)$. If $r=1$,
one has $\rho_r=-1$ because $H_I(-1)=0$. The
\textit{regularity index} of the $r$-th generalized Hamming weight 
function $\delta_{\mathbb{X}}(\,\cdot\,,r)$ is given by
\begin{equation}\label{mar8-24-2}
R_r:={\rm{reg}} \, (\delta_{\mathbb{X}}(\,\cdot\,,r)):=
\min \{d  \geq 1\mid \delta_{\mathbb{X}}(d,r)=r\}.
\end{equation}
\quad If $r=1$, recall that ${\rm{reg}} \, (\delta_{\mathbb{X}}(\,\cdot\,,1))$ is
denoted by ${\rm{reg}}(\delta_{\mathbb{X}})$. By \cite[Theorem 5.3]{min-dis-generalized}, one has
\begin{equation}\label{sarabia-baby}
\delta_{\mathbb{X}}(\rho_r+1,r) > \delta_{\mathbb{X}}(\rho_r+2,r)> 
\cdots
>\delta_{\mathbb{X}}(R_r,r)=\delta_{\mathbb{X}}(d,r)=r,\quad\forall\, d \geq R_r.
\end{equation}

The main result of this section shows 
that the regularity index $R_r$ of $\delta_{\mathbb{X}}(\,\cdot\,,r)$,
defined in Eq.~\eqref{mar8-24-2}, is precisely
the $r$-th v-number ${\rm v}_r(I)$ of the vanishing ideal $I$ of
$\mathbb{X}$, defined in Eq.~\eqref{mar8-24-1}.

\begin{theorem}\label{sarabia-vila-2}
Let $F=\{f_1,\ldots,f_m\}$ be the unique set of standard indicator
functions of $\mathbb{X}=\{[P_1],\ldots,[P_m]\}$ and let $R_r$ be the regularity index
of $\delta_\mathbb{X}(\,\cdot\,,r)$. The following hold.
\begin{enumerate}
\item[(a)] $R_r \leq {\rm{reg}}(S/I)$. 
\item[(b)] If $\deg(f_1)\leq\cdots\leq\deg(f_m)$, then
${\rm{v}}_r(I)=\deg(f_r)=R_r$ for $r=1,\ldots,m$.
\item[(c)] Let $\pi$ be a permutation of $\{1,\ldots,m\}$ such that
$\deg(f_{\pi(i)})\leq\deg(f_{\pi(j)})$ for $i<j$. Then, 
${\rm{v}}_r(I):=\deg(f_{\pi(r)})=R_r$ for $r=1,\ldots,m$.
\end{enumerate}
\end{theorem}

\begin{proof} (a) By \cite[Theorem
7.10.1]{Huffman-Pless} and the generalized Singleton bound
\cite[Theorem~7.10.6]{Huffman-Pless}, we know that
$$
r \leq \delta_{\mathbb{X}}(d,r) \leq |\mathbb{X}|-H_I(d)+r\ \mbox{
for all }\ 1\leq
r\leq H_I(d).
$$
\quad If $r_0={\rm{reg}}(S/I)$, then 
$m=|\mathbb{X}|=H_I(r_0)$, 
$\delta_{\mathbb{X}}(r_0,r)=\delta_r({C_\mathbb{X}}(r_0))=\delta_r(K^m)=r$,
and
$$
R_r \leq r_0={\rm{reg}}\, (S/I).
$$
\quad (b)  Since
$\delta_\mathbb{X}(R_r,r)=\delta_r(C_\mathbb{X}(R_r))=r$ and
$\rho_r<R_r$, there is a subspace $D$ of $C_{\mathbb{X}}(R_r)$ of
dimension $r$ such that
$$
\chi(D)=\{\ell_1,\ldots,\ell_r\},\ 1\leq \ell_1<\cdots< \ell_r\leq m.
$$
\quad The support $\chi(\beta)$ of a vector $\beta\in K^m$ 
is $\chi(K\beta)$, that is, $\chi(\beta)$ is the set of non-zero
entries of $\beta$. Let $\beta_{i}=(\beta_{i,1},\ldots,\beta_{i,m})$,
$i=1,\ldots,r$, be a $K$-basis for $D$. According to
\cite[Lemma~2.1]{rth-footprint}, $\chi(D)=\bigcup_{i=1}^r\chi(\beta_i)$
and  the number of elements of
$\chi(D)$ is the number of non-zero columns of the generator matrix:
$$   
B=\left[\begin{matrix}
\beta_{1,1}&\cdots&\beta_{1,i}&\cdots&\beta_{1,m}\\
\beta_{2,1}&\cdots&\beta_{2,i}&\cdots&\beta_{2,m}\\
\vdots&\cdots&\vdots&\cdots&\vdots\\
\beta_{r,1}&\cdots&\beta_{r,i}&\cdots&\beta_{r,m}
\end{matrix}\right].
$$
\quad Thus, the non-zero columns of $B$ are $\ell_1,\ldots,\ell_r$. As
$B$ has rank $r$, the submatrix $B'$ of $B$ with columns
$\ell_1,\ldots\ell_r$ is non-singular. Hence, by applying only
elementary row operations, $B'$ can be reduced to the $r\times r$
identity matrix. Hence, $D$ is generated by the unit vectors 
$e_{\ell_1},\ldots,e_{\ell_r}$ of $K^m$, 
 and consequently there are $g_1,\ldots,g_r$ in $S_{R_r}$ such that
 ${\rm ev}_{R_r}(g_i)=e_{\ell_i}$ for $i=1,\ldots,r$. Thus, $g_i$ is
 an indicator function for $[P_{\ell_i}]$ for $i=1,\ldots,r$. By the
 division algorithm \cite[Theorem~3, 
p.~63]{CLO}, we can write $g_i=h_i+r_{g_i}$, where $h_i\in I_{R_r}$
and $r_{g_i}\in\Delta_\prec(I)_{R_r}$. As $r_{g_i}$ is a standard
indicator function of $[P_{\ell_i}]$, by Lemma~\ref{if}(b) and 
Proposition~\ref{indicator-function-prop}(a), we get 
\begin{equation}\label{jun24-23}
\deg(g_i)=R_r\geq\deg(f_{\ell_i})\mbox{ for }i=1,\ldots,r. 
\end{equation}
\quad For each $[P_i]$, there is $t_{j_i}$, $1\leq j_i\leq s$, such 
that $t_{j_i}(P_i)\neq 0$. Setting $\epsilon_i=\deg(f_r)-\deg(f_i)$,
$i=1,\ldots,r$, the indicator function $t_{j_i}^{\epsilon_i}f_i$ of
$[P_i]$ has degree $\deg(f_r)$ for $i=1,\ldots,r$. Since  
$$\{{\rm ev}_{\deg(f_r)}(t_{j_i}^{\epsilon_i}f_i)\}_{i=1}^r$$ 
generates a linear 
subspace of $C_\mathbb{X}(\deg(f_r))\subset K^m$ of dimension $r$
with support 
$\{1,\ldots,r\}$ (because
$t_{j_i}(P_i)\neq 0$ and $f_i(P_i)\neq 0$), we
get, from the definition of
$r$-generalized Hamming weight, that $\delta_r(C_\mathbb{X}(\deg(f_r)))\leq r$, $1\leq r\leq
H_I(\deg(f_r))$, and $\deg(f_r)\geq\rho_r+1$. Hence, by
Eq.~\eqref{sarabia-baby} and a previous inequality, 
we obtain
$$
r\leq\delta_\mathbb{X}(\deg(f_r),r)=\delta_r(C_\mathbb{X}(\deg(f_r)))\leq
r.
$$
\quad Thus, $\delta_r(C_\mathbb{X}(\deg(f_r)))=r$, and $\deg(f_r)\geq R_r$. Therefore, using
Eq.~\eqref{jun24-23}, one has
$$
\deg(f_r)\geq R_r\geq\deg(f_{\ell_r})\geq\deg(f_r),
$$
that is, $\deg(f_r)=R_r$, and the proof is complete.

(c) We set $[Q_i]=[P_{\pi(i)}]$, $g_i=f_{\pi(i)}$, 
$\mathbb{X}_1=\{[Q_i]\}_{i=1}^m$ and $F_1=\{f_{\pi(i)}\}_{i=1}^m$. 
Since $\delta_\mathbb{X}(d,r)$ is independent of how we order the 
points of $\mathbb{X}$, that is, 
$\delta_\mathbb{X}(d,r)=\delta_{\mathbb{X}_1}(d,r)$ 
(cf. \cite[Theorem~4.5]{rth-footprint}), by part (b) one has
\begin{align*}
{\rm v}_r(I)&={\rm v}_r(I(\mathbb{X}_1))=\deg(g_r)=
{\rm reg}(\delta_{\mathbb{X}_1}(\,\cdot\,,r))=\min \{d  \geq 1\mid
\rho_r<d,\ \delta_{\mathbb{X}_1}(d,r)=r\}\\ 
&=\min \{d  \geq 1\mid \rho_r<d,\ \delta_{\mathbb{X}}(d,r)=r\}=
{\rm reg}(\delta_{\mathbb{X}}(\,\cdot\,,r))=R_r.
\end{align*}
\quad Thus, ${\rm{v}}_r(I):=\deg(f_{\pi(r)})=\deg(g_r)=R_r$ for $r=1,\ldots,m$. 
\end{proof}

\begin{corollary}\label{then-so-does} Let $F=\{f_i\}_{i=1}^m$ be the set of standard
indicator functions of $\mathbb{X}=\{[P_i]\}_{i=1}^m$. The following
conditions are equivalent.
\begin{enumerate}
\item[\rm(a)] The degree of $f_i$ is equal to ${\rm reg}(S/I)$ 
for all $1\leq i\leq m$. 
\item[\rm(b)] ${\rm reg}(S/I)={\rm v}(I)={\rm reg}(\delta_\mathbb{X})$.
\item[\rm(c)] ${\rm v}_r(I)={\rm reg}(\delta_\mathbb{X}(\,\cdot\,,r))={\rm
reg}(S/I)$ for all $1\leq r\leq m$.
\end{enumerate}
\end{corollary}
\begin{proof} (a)$\Rightarrow$(b) This implication 
follows from Propositions~\ref{indicator-function-prop} and
\ref{reg-min-dis}.

(b)$\Rightarrow$(c) This follows using Lemma~\ref{if},
Proposition~\ref{indicator-function-prop} 
and Theorem~\ref{sarabia-vila-2}.

(c)$\Rightarrow$(a) This implication follows from the definition of ${\rm v}_r(I)$.
\end{proof}

\section{Examples}\label{examples-section}

\begin{example}\label{gorenstein-2essential-example} 
Consider the following set of
points in $\mathbb{P}^3$ over the field $K=\mathbb{F}_3$: 
$$
\mathbb{X}=\{[(-1,-1,-1,1)],[(1,1,1,1)],[(0,1,1,1)],[(0,-1,-1,1)]\}
=\{[P_i]\}_{i=1}^4.
$$
\quad Adapting 
Procedure~\ref{sep12-18}, one can verify the following assertions. If
$S=K[t_1,t_2,t_3,t_4]$ is a polynomial ring with the GRevLex 
order $t_1\succ\cdots\succ t_4$, then the vanishing ideal of
$\mathbb{X}$ is a complete intersection (hence, Gorenstein) given by
$$
I=(t_2-t_3,\, t_3^2-t_4^2,\, t_1^2-t_1t_3),
$$
$H_I(0)=1$, $H_I(1)=3$, $H_I(d)=4$ for $d\geq 2$, ${\rm reg}(S/I)=2$.
The unique set $F=\{f_1,f_2,f_3,f_4\}$, up to 
multiplication by scalars from $K^*$, of standard indicator functions
of $\mathbb{X}$ is given by
$$
F=
\{t_1t_3-t_1t_4,\,  t_1t_3+t_1t_4,\,  t_1t_3+t_1t_4-t_3t_4-t_4^2,\,
t_1t_3-t_1t_4+t_3t_4-t_4^2\}, 
$$
$f_1(P_1)=-1$, $f_2(P_2)=-1$, $f_3(P_3)=1$, $f_4(P_4)=1$, 
and $t_1t_3$, $t_1t_4$ are the only essential monomials. Setting
$\Gamma_1=\{1\}$, $\Gamma_2=\{t_1,t_3,t_4\}$, and $\gamma=(-1,-1,1,1)$,
by Corollary~\ref{combinatorial-condition-projective}
one has
$$
C_\mathbb{X}(1)^\perp={\rm
ev}_1(K\Gamma_2)^\perp=\gamma\cdot{\rm ev}_0(K\Gamma_1)=\gamma\cdot 
C_\mathbb{X}(0)=K\gamma.
$$
\end{example}

\begin{example}\label{gorenstein-essential-example} 
Let $\mathbb{X}$ be the following set of
points in $\mathbb{P}^3$ over the field $K=\mathbb{F}_3$: 
$$
\mathbb{X}=\{[(1,0,-1,1)],[(1,0,1,1)],[(0,1,-1,-1)],[(0,0,1,-1)],[(0,1,1,1)]\}
=\{[P_i]\}_{i=1}^5,
$$
and let $S=K[t_1,t_2,t_3,t_4]$ be a polynomial ring over the field $K$
with the GRevLex order. Using 
Procedure~\ref{aug23-23}, we obtain the following information. The vanishing ideal of
$\mathbb{X}$ is
$$
I=(t_3^2-t_4^2,\ t_2t_3-t_2t_4,\ t_2^2-t_1t_3-t_1t_4+t_3t_4+t_4^2,\
t_1t_2,\ t_1^2-t_1t_4),
$$
and the unique set $F$, up to 
multiplication by scalars from $K^*$, of the standard indicator functions
of $\mathbb{X}$ is given by
\begin{align*}
&F=\{f_1=t_1t_3-t_1t_4,\,  f_2=t_1t_3+t_1t_4,\, 
f_3=t_1t_3+t_1t_4-t_2t_4-t_3t_4-t_4^2,\,\\ 
&\quad\quad\ f_4=t_1t_3-t_1t_4-t_3t_4+t_4^2,\,
f_5=t_1t_3+t_1t_4+t_2t_4-t_3t_4-t_4^2\},
\end{align*}
$f_1(P_1)=1$, $f_i(P_i)=-1$ for $i=2,\ldots,5$, and $t_1t_3$ is an
essential monomial. Note that $t_1t_4$ is also essential. 
The linear form $h=t_1+t_4$  is regular on $S/I$,
that is, $(I\colon h)=I$. The variable $t_4$ is also regular on $S/I$.
Setting $J=(I,h)$, the socle of $S/J$ is
$$
{\rm Soc}(S/J)=(J \colon\mathfrak{m})/J=K(f_i+J)=K(t_3t_4+J),   
$$
where $\mathfrak{m}=(t_1,\ldots,t_4)$. Thus, 
$I$ is a Gorenstein ideal (see Section~\ref{section-gorenstein}) and $I$ is not a complete
intersection. The remainder on division of $f_i$ by $J$ is equal to
$t_3t_4$ for all $i$ (Corollary~\ref{socle-gorenstein-coro}). 

The standard indicator functions of $\mathbb{X}$ and the essential monomials 
depend on the monomial order we choose but the degrees of the standard
indicator functions are independent of the monomial order
(Proposition~\ref{indicator-function-prop}). If
$S=K[t_4,t_3,t_2,t_1]$ has the graded lexicographical 
order (GLex order) $t_4\succ\cdots\succ t_1$, the unique set $F$, up to 
multiplication by scalars from $K^*$, of the standard indicator functions
of $\mathbb{X}$ is given by
$$
F=\{f_1=t_3t_1-t_1^2,\, f_2=t_3t_1+t_1^2,\, f_3=t_3t_2-t_2^2,\,
f_4=t_3^2-t_2^2-t_1^2,\,f_5=t_3t_2+t_2^2\},
$$
$f_1(P_1)=1$, $f_2(P_2)=-1$, $f_3(P_3)=1$, $f_4(P_4)=1$,
$f_5(P_5)=-1$, and for this order there is no essential monomial. This
order is equivalent to the \textit{inverse lexicographical} order
(\textit{invlex} order) on $S=K[t_1,t_2,t_3,t_4]$ defined as $t^a\succ t^b$
if and only if the last non-zero entry of $a-b$ is positive
\cite[p.~59]{CLO}. The
invlex order is implemented in SageMath \cite{sage}.
\end{example}

\begin{example}\label{local-duality-example}
Let $K$ be the finite field $\mathbb{F}_3$, let $S=K[t_1,t_2,t_3]$ be a
polynomial ring with the GRevLex order
$t_1\succ t_2\succ
t_3$, let
$\mathbb{X}=\{[P_1],\ldots,[P_{9}]\}$ be the
following set of points in $\mathbb{P}^2$:
$$
\begin{matrix}
[(0, 0, 1)],& [(0, 1, 1)],& [(0, 2, 1)],& [(1, 0, 1)],& [(1, 1, 1)],\cr 
[(1, 2, 1)],& [(2, 0, 1)],& [(2, 1, 1)],& [(2, 2, 1)],
\end{matrix}
$$
and let $\mathfrak{p}_1,\ldots,\mathfrak{p}_{9}$ be the vanishing
ideals of these points:
$$
\begin{array}{ccccc}
(t_1,t_2),&(t_1,t_2-t_3), &(t_1,t_2+t_3), &(t_2,t_1-t_3),&(t_1-t_3,t_2-t_3),\cr
(t_1+t_2,t_1-t_3),&(t_1+t_3,t_2),
&(t_1+t_2,t_1+t_3),& (t_1-t_2,t_1+t_3).
\end{array}
$$
\quad Using Eq.~\eqref{jun17-23} and
Proposition~\ref{indicator-function-prop}, and adapting 
Procedure~\ref{sep12-18}, we obtain that the vanishing
ideal $I$ of $\mathbb{X}$ is a complete intersection given by 
$$
I=(t_2^3-t_2t_3^2,\, t_1^3-t_1t_3^2), 
$$
the unique list $f_1,\ldots,f_{9}$, up to 
multiplication by scalars from $K^*$, of the standard indicator functions of  
$\mathbb{X}$ is given by
\begin{align*}
&t_1^2t_2^2-t_1^2t_3^2-t_2^2t_3^2+t_3^4,\,t_1^2t_2^2+t_1^2t_2t_3-t_2^2t_3^2-t_2t_3^3 ,\, 
t_1^2t_2^2-t_1^2t_2t_3-t_2^2t_3^2+t_2t_3^3,\\
&t_1^2t_2^2+t_1^2t_2t_3+t_1t_2^2t_3+t_1t_2t_3^2,\,t_1^2t_2^2+t_1t_2^2t_3-t_1^2t_3^2-t_1t_3^3,\, 
t_1^2t_2^2-t_1^2t_2t_3+t_1t_2^2t_3-t_1t_2t_3^2,\\
&t_1^2t_2^2-t_1t_2^2t_3-t_1^2t_3^2+t_1t_3^3,\,t_1^2t_2^2+t_1^2t_2t_3-t_1t_2^2t_3-t_1t_2t_3^2
,\, 
t_1^2t_2^2-t_1^2t_2t_3-t_1t_2^2t_3+t_1t_2t_3^2,
\end{align*}
the v-number ${\rm v}_{\mathfrak{p}_i}(I)$ locally at $\mathfrak{p}_i$
is equal to $4$ for all $i$,  
$f_i(P_i)=1$ for all $i$, the monomial $t_1^2t_2^2$ is essential, 
the Hilbert function of $I$ is
$$
H_I(0)=1,\ H_I(1)=3,\ H_I(2)=6,\ H_I(3)=8,\ H_I(d)=9\mbox{ for }d\geq
4,
$$
the footprint in degrees $1$ to $4$ is given by 
\begin{align*}
&\Delta_\prec(I)_1=\{t_1,t_2,t_3\},\, \Delta_\prec(I)_2=\{t_1^2, t_1t_2, t_1t_3,
t_2^2, t_2t_3, t_3^2\},\\
& \Delta_\prec(I)_3=\{t_1^2t_2,
t_1^2t_3, t_1t_2^2, t_1t_2t_3, t_1t_3^2, t_2^2t_3, t_2t_3^2,
t_3^3\},\\
&\Delta_\prec(I)_4=\{t_1^2t_2^2, t_1^2t_2t_3, t_1^2t_3^2, t_1t_2^2t_3,
t_1t_2t_3^2, t_1t_3^3, t_2^2t_3^2, t_2t_3^3, t_3^4\}.
\end{align*}
\quad To illustrate the local duality criterion of
Theorem~\ref{combinatorial-condition}, we set
$$\Gamma_1=\{t_1,t_2,t_3\}\ \mbox{ and }\ 
\Gamma_2=\{t_1^2t_3, t_1t_2t_3, t_1t_3^2, t_2^2t_3, t_2t_3^2,
t_3^3\}.
$$ 
\quad Conditions (1)-(3) in Theorem~\ref{combinatorial-condition} are satisfied and
$t^e=t_1^2t_2^2$ is essential, then we have 
\begin{align*}
&\beta_i=1\, \forall i,\quad \gamma=(1,\ldots,1),\quad\gamma\cdot{\rm ev}_1(K\Gamma_1)={\rm ev}_1(K\Gamma_1)={\rm
ev}_3(K\Gamma_2)^\perp,\\ 
&{\rm ev}_1(K\Gamma_1)=K\{(0,0,0,1,1,1,-1,-1,-1),\,
(0,1,-1,0,1,-1,0,1,-1),\,
(1,1,1,1,1,1,1,1,1)\},\\
&{\rm
ev}_3(K\Gamma_2)=K\{(0,0,0,1,1,1,1,1,1),\,(0,0,0,0,1,-1,0,-1,1),\,(0,0,0,1,1,1,-1,-1,-1),\\
&\quad\quad\quad\quad\quad\quad(0,1,1,0,1,1,0,1,1),\,(0,1,-1,0,1,-1,0,1,-1),\,(1,1,1,1,1,1,1,1,1)\},
\end{align*}
${\rm v}(I)={\rm
reg}(\delta_\mathbb{X})=4$, ${\rm
reg}(S/I)=4$, $\delta_\mathbb{X}(1)=6$, $\delta_\mathbb{X}(2)=3$,
$\delta_\mathbb{X}(3)=2$, and
$\delta_\mathbb{X}(d)=1$ for $d\geq 4$. By
Theorem~\ref{duality-criterion}, one has
$C_\mathbb{X}(1)^\perp=C_\mathbb{X}(2)$ and, by
Theorem~\ref{sarabia-vila-2}, one has 
$$
R_i = 4  \,\, \mbox{ for } \,\, i=1,\ldots, 9.
$$
\end{example}

\begin{example}\label{Ivan-Hiram-p}
Let $K$ be the finite field $\mathbb{F}_3$, let $S=K[t_1,t_2,t_3,t_4]$ be
a polynomial ring, 
let $\prec$ be the GRevLex
order on $S$, let $I=I(\mathbb{X})$ be the vanishing ideal of the set of
points
$$
\mathbb{X}=\{[(1,0,0,1)],\, [(0,1,0,1)],\, [(0,0,1,1)],\,
[(0,0,0,1)],\, [(2,2,2,1)]\},
$$
let $[P_i]$ be the point in $\mathbb{X}$ in the $i$-th position from the left, and let $\mathfrak{p}_i$ be the
vanishing ideal of $[P_i]$. Using the procedures of
Appendix~\ref{Appendix}, we get the following information.
The set  
\begin{align*}
\mathcal{G}=&\{t_2t_3+t_3^2-t_3t_4,\, t_1t_3+t_3^2-t_3t_4,\,
t_2^2-t_3^2-t_2t_4+t_3t_4,\, \\ 
&\ t_1t_2+t_3^2-t_3t_4,\, t_1^2-t_3^2-t_1t_4+t_3t_4,\, 
t_3^3-t_3t_4^2\},
\end{align*}
is a Gr\"obner basis of $I$. The ideal $I$ is Gorenstein
because  $I$ is a Cohen--Macaulay ideal of height $3$ and 
the minimal resolution of $S/I$ by free $S$-modules is
given by 
$$
0\longrightarrow S(-5)\longrightarrow S(-3)^5\longrightarrow
S(-2)^5\longrightarrow S\longrightarrow S/I\longrightarrow 0. 
$$
\quad The projective dimension of $S/I$ is $3$, and $I$ is not a
complete intersection, that is, $I$ cannot be generated by $3$
elements. The graded ring $S/I$ has symmetric h-vector given by 
${\rm h}(S/I)=(1,3,1)$ and $r_0={\rm reg}(H_I)=2$. By
Proposition~\ref{jul5-23}, $H_I(d)+H_I(r_0-d-1)=|\mathbb{X}|$ 
for all $0\leq d\leq r_0$. The set of standard
monomials of $S/I$ of degree $r_0$ is
$$
\Delta_\prec(I)_{r_0}=\{t_1t_4,\, t_2t_4,\, t_3^2,\, t_3t_4,\, t_4^2\},
$$
and the unique set $F=\{f_i\}_{i=1}^5$, up to 
multiplication by scalars from $K^*$, of standard indicator functions
of $\mathbb{X}$ is
\begin{align*}
&f_1=t_3^2-t_1t_4-t_3t_4,\, f_2=t_3^2-t_2t_4-t_3t_4,\, f_3=t_3^2+t_3t_4,\\ 
&f_4=t_3^2-t_1t_4-t_2t_4+t_3t_4+t_4^2,\, f_5=t_3^2-t_3t_4.
\end{align*}
\quad Thus, ${\rm v}_{\mathfrak{p}_i}(I)=r_0$ for all $i$
(Proposition~\ref{indicator-function-prop}), and the two conditions of
Theorem~\ref{duality-criterion}(b) are satisfied. Letting 
$\beta=({\rm lc}(f_1)f_1(P_1)^{-1},\ldots,{\rm
lc}(f_m)f_m(P_m)^{-1})$, one has $f_i(P_i)=-1$ for $i\neq 
4$, $f_4(P_4)=1$, and $\beta=(-1,-1,-1,1,-1)$. By 
Proposition~\ref{gorenstein-essential}(b), $H=\beta$ is a parity check matrix
of $C_\mathbb{X}(r_0-1)=C_\mathbb{X}(1)$. Hence, by
Theorem~\ref{duality-criterion}, 
we obtain 
$$
C_\mathbb{X}(1)^\perp=\beta\cdot C_\mathbb{X}(0)=K(-1,-1,-1,1,-1).
$$
\end{example}

\begin{example}\label{self-dual-example}
Let $K$ be the finite field $\mathbb{F}_{4}$, let $a$ be a generator
of the cyclic group $K^*$, let $S=K[t_1,t_2,t_3]$ be a
polynomial ring with the GRevLex order, let
$\mathbb{X}=\{[P_1],\ldots,[P_{6}]\}$ be the
following set of points in $\mathbb{P}^2$:
$$
\begin{matrix}
[(1, 0, 1)],& [(a, 0, 1)],& [(a^2, 0, 1)],& [(0, 1, 1)],& [(0, a^2, 1)],& [(0,
a, 1)],
\end{matrix}
$$
and let $\mathfrak{p}_1,\ldots,\mathfrak{p}_{6}$ be the vanishing
ideals of these points. Using Eq.~\eqref{jun17-23} and
Proposition~\ref{indicator-function-prop}, 
together with Procedures~\ref{sep12-18} and \ref{aug23-23}, 
we obtain that the vanishing
ideal $I$ of $\mathbb{X}$ is given by 
$$
I=(t_1t_2,\ t_1^3+t_2^3+t_3^3), 
$$
a Gr\"obner basis of $I$ is $\mathcal{G}=\{t_1t_2,\,
t_1^3+t_2^3+t_3^3,\, t_2^4+t_2t_3^3\}$, 
the unique list $f_1,\ldots,f_{6}$, up to 
multiplication by scalars from $K^*$, of the standard indicator functions of 
$\mathbb{X}$ is given by
\begin{align*}
&t_2^3+t_1^2t_3+t_1t_3^2+t_3^3,\,
t_2^3+at_1^2t_3+a^2t_1t_3^2+t_3^3,\,
t_2^3+a^2t_1^2t_3+at_1t_3^2+t_3^3,\\
&t_2^3+t_2^2t_3+t_2t_3^2,\, t_2^3+a^2t_2^2t_3+at_2t_3^2,\,
t_2^3+at_2^2t_3+a^2t_2t_3^2,
\end{align*}
${\rm lc}(f_i)=f_i(P_i)=1$ for $i=1,\ldots,6$, 
the Hilbert function of $I$ is
$$
H_I(0)=1,\ H_I(1)=3,\ H_I(2)=5,\ H_I(3)=6,\ H_I(d)=6\mbox{ for
}d\geq 3,
$$
${\rm v}(I)={\rm
reg}(\delta_\mathbb{X})=3$, $r_0={\rm
reg}(S/I)=3$, $\delta_\mathbb{X}(1)=3$, $\delta_\mathbb{X}(2)=2$, and
$\delta_\mathbb{X}(d)=1$ for $d\geq 3$. The footprint of $I$ in
degree $r_0$ is 
$$\Delta_\prec(I)_{r_0}=\{t_1^2t_3,\  t_1t_3^2,\  t_2^3,\  t_2^2t_3,\  t_2t_3^2,\ 
t_3^3\}.$$
\quad As $I$ is a complete
intersection and $r_0-d-1=d$, where $d=1$, by
Corollary~\ref{duality-criterion-g} and
Proposition~\ref{gorenstein-essential}, we have 
$C_\mathbb{X}(1)^\perp=C_\mathbb{X}(1)$. Note that
$C_\mathbb{X}(0)\subsetneq C_\mathbb{X}(0)^\perp$ because
$C_\mathbb{X}(0)=K(1,\ldots,1)$ and ${\rm char}(K)=2$. Let $P$ be the matrix with
rows $P_1,\ldots,P_6$ and let $C_1,C_2,C_3$ be the columns of $P$. By
Proposition~\ref{self-dual-1},
$C_\mathbb{X}(1)^\perp=C_\mathbb{X}(1)=K\{C_1,C_2,C_3\}$. 
\end{example}

\begin{example}\label{algorithm-example}
Let $K$ be the finite field $\mathbb{F}_{9}$, let $a$ be a generator
of the cyclic group $K^*$, let $S=K[t_1,t_2]$ be a polynomial ring,
and let $\mathbb{X}$ be the projective space $\mathbb{P}^1$:
$$\mathbb{X}=
\{[(1,0)],[(1,1)],[(1,a)],[(1,a^2)],[(1,a^3)],[(1,a^4)],[(1,a^5)],[(1,a^6)],
[(1,a^7)],[(0,1)]\}.
$$
\quad The vanishing ideal of $\mathbb{X}$ is
$I=(t_1^9t_2-t_1t_2^9)$. Using Procedure~\ref{sep29-23}, we get 
that $C_\mathbb{X}(4)$ is the only self dual code of the form
$C_\mathbb{X}(d)$ (cf. \cite[Theorem~2]{sorensen}) and there are no
other self orthogonal codes of the form $C_\mathbb{X}(d)$.  
\end{example}

\begin{example}\label{p2} Let $\mathbb{X}$ be the set of all
$K$-rational points of 
$\mathbb{P}^2$, where $K=\mathbb{F}_3$, and let
$S$ be the polynomial ring $K[t_1,t_2,t_3]$. The vanishing ideal $I$ of
$\mathbb{X}$ is given by \cite{sorensen}:
$$I=(t_1t_2^3-t_1^3t_2,t_1t_3^3-t_1^3t_3,t_2t_3^3-t_2^3t_3).$$ 
\quad Using Procedure~\ref{sep29-23}, we get that $C_\mathbb{X}(1)$ and 
$C_\mathbb{X}(2)$ are the only self orthogonal codes of the form 
$C_\mathbb{X}(d)$ and they are not self dual. The set $\mathbb{X}$ has
$13$ elements and ${\rm reg}(S/I)=5$. From the minimal graded free
resolution of $S/I$:
$$
0\longrightarrow S(-5)\oplus S(-7)\longrightarrow
S(-4)^3\longrightarrow S\longrightarrow S/I\longrightarrow 0,
$$ 
we obtain that $I$ is neither complete intersection, nor Gorenstein.
The v-number of $I$ is $5$ and the minimum socle 
degree of $I$ is $3$.
\end{example}

\begin{example}\label{monomially-self-dual-example}
Let $K$ be the finite field $\mathbb{F}_{5}$, let $S=K[t_1,t_2]$ be a
polynomial ring with the GRevLex order, let
$\mathbb{X}=\{[P_1],\ldots,[P_{4}]\}$ be the
projective torus in $\mathbb{P}^1$ defined by the points:
$$
\begin{matrix}
P_1=[(1, 1)],&P_2= [(2,1)],& P_3=[(3,1)],& P_4=[(4,1)].
\end{matrix}
$$
\quad The vanishing ideal of $\mathbb{X}$ is $I=(t_1^4-t_2^4)$ and the
regularity $r_0$ of $S/I$ is $3$. 
Using Procedures~\ref{sep12-18} and \ref{aug23-23}, 
we obtain that the unique list $f_1,\ldots,f_{4}$, up to 
multiplication by scalars from $K^*$, of the standard indicator functions of 
$\mathbb{X}$ is given by
\begin{align*}
&t_1^3+t_1^2t_2+t_1t_2^2+t_2^3,\, t_1^3+2t_1^2t_2-t_1t_2^2-2t_2^3,\, 
t_1^3-2t_1^2t_2-t_1t_2^2+2t_2^3,\, t_1^3-t_1^2t_2+t_1t_2^2-t_2^3,
\end{align*}
${\rm lc}(f_i)=1$ for $i=1,\ldots,4$,
$(f_1(P_1),\ldots,f_4(P_4))=(-1,2,-2,1)$, and 
$$\beta:=({\rm lc}(f_1)f_1(P_1)^{-1},\ldots,{\rm
lc}(f_4)f_4(P_4)^{-1})=(-1,3,-3,1).$$
\quad By Proposition~\ref{gorenstein-essential}, the $1\times 4$ matrix $H$
defined by the vector $\beta$ is a parity check matrix of
the linear code $C_\mathbb{X}(r_0-1)$, that is,
\begin{equation*}
C_\mathbb{X}(2)=C_\mathbb{X}(r_0-1)=\{x\in K^4\mid Hx=0\}.
\end{equation*}
\quad As $I$ is a complete intersection and $r_0-d-1=d$, where $d=1$, by
Proposition~\ref{self-dual}, $C_\mathbb{X}(1)$ is monomially equivalent
to $C_\mathbb{X}(1)^\perp$. Then, by Corollary~\ref{duality-criterion-g} and
Proposition~\ref{gorenstein-essential}, we have 
$C_\mathbb{X}(1)^\perp=\beta\cdot C_\mathbb{X}(1)$. Note that the
linear code $C_\mathbb{X}(1)$ is not self dual because the vector 
$(t_2(P_1),\ldots,t_2(P_4))=(1,1,1,1)$ is in $C_\mathbb{X}(1)$ and is
not orthogonal to itself (cf. Theorem~\ref{self-dual-general}). In
this example $\delta_\mathbb{X}(1)=3$, $H_I(1)=2$, and the code
$C_\mathbb{X}(1)$ is
MDS, that is, equality holds in the Singleton bound
\cite[p.~71]{Huffman-Pless}.
\end{example}

\begin{example}\label{Hiram-first-counterxample}
Let $K$ be the finite field $\mathbb{F}_3$, let $S=K[t_3,t_2,t_1]$ be a
polynomial ring with the GLex order $t_3\succ t_2\succ
t_1$, let
$\mathbb{X}=\{[P_1],\ldots,[P_{10}]\}$ be the
following set of points in $\mathbb{P}^2$:
$$
\begin{matrix}
[(1,0,1)],&[(1,0,0)],&[(1,0,2)],&[(1,1,0)],&[(1,1,1)],\cr
[(1,1,2)],&[(0,0,1)],&[(0,1,0)],&[(0,1,1)],&[(0,1,2)],
\end{matrix}
$$
and let $\mathfrak{p}_1,\ldots,\mathfrak{p}_{10}$ be the vanishing
ideals of these points:
$$
\begin{array}{ccccc}
(t_2,-t_3+t_1),&(t_2,t_3),&(t_2,-t_3-t_1), &(-t_2+t_1,t_3), &(-t_3+t_1,-t_3+t_2),\cr
(-t_3-t_1,-t_3-t_2),&(t_1,t_2),&(t_1,t_3), &(t_1,-t_3+t_2),& (t_1,-t_3-t_2).
\end{array}
$$
\quad Using Eq.~\eqref{jun17-23} and
Proposition~\ref{indicator-function-prop}, 
together with Procedure~\ref{sep12-18}, we obtain that the vanishing
ideal $I$ of $\mathbb{X}$ is given by 
$$
I=(t_1t_2^2-t_1^2t_2,\,t_1t_3^3-t_1^3t_3,\,t_2t_3^3-t_2^3t_3), 
$$
the unique list $f_1,\ldots,f_{10}$, up to 
multiplication by scalars from $K^*$,  of the standard indicator functions of 
$\mathbb{X}$ is given by
\begin{align*}
&t_3^2t_2t_1-t_3^2t_1^2+t_3t_2t_1^2-t_3t_1^3,\, t_3^2t_2t_1-t_3^2t_1^2-t_2t_1^3+t_1^4,\, 
 t_3^2t_2t_1-t_3^2t_1^2-t_3t_2t_1^2+t_3t_1^3,\, \\ 
&t_3^2t_2t_1-t_2t_1^3,\,  t_3^2t_2t_1+t_3t_2t_1^2,\, 
t_3^2t_2t_1-t_3t_2t_1^2,\, t_3^3-t_3t_2^2+t_3t_2t_1-t_3t_1^2,\, 
\\
& t_3^2t_2^2-t_3^2t_2t_1-t_2^4+t_2t_1^3,\, 
t_3^2t_2^2-t_3^2t_2t_1+t_3t_2^3-t_3t_2t_1^2,\, 
t_3^2t_2^2-t_3^2t_2t_1-t_3t_2^3+t_3t_2t_1^2,
\end{align*} 
and ${\rm v}(I)=\deg(f_7)=\deg(t_3^3-t_3t_2^2+t_3t_2t_1-t_3t_1^2)=3$.
The Hilbert function of $I$ is
$$
H_I(0)=1,\ H_I(1)=3,\ H_I(2)=6,\ H_I(3)=9,\ H_I(d)=10\mbox{ for
}d\geq 4,
$$
${\rm v}(I)={\rm
reg}(\delta_\mathbb{X})=3$ and ${\rm
reg}(S/I)=4$. By
Theorem~\ref{sarabia-vila-2}, one has 
$$
R_1=3,\  R_i = 4  \,\, {\mbox{for}} \,\, i=2,\ldots, 10.
$$
\quad By convention we set $\delta_\mathbb{X}(d,r)=\infty$ if 
$\rho_r\geq d$ or equivalently if $r>H_I(d)$. The matrix of
generalized Hamming weights is the $r_0 \times |\mathbb{X}|=4 \times
10$ matrix given by:
$$
(\delta_{\mathbb{X}}(d,r))=\begin{pmatrix}
6 & 9 & 10 & \infty & \infty & \infty & \infty & \infty & \infty & \infty \\
3 & 5 & 6 & 8 & 9 & 10 & \infty & \infty & \infty & \infty \\
1 & 3 & 4 & 5 & 6  & 7 & 8 & 9 & 10  & \infty \\
1 & 2 & 3 & 4 & 5 & 6 & 7 & 8 & 9 & 10
\end{pmatrix}.
$$
\quad This follows using that $({\rm
fp}_I(d,r))\leq(\delta_{\mathbb{X}}(d,r))$
\cite[Theorem~4.9]{rth-footprint}, where the matrix on the
left is the footprint matrix of $I$ \cite[p.~319--320]{rth-footprint},
and the fact that the rows (resp. columns) of
$(\delta_{\mathbb{X}}(d,r))$ form an increasing (resp. decreasing) 
sequence \cite{min-dis-generalized,wei}.
It is worth noticing that $I$ is strongly Geil--Carvalho
\cite[Definition~3.12]{min-dis-generalized}, that is, 
$({\rm fp}_I(d,r))=(\delta_{\mathbb{X}}(d,r))$.
\end{example}

\begin{example}\label{example-parameterized} Let $S=K[t_1,\ldots,t_s]$
be a polynomial ring over $K=\mathbb{F}_q$, 
let $y_1,\ldots,y_n$ be the indeterminates of a ring of Laurent
polynomials with coefficients in $K$ and let 
$\{y^{v_1},\ldots,y^{v_s}\}$ be a finite set of Laurent
monomials.  Given an integer vector $v_i=(v_{i,1},\ldots,v_{i,n})\in\mathbb{Z}^n$, we set 
$$
y^{v_i}:=y_1^{v_{i,1}}\cdots y_n^{v_{i,n}},\ \ \ \ i=1,\ldots,s. 
$$
\quad Following \cite{afinetv,algcodes}, we define the 
{\it algebraic toric set\/} parameterized  by $y^{v_1},\ldots,y^{v_s}$
as 
$$
\mathbb{X}:=\{[(x_1^{v_{1,1}}\cdots x_n^{v_{1,n}},\ldots,x_1^{v_{s,1}}\cdots
x_n^{v_{s,n}})]\, \vert\, x_i\in K^*\mbox{ for all
}i\}\subset\mathbb{P}^{s-1},
$$
where $K^*=K\setminus\{0\}$. According to \cite{algcodes}, $I(\mathbb{X})$ is a
lattice ideal of $S$, that is, $I(\mathbb{X})$ is generated by binomials and
$t_i$ is not a zero-divisor of $S/I(\mathbb{X})$ for all $i$. The
results of Section~\ref{section-gorenstein} can be used to study 
the socle and Gorenstein property of $S/I(\mathbb{X})$, and duality
criteria. If
$y^{v_i}=y_i$ for $i=1,\ldots,s$, $\mathbb{X}$ is called a
\textit{projective torus} and $C_\mathbb{X}(d)$ is called a 
\textit{generalized Reed--Solomon
code} of degree $d$ \cite{duursma-renteria-tapia,GRH}. Since a projective torus is a complete
intersection \cite{GRH}, by Proposition~\ref{gorenstein-essential},
in this case essential
monomials exist when $\prec$ is the GRevLex order. 
\end{example}

\begin{example}\label{example2}
Let $K$ be the field $\mathbb{F}_3$, let $S=K[t_1,t_2,t_3]$ be a
polynomial ring with the GRevLex order, let
$\mathbb{X}=\{[P_1],\ldots,[P_{7}]\}$ be the
following set of points in $\mathbb{P}^2$:
$$
[(1,0,1)],\,[(1,1,1)], \,[(1,1,2)], \,[(0,0,1)],\, 
\,[(0,1,0)],\,[(0,1,1)],\,[(0,1,2)],
$$
and let $\mathfrak{p}_1,\ldots,\mathfrak{p}_{7}$ be the vanishing
ideals of these points:
$$
\begin{array}{ccccc}
(t_2,-t_3+t_1),&(-t_3+t_1,-t_3+t_2),&(-t_3-t_1,-t_3-t_2),&(t_1,t_2),&(t_1,t_3),\cr
(t_1,-t_3+t_2),& (t_1,-t_3-t_2).&&&
\end{array}
$$
\quad Using Eq.~\eqref{jun17-23} and
Proposition~\ref{indicator-function-prop}, 
and changing initial data in Procedure~\ref{sep12-18}, we obtain that the vanishing
ideal $I$ of $\mathbb{X}$ is given by 
$$
I=(t_1t_2^2+t_1^2t_3-t_1t_2t_3-t_1t_3^2,\,
t_1^2t_2+t_1^2t_3-t_1t_2t_3-t_1t_3^2,\, 
t_1^3-t_1t_3^2,\, t_2^3t_3-t_2t_3^3),
$$
the unique list $f_1,\ldots,f_{7}$, up to 
multiplication by scalars from $K^*$, of the standard indicator functions for
$\mathbb{X}$ is given by
\begin{align*}
&t_1^2-t_1t_2,\, t_1^2+t_1t_2-t_1t_3,\, t_1^2-t_1t_3,\, 
-t_1^2t_3+t_1t_2t_3-t_2^2t_3+t_3^3,\, t_2^3-t_2t_3^2,\,\\
&-t_1^2t_3-t_1t_2t_3-t_2^2t_3+t_1t_3^2-t_2t_3^2,\,  
-t_1^2t_3+t_2^2t_3+t_1t_3^2-t_2t_3^2,
\end{align*} 
the Hilbert function of $I$ is
$$
H_I(0)=1,\ H_I(1)=3,\ H_I(2)=6,\ H_I(d)=7\mbox{ for }d\geq 3.
$$
${\rm v}(I)={\rm
reg}(\delta_\mathbb{X})=2$ and $r_0={\rm
reg}(S/I)=3$. By Theorem~\ref{sarabia-vila-2}, one has 
$$
R_1=R_2=R_3=2, R_4=R_5=R_6=R_7=3.
$$
\quad By convention we set $\delta_\mathbb{X}(d,r)=\infty$ if 
$\rho_r\geq d$ or equivalently if $r>H_I(d)$. The matrix of
generalized Hamming weights is the $r_0\times|\mathbb{X}|=3\times 7$
matrix given by:
$$
(\delta_{\mathbb{X}}(d,r))=\begin{pmatrix}
3 & 6 & 7 & \infty & \infty & \infty & \infty \\
1 & 2 & 3 & 5 & 6 & 7 & \infty \\
1 & 2 & 3 & 4 & 5 & 6 & 7
\end{pmatrix}.
$$
\end{example}

\begin{appendix}
\section{Procedures for {\it Macaulay\/}$2$}\label{Appendix}
\begin{procedure}\label{sep12-18} Using \textit{Macaulay}$2$
\cite{mac2}, version 1.15, to
compute the standard indicator 
functions of a finite set $\mathbb{X}$ of projective points, 
the v-number of the vanishing ideal $I(\mathbb{X})$, the regularity
index of $\delta_\mathbb{X}$, and the minimum
distance $\delta_\mathbb{X}(d)$ of the Reed--Muller-type code
$C_\mathbb{X}(d)$. This procedure also computes the generalized
Hamming weights of $C_\mathbb{X}(d)$ over small fields and the
generalized footprint of $I(\mathbb{X})$.  This procedure
corresponds to Example~\ref{Hiram-first-counterxample}.
\begin{verbatim}
q=3, G=ZZ/q, S=G[t3,t2,t1,MonomialOrder=>GLex]
p1=ideal(t2,t1-t3),p2=ideal(t2,t3),p3=ideal(t2,2*t1-t3)
p4=ideal(t1-t2,t3),p5=ideal(t1-t3,t2-t3), p6=ideal(2*t1-t3,2*t2-t3)
p7=ideal(t1,t2),p8=ideal(t1,t3),p9=ideal(t1,t2-t3),p10=ideal(t1,2*t2-t3)
I=intersect(p1,p2,p3,p4,p5,p6,p7,p8,p9,p10)
M=coker gens gb I, regularity M, degree M
init=ideal(leadTerm gens gb I)
aux=(d,r)->degree M-max apply(apply(subsets(toList set apply(
toList set(apply(apply(apply(apply(toList ((set(0..q-1))^**(
#flatten entries basis(d,M))-(set{0})^**(#flatten entries 
basis(d,M)))/deepSplice,toList),x->basis(d,M)*vector x),entries),
n->n#0)),m->(leadCoefficient(m))^(-1)*m),r),ideal),
x-> if #(set flatten entries leadTerm gens x)==r and not 
quotient(I,x)==I then degree(I+x) else 0)
--This computes the generalized Hamming weights 
--and gives "infinite" when they are not defined
newgenmd=(d,r)->if r>hilbertFunction(d,M) then Infinite else aux(d,r)
time newgenmd(1,1)--minimum distance gives 6
time newgenmd(2,1)--minimum distance gives 3
time newgenmd(3,1)--minimum distance gives 1
time newgenmd(2,2)--gives 5
time newgenmd(1,4)--gives infinite
er=(x)-> if not quotient(init,x)==init then degree ideal(init,x) else 0
--this is the generalized footprint
fpr=(d,r)->degree M - max apply(apply(apply(subsets(flatten entries 
basis(d,M),r),toSequence),ideal),er)
time fpr(1,1), fpr(2,1), fpr(3,1), fpr(2,2), fpr(1,4)
--Computing the set F={f1,f2,f3,f4,f5,f6,f7,f8,f9,f10}
--of standard indicator functions 
J1=quotient(I,p1), soc1=J1/I, degrees mingens soc1--gives 4
f1=toString flatten entries mingens soc1--gives f1
J2=quotient(I,p2), soc2=J2/I, degrees mingens soc2--gives 4
f2=toString flatten entries mingens soc2--gives f2
J3=quotient(I,p3), soc3=J3/I, degrees mingens soc3--gives 4
f3=toString flatten entries mingens soc3--gives f3
J4=quotient(I,p4), soc4=J4/I, degrees mingens soc4--gives 4
f4=toString flatten entries mingens soc4--gives f4
J5=quotient(I,p5), soc5=J5/I, degrees mingens soc5--gives 4
f5=toString flatten entries mingens soc5--gives f5
J6=quotient(I,p6), soc6=J6/I, degrees mingens soc6--gives 4
f6=toString flatten entries mingens soc6--gives f6
J7=quotient(I,p7), soc7=J7/I, degrees mingens soc7--gives 3
f7=toString flatten entries mingens soc7--gives f7
J8=quotient(I,p8), soc8=J8/I, degrees mingens soc8--gives 4
f8=toString flatten entries mingens soc8--gives f8
J9=quotient(I,p9), soc9=J9/I, degrees mingens soc9--gives 4
f9=toString flatten entries mingens soc9--gives f9
J10=quotient(I,p10), soc10=J10/I, degrees mingens soc10--gives 4
f10=toString flatten entries mingens soc10--gives f10
\end{verbatim}
\end{procedure}

\begin{procedure}\label{sep29-23}
This procedure is based on 
Proposition~\ref{selfo-char} and 
Corollary~\ref{selfd-char}. It determines whether 
or not a given Reed--Muller type code $C_\mathbb{X}(d)$ is self orthogonal or self
dual. The pseudocode is: (a) If $(1,\ldots,1)\in
C_\mathbb{X}(2d)^\perp$ then
``self orthogonal'' else ``not self orthogonal``. (b) If $(1,\ldots,1)\in
C_\mathbb{X}(2d)^\perp$ and $|\mathbb{X}|=2H_I(d)$ then
``self dual'' else ``not self dual''. 
The following is the implementation in \textit{Macaulay}$2$
\cite{mac2} and corresponds to Example~\ref{algorithm-example}. We 
load a package on numerical algebraic
geometry to evaluate systems of polynomials at given sets of points. 
\begin{verbatim}
load "NAGtypes.m2"
--Polynomial ring with the default order GRevLex:
q=9, S=GF(q,Variable=>a)[t1,t2]
I=ideal(t1^q*t2-t1*t2^q), G=gb I
M=coker gens gb I, regularity M, degree M
init=ideal(leadTerm gens gb I)
X={{1,0},{1,1},{1,a},{1,a^2},{1,a^3},{1,a^4},{1,a^5},{1,a^6},
{1,a^7},{0,1}}
--These are the points in the right format:
B=apply(X,x->point{toList x})
--This function determines whether or not C_X(d) is self dual 
selfdual=(d)->if degree(M)==2*hilbertFunction(d,M) and (toList set
apply (apply(apply(0..#(flatten entries basis(2*d,M))-1,
n-> apply(B,x->evaluate(polySystem{(flatten entries
basis(2*d,M))#n},x))), x-> (toList x)),sum)=={0})  then 
SELFDUALCODE else NOTSELFDUAL 
apply(0..regularity(M),selfdual)
selfdual(4), selfdual(9)
--This function determines whether or not C_X(d) is self orthogonal 
selforthogonal=(d)->if toList set apply(apply(apply(0..#(flatten
entries basis(2*d,M))-1,n-> apply(B,x->evaluate
(polySystem{(flatten entries basis(2*d,M))#n},x))),x-> (toList
x)),sum)=={0} then SELFORTHOGONAL  else NOTSELFORTHOGONAL
apply(0..regularity(M),selforthogonal)
selforthogonal(4), selforthogonal(9)
\end{verbatim}
\end{procedure}

\begin{procedure}\label{aug23-23} Using \textit{Macaulay}$2$ 
\cite{mac2}, version 1.15,  to
compute the socle of an Artinian algebra and deciding whether or not a
vanishing ideal is Gorenstein. This procedure also computes the 
essential monomials of a finite set of projective points, and the Hilbert
function and Hilbert series of a vanishing ideal that can be used to
check the symmetry of its h-vector and the hypothesis of our duality
criteria. This procedure
corresponds to Example~\ref{gorenstein-essential-example}.
\begin{verbatim}
restart
load "NAGtypes.m2"
--Polynomial ring with the default order GRevLex:
q=3, S=ZZ/q[t1,t2,t3,t4]
--Use the following order for computations using the invlex order
--q=3, S=ZZ/q[t4,t3,t2,t1,MonomialOrder=>Lex]
--vanishing ideal of points
p1=ideal(t2,t1+t3,t1-t4),p2=ideal(t2,t1-t3,t1-t4),p3=ideal(t1,t2+t3,t2+t4)
p4=ideal(t1,t2,t3+t4),p5=ideal(t1,t2-t3,t2-t4)
I=intersect(p1,p2,p3,p4,p5), G=gb I
--The quotient ring S/I
M=coker gens gb I, regularity M, degree M
init=ideal(leadTerm gens gb I)
--Computing the set of standard indicator functions F={f1,f2,f3,f4,f5}
--The minimal generator of (I: pi)/I for i=1,...,5 gives 
--the standard indicator functions fi
J1=quotient(I,p1), soc1=J1/I
f1=toString flatten entries mingens soc1
J2=quotient(I,p2), soc2=J2/I
f2=toString flatten entries mingens soc2
J3=quotient(I,p3), soc3=J3/I
f3=toString flatten entries mingens soc3
J4=quotient(I,p4), soc4=J4/I
f4=toString flatten entries mingens soc4
J5=quotient(I,p5), soc5=J5/I
f5=toString flatten entries mingens soc5
--Finding an Artinian reduction
H=t1+t4, quotient(I,H)==I
J=I+ideal(H), gens gb J
--The Socle of S/J can be used to check the Gorenstein property
SocleJ=quotient(J,ideal(t1,t2,t3,t4))/J
degrees mingens SocleJ
--If the Socle of S/J has one minimal generator the
--ring S/I is Gorenstein
gensoc=toString flatten entries mingens SocleJ
G1=gb J
--dividing each fi by J gives the remainder
(flatten entries mingens soc1)#0 % G1
(flatten entries mingens soc2)#0 % G1
(flatten entries mingens soc3)#0 % G1
(flatten entries mingens soc4)#0 % G1
(flatten entries mingens soc5)#0 % G1
--Hilbert series and h-vector
HS=hilbertSeries M, reduceHilbert HS
--To check the Gorenstein property directly from the 
--graded minimal resolution
res M
r0=regularity M
--Hilbert function
hilbertFunction(0,M), hilbertFunction(1,M), hilbertFunction(2,M)
--To compute the values of standard indicator functions
--and normalize to value 1 if necessary
g1=(flatten entries mingens soc1)#0
g2=(flatten entries mingens soc2)#0
g3=(flatten entries mingens soc3)#0
g4=(flatten entries mingens soc4)#0
g5=(flatten entries mingens soc5)#0
--Indicator functions 
ps1=polySystem{g1}
ps2=polySystem{g2}
ps3=polySystem{g3}
ps4=polySystem{g4}
ps5=polySystem{g5}
--These are the points of X that define the evaluation code:
X={{1,0,-1,1},{1,0,1,1},{0,1,-1,-1},{0,0,1,-1},{0,1,1,1}}
--These are the points in the right format:
B=apply(X,x->point{toList x})
ev1=(x)->evaluate(ps1,x), ev2=(x)->evaluate(ps2,x)
ev3=(x)->evaluate(ps3,x), ev4=(x)->evaluate(ps4,x)
ev5=(x)->evaluate(ps5,x)
--These are the values of the standard indicator function fi 
--at all points of X
apply(B,ev1), apply(B,ev2), apply(B,ev3)
apply(B,ev4), apply(B,ev5)
\end{verbatim}
\end{procedure}
\end{appendix}

\section*{Acknowledgments} 
We thank the referees for a careful
reading of the paper and for the improvements suggested. 
Computations with \textit{Macaulay}$2$ \cite{mac2} were important to
find the standard indicator functions of a finite set of projective
points and to find the parameters of Reed--Muller-type codes.  

\section*{\ } No datasets were generated or analysed during the current study.

\subsection*{Statement} On behalf of all authors, the corresponding
author states that there is no conflict of interest.

\bibliographystyle{plain}

\begin{thebibliography}{10}

\bibitem{AM}{M.~F. Atiyah and I.~G. Macdonald, {\it Introduction to
Commutative Algebra}, Addison-Wesley, Reading, MA, 1969.}

\bibitem{package} T. Ball, E. Camps, H. Chimal-Dzul, D.
Jaramillo, H. L\'opez, N. Nichols, M. Perkins, I. Soprunov, G.
Vera and G. Whieldon, {Coding theory package for Macaulay2}, 
J. Softw. Algebra Geom. {\bf 11} (2021), no. 1, 113--122.
 


\bibitem{Becker-Weispfenning} T. Becker and V. Weispfenning, \textit{Gr\"obner bases A
Computational Approach to Commutative Algebra}, in cooperation with
Heinz Kredel,  Graduate Texts in Mathematics {\bf 141},
Springer-Verlag, New York, 1993. 

\bibitem{boij-level} M. Boij, Artin level modules, 
J. Algebra {\bf 226} (2000), no. 1, 361--374.

\bibitem{Britz} T. Britz, MacWilliams identities and matroid polynomials, 
Electron. J. Combin. {\bf 9} (2002), no. 1, Paper 19, 16 pp.

\bibitem{Buch}{B. Buchberger, An algorithmic method in polynomial
ideal theory, in Recent Trends in Mathematical Systems Theory (N.K.
Bose, Ed.), Reidel, Dordrecht, 1985, 184--232.}

\bibitem{carvalho} C.  Carvalho, On the second Hamming weight of some
Reed--Muller type codes, Finite Fields Appl. {\bf 24} (2013), 
88--94.

\bibitem{Ceria-etal} M. Ceria, S. Lundqvist and T. Mora, 
Degr\"obnerization: a political manifesto, Appl. Algebra Engrg. Comm.
Comput. {\bf 33} (2022), no. 6, 675--723.

\bibitem{min-dis-generalized} S. M. Cooper, A. Seceleanu, S. O. Toh\v{a}neanu,
M. Vaz Pinto and R. H. Villarreal, 
Generalized minimum distance functions and algebraic invariants of
Geramita ideals, Adv. in Appl. Math. {\bf 112} (2020), 101940.

\bibitem{CLO} D. Cox, J. Little and D. O'Shea, {\it Ideals, 
Varieties, and Algorithms\/}, Springer-Verlag, 1992.

\bibitem{Davis-etal} E. D. Davis, A. V. Geramita and F. Orecchia, 
Gorenstein algebras and the Cayley--Bacharach theorem, 
Proc. Amer. Math. Soc. {\bf 93} (1985), no. 4, 593--597.

\bibitem{duursma-renteria-tapia} I. M. Duursma, C. Renter\'\i a and
H. Tapia-Recillas,  
Reed--Muller codes on complete intersections, Appl. Algebra Engrg.
Comm. Comput.  {\bf 11}  (2001),  no. 6, 455--462.

\bibitem{Eisen}{D. Eisenbud, {\it Commutative Algebra with a view
toward Algebraic Geometry\/}, Graduate
Texts in  Mathematics {\bf 150}, Springer-Verlag, 1995.}

\bibitem{eisenbud-syzygies} D. Eisenbud, {\it The geometry of syzygies: A
second course in commutative algebra and algebraic geometry},
Graduate Texts in Mathematics {\bf 229}, Springer-Verlag, New York,
2005.

\bibitem{Elias-Rossi} J. Elias and M. E. Rossi,
Inverse system of Gorenstein points in $\mathbb{P}^n_{k}$. Preprint,
2023, {\tt arXiv:2301.07056v2}.
%To appear in Israel J. of Math
%https://doi.org/10.48550/arXiv.2301.07056


\bibitem{EvGr}{E. G. Evans and P. Griffith, {\it Syzygies\/}, 
London Math. Soc., Lecture Note Series {\bf 106}, Cambridge 
University Press, Cambridge, 1985.}

\bibitem{Fro3}{R. Fr\"{o}berg and D. Laksov, {\it Compressed Algebras}, 
Lecture Notes in Mathematics {\bf 1092} (1984), 
Springer-Verlag, pp. 121--151.}

\bibitem{geil-2008} O. Geil, Evaluation codes from an affine variety
code perspective, Advances in algebraic geometry codes, 153--180, 
Ser. Coding Theory Cryptol., 5, World Sci. Publ., Hackensack, NJ,
2008.

\bibitem{geil-hoholdt} 
O. Geil and T. H{\o}holdt, 
Footprints or generalized Bezout's theorem, 
IEEE Trans. Inform. Theory {\bf 46}  (2000),  no. 2, 635--641.

\bibitem{geil-pellikaan} O. Geil and R. Pellikaan, 
On the structure of order domains, Finite Fields Appl. {\bf 8}
(2002), no. 3, 369--396. 

\bibitem{geramita-cayley-bacharach} A. V. Geramita, M. Kreuzer and L.
Robbiano, 
Cayley--Bacharach schemes and their canonical modules, Trans. Amer.
Math. Soc. {\bf 339}  (1993),  no. 1, 163--189. 

\bibitem{rth-footprint}
M. Gonz\'alez-Sarabia, J. Mart\'\i nez-Bernal, R. H. Villarreal and
 C. E. Vivares, Generalized minimum distance functions, 
J. Algebraic Combin. {\bf 50} (2019), no. 3, 317--346. 

\bibitem{sarabia7} M. Gonz\'alez--Sarabia, and C. Renter\'{\i}a, The
dual code of some Reed--Muller type codes, Appl. Algebra Engrg. 
Comm. Comput. 14 (2004) 329--333.

\bibitem{GRH} M. Gonz\'alez-Sarabia, C. Renter\'\i a and 
M. Hern\'andez de 
la Torre, 
Minimum distance and second generalized Hamming weight of two
particular linear codes, Congr. Numer. {\bf 161} (2003), 105--116. 

\bibitem{GRT} M. Gonz\'alez-Sarabia, C. Renter\'\i a and H.
Tapia-Recillas, Reed--Muller-type codes over the Segre variety,  
Finite Fields Appl. {\bf 8}  (2002),  no. 4, 511--518. 

\bibitem{mac2} D. R. Grayson and M. E. Stillman, {\em Macaulay\/}$2$, 
a software system for research in algebraic geometry. 
\newline Available at 
\url{https://macaulay2.com/}.

\bibitem{Guardo-Marino-Van-Tuyl} E. Guardo, L. Marino and A. Van
Tuyl, Separators of fat points in $\mathbb{P}^n$, J. Algebra {\bf
324}  (2010), no. 7, 1492--1512.

\bibitem{Pellikaan} P. Heijnen and R. Pellikaan, Generalized Hamming
weights of $q$--ary Reed--Muller codes, IEEE Trans. Inform. Theory {\bf
44} (1998), no. 1, 181--196.

\bibitem{helleseth} T. Helleseth, T. Kl{\o}ve and J. Mykkeltveit, 
The weight distribution of irreducible cyclic codes with block
lengths $n_1((q^l-1)/N)$, Discrete Math. {\bf 18} (1977), no. 2, 179--211.

\bibitem{Huffman-Pless} W. C. Huffman and V. Pless,
\textit{Fundamentals of error-correcting 
codes}, Cambridge University Press, Cambridge, 2003. 

\bibitem{toric-codes} D. Jaramillo, M. Vaz Pinto and R. H.
Villarreal, Evaluation codes and their basic parameters, Des. Codes
Cryptogr, {\bf 89} (2021), no. 2, 269--300.

\bibitem{JohVer} T. Johnsen and H. Verdure, 
Hamming weights and Betti numbers of Stanley--Reisner rings associated
to matroids, Appl. Algebra Engrg. Comm. Comput. {\bf 24} (2013), no.
1,  73--93. 

\bibitem{Johnsen} T. Johnsen and H. Verdure, Generalized Hamming
weights for almost affine codes, IEEE Trans. Inform. Theory {\bf 63}
(2017), no. 4, 1941--1953.

\bibitem{Klove-1992} T. Kl{\o}ve, Support weight distribution of linear 
codes, Discrete Math. {\bf 106/107} (1992), 311--316.

\bibitem{Kreuzer} M. Kreuzer, 
On the canonical module of a $0$-dimensional scheme, 
Canad. J. Math. {\bf 46} (1994), no. 2, 357--379.

\bibitem{cocoa-book} M. Kreuzer and L. Robbiano, {\it Computational
Commutative Algebra} 2,
Springer-Verlag, Berlin, 2005. 

\bibitem{cartesian-codes} H. H. L\'opez, C. Renter\'\i a and R. H.
Villarreal, Affine cartesian codes,
Des. Codes Cryptogr. {\bf 71} (2014), no. 1, 5--19.

\bibitem{affine-codes} H. H. L\'opez, E. Sarmiento, M. Vaz
Pinto and R. H. Villarreal, Parameterized affine codes, 
Studia Sci. Math. Hungar. {\bf 49} (2012), no. 3, 406--418.

\bibitem{dual}  H. H. L\'opez, I. Soprunov and R. H. Villarreal, 
The dual of an evaluation code, Des. Codes Cryptogr. {\bf 89}
(2021), no. 7, 1367--1403.


\bibitem{MacWilliams-Sloane} F. J. MacWilliams and N. J. A. Sloane, 
\textit{The Theory of Error-correcting Codes}, North-Holland, 1977. 

\bibitem{hilbert-min-dis} J. Mart\'\i nez-Bernal, Y. Pitones 
and R. H. Villarreal, Minimum
distance functions of graded ideals   
and Reed--Muller-type codes, 
J. Pure Appl. Algebra {\bf 221} (2017), 251--275. 

\bibitem{Mat}{H. Matsumura, {\it Commutative Algebra\/}, Benjamin-Cummings, Reading, MA, 1980.}

\bibitem{Mats}{H. Matsumura, {\it Commutative Ring Theory\/},
Cambridge
Studies in Advanced Mathematics {\bf 8},
Cambridge University Press, 1986.}

\bibitem{olaya} W. Olaya--Le\'on and C. Granados--Pinz\'on, The
second generalized Hamming weight of certain Castle codes, Des. Codes
Cryptogr. {\bf 76} (2015), no. 1, 81--87.

\bibitem{algcodes} C. Renter\'\i a, A. Simis and R. H. Villarreal,
Algebraic methods for parameterized codes 
and invariants of vanishing ideals over finite fields, 
Finite Fields Appl. {\bf 17} (2011), no. 1, 81--104.  

\bibitem{afinetv}{E. Reyes, R. H. Villarreal and L. Z\'arate, 
A note on affine toric varieties, Linear Algebra 
Appl. {\bf 318} (2000), 
173--179.}

\bibitem{Eduardo-Saenz-indicator}
E. S\'aenz-de-Cabez\'on and H. P. Wynn, 
Betti numbers and minimal free resolutions for multi-state system
reliability bounds, J. Symbolic Comput. {\bf 44} (2009), no. 9, 1311--1325.

\bibitem{sage} SageMath, the Sage Mathematics Software System
(Version 8.4), The Sage Developers, 2018, available from
\url{http://www.sagemath.org}.

\bibitem{schaathun-willems} H. G. Schaathun and W. Willems, A lower bound on the weight
hierarchies of product codes, Discrete Appl. Math. {\bf 128} (2003),
no. 1, 251--261.

\bibitem{sorensen} A. S{\o}rensen, Projective Reed--Muller codes, 
IEEE Trans. Inform. Theory {\bf 37} (1991), no. 6, 1567--1576.

\bibitem{Stanley-level} R. Stanley, Cohen--Macaulay complexes, 
in \textit{Higher Combinatorics} (M. Aigner, Ed.),
pp. 51--62, Reidel, Dordrecht/Boston, 1977.

\bibitem{Sta1}{R. Stanley, Hilbert functions of graded 
algebras, Adv.
Math. {\bf 28} (1978), 57--83.}

\bibitem{Stur1}{B. Sturmfels, {\em Gr\"obner Bases and Convex 
Polytopes\/}, University Lecture Series {\bf 8}, American Mathematical
Society, Rhode Island, 1996.}  

\bibitem{stefan} S. O. Toh\v{a}neanu, 
Interactions between commutative algebra and coding theory. A second
draft. In progress.

\bibitem{tohaneanu-vantuyl} S. Toh\v{a}neanu and A. Van Tuyl,
Bounding invariants of fat points using a coding theory construction,
J. Pure Appl. Algebra {\bf 217} (2013), no. 2, 269--279.

\bibitem{tsfasman} M. Tsfasman, S. Vladut and D. Nogin, {\it
Algebraic 
geometric codes{\rm:} basic notions}, Mathematical Surveys and
Monographs {\bf 139}, American Mathematical Society, 
Providence, RI, 2007. 

\bibitem{Vas1}{W. V. Vasconcelos, {\it Computational Methods in
Commutative Algebra and Algebraic Geometry\/}, 
Springer-Verlag, 1998.}

\bibitem{monalg-rev} R. H. Villarreal, {\it Monomial Algebras\/},
Second edition, 
Monographs and Research Notes in Mathematics, Chapman and Hall/CRC,
Boca Raton, FL, 2015.

\bibitem{wei} V. K. Wei, Generalized Hamming 
weights for linear codes, 
IEEE Trans. Inform. Theory {\bf 37}  (1991),  no. 5, 1412--1418. 

\bibitem{Yang} M. Yang, J. Lin, K. Feng and D. Lin, 
Generalized Hamming weights of irreducible cyclic codes, IEEE Trans.
Inform. Theory {\bf 61} (2015), no. 9, 4905--4913. 

\end{thebibliography}

\end{document}